\documentclass{amsart}

\usepackage{amsmath, amsthm, graphics,setspace,xypic}
\usepackage{latexsym}
\usepackage{mathtools}
\usepackage{eucal}
\usepackage{caption}

\usepackage{amssymb}
\usepackage{color}
\usepackage{amscd}
\usepackage{palatino}
\usepackage{latexsym}
\usepackage{epsfig}
\usepackage{graphicx}
\usepackage{amsfonts}
\usepackage{psfrag}
\usepackage[all]{xy}
\usepackage{youngtab,ytableau}

\usepackage{url}
\usepackage{cite}

\usepackage[margin=1in]{geometry}

\definecolor{ltgrey}{RGB}{180, 187, 198}

\input xy
\xyoption{all}
\UseComputerModernTips

\numberwithin{equation}{section}
\numberwithin{figure}{section}

\newtheorem{theorem}{Theorem}[section]
\newtheorem{lemma}[theorem]{Lemma}
\newtheorem{proposition}[theorem]{Proposition}
\newtheorem{corollary}[theorem]{Corollary}

\newtheorem{fact}[theorem]{Fact}
\newtheorem{conjecture}[theorem]{Conjecture}
\newtheorem{remark}[theorem]{Remark}
\newtheorem{example}[theorem]{Example}

\theoremstyle{definition}
\newtheorem{definition}[theorem]{Definition}

\newcommand{\C}{{\mathbb{C}}}

\newcommand{\Z}{{\mathbb{Z}}}
\newcommand{\Q}{{\mathbb{Q}}}

\renewcommand{\P}{{\mathbb{P}}}

\newcommand{\g}{\mathfrak{g}}

\DeclareMathOperator{\Stab}{Stab}
\DeclareMathOperator{\Lie}{Lie}

\DeclareMathOperator{\row}{row}
\DeclareMathOperator{\col}{col}

\newcommand{\hsm}{{\hspace{1mm}}}

\newcommand{\Symm}{{\mathfrak{S}}}
\renewcommand{\Stab}{{\mathrm{Stab}}}
\newcommand{\Flags}{{\mathcal{F}\ell ags}}

\newcommand{\Hess}{{\mathcal{H}ess}}
\newcommand{\src}{{\mathrm{src}}}
\newcommand{\tgt}{{\mathrm{tgt}}}
\newcommand{\sk}{{\mathrm{sk}}}
\newcommand{\SK}{{\mathrm{SK}}} 
\newcommand{\asc}{{\mathrm{asc}}}
\newcommand{\inv}{{\mathrm{inv}}}

\begin{document}

\title{The cohomology of abelian Hessenberg varieties and the Stanley-Stembridge conjecture}

\author{Megumi Harada}
\address{Department of Mathematics and
Statistics\\ McMaster University\\ 1280 Main Street West\\ Hamilton, Ontario L8S4K1\\ Canada}
\email{Megumi.Harada@math.mcmaster.ca}
\urladdr{\url{http://www.math.mcmaster.ca/Megumi.Harada/}}
\thanks{The first author is partially supported by an NSERC Discovery Grant and a Canada Research Chair (Tier 2) award.}

\author{Martha Precup}
\address{Department of Mathematics\\ Northwestern University \\ 2033 Sheridan Road \\ Evanston, Illinois 60208 \\ U.S.A. }
\email{mprecup@math.northwestern.edu}
\urladdr{\url{http://www.math.northwestern.edu/~mprecup/}}


\date{\today}


\begin{abstract}
We define a subclass of Hessenberg varieties called abelian Hessenberg varieties, inspired by the theory of abelian ideals in a Lie algebra developed by Kostant and Peterson. We give an inductive formula for the $\Symm_n$-representation on the cohomology of an abelian regular semisimple Hessenberg variety with respect to the action defined by Tymoczko. Our result implies that a graded version of the Stanley-Stembridge conjecture holds in the abelian case, and generalizes results obtained by Shareshian-Wachs and Teff. Our proof uses previous work of Stanley, Gasharov, Shareshian-Wachs, and Brosnan-Chow, as well as results of the second author on the geometry and combinatorics of Hessenberg varieties. As part of our arguments, we obtain inductive formulas for the Poincar\'e polynomials of regular abelian Hessenberg varieties. 
\end{abstract}

\maketitle

\setcounter{tocdepth}{1}
\tableofcontents

\section{Introduction}\label{sec:intro}

Hessenberg varieties in type A are subvarieties of the full flag variety $\Flags(\C^n)$ of nested sequences of linear subspaces in $\C^n$.  These varieties are parameterized by a choice of linear operator $\mathsf{X}\in \mathfrak{gl}(n,\C)$ and Hessenberg funciton $h: \{1,2,...,n\} \to \{1,2,...,n\}$.  We denote the corresponding Hessenberg variety by $\Hess(\mathsf{X}, h)$. The geometry and (equivariant) topology of Hessenberg varieties has been studied
extensively since the late 1980s \cite{DeMari1987, DeMProSha92}.  This
subject lies at the intersection of, and makes connections between,
many research areas such as geometric representation theory, 
combinatorics,
and
algebraic geometry and topology.

In this manuscript, we are concerned with the connection between Hessenberg varieties and the famous Stanley-Stembridge conjecture in combinatorics, which states that
the chromatic symmetric function of the incomparability graph of a so-called $(3+1)$-free poset is $e$-positive, i.e., it is a non-negative linear combination of elementary symmetric polynomials \cite[Conjecture 5.5]{StanleyStembridge1993} (see also \cite{Stanley1995}). Guay-Paquet has proved that this conjecture can be reduced to the statement that the chromatic symmetric function of the incomparability graph of a unit interval order is $e$-positive \cite{Guay-Paquet2013}.
Shareshian and Wachs then linked the Stanley-Stembridge conjecture to Hessenberg varieties via the ``dot action'' $\Symm_n$-representation on the cohomology ring of a regular semisimple Hessenberg variety defined by Tymoczko \cite{Tym08}.  The Hessenberg variety $\Hess(\mathsf{S}, h)$ is called a regular semisimple Hessenberg variety if $\mathsf{S}$ is a regular semisimple element of $\mathfrak{gl}(n,\C)$.  Specifically, Shareshian and Wachs established a bijection between Hessenberg functions and unit interval orders \cite[Proposition 4.1]{ShareshianWachs2016} and then formulated a conjecture relating the chromatic quasisymmetric function of the incomparability graph of a unit interval order to the dot action representation on the cohomology of an associated regular semisimple Hessenberg variety.  In fact, since cohomology rings are graded by degree, the Shareshian-Wachs conjecture naturally suggests a stronger, ``graded'' conjecture based on the original Stanley-Stembridge conjecture (see \cite[Conjecture 10.4]{ShareshianWachs2016}). We formalize this in Conjecture~\ref{conj:Stanley-Stembridge} below.

The Shareshian-Wachs conjecture was proved in 2015 by Brosnan and Chow \cite{BrosnanChow2015} (also independently by Guay-Paquet \cite{Guay-Paquet2016}) by showing a remarkable relationship between the Betti numbers of different Hessenberg varieties.  (Direct computations of cohomology rings of certain Hessenberg varieties also yield partial proofs of the Shareshian-Wachs conjecture; see \cite{AHHM, AHM2017}.) 
It then follows that, in order to prove the graded Stanley-Stembridge conjecture, it suffices to 
prove that the cohomology 
$H^{2i}(\Hess(\mathsf{S},h))$ for each $i$ is a non-negative combination of the tabloid representations $M^\lambda$ \cite[Part II, Section 7.2]{Ful97} of $\Symm_n$ for $\lambda$ a partition of $n$. In other words, given the decomposition 
\begin{equation}\label{eq: decomp intro}
H^{2i}(\Hess(\mathsf{S},h)) = \sum_{\lambda \vdash n}  c_{\lambda, i} M^\lambda
\end{equation}
in the representation ring $\mathcal{R}ep(\Symm_n)$ of $\Symm_n$, the coefficients $c_{\lambda,i}$ are non-negative.  

The above discussion explains the motivation for this manuscript, and we now describe our main results. 
Let $h: \{ 1, 2, ...,n\} \to \{ 1,2,...,n \}$ be a Hessenberg function. Our approach to the graded Stanley-Stembridge conjecture is by induction.  Roughly, the idea is as follows. From any Hessenberg function $h$ we can construct a corresponding incomparability graph $\Gamma_h$ (made precise in Section~\ref{sec: graphs and orientations}). Previous results of Stanley show that the acyclic orientations of $\Gamma_h$, and their corresponding sets of sinks, encode information about the coefficients $c_{\lambda,i}$. We develop this idea further by decomposing the set of acyclic orientations according to their \textbf{sink sets}, and make a key observation (Proposition~\ref{proposition: max sink set induction}) that, if the size of a sink set is maximal, then the set of acyclic orientations with that fixed sink set corresponds precisely to the set of \emph{all} acyclic orientations on a smaller incomparability graph. This observation sets the stage for an inductive argument.

Any Hessenberg function corresponds uniquely to a certain subset $I_h$ of the negative roots of $\mathfrak{gl}(n,\C)$.  In this manuscript, in the special case when $I_h$ is  \textbf{abelian} (cf. Definition~\ref{def: abelian} below), we are able to fully implement an argument yielding an inductive formula for the coefficients of the tabloid representations. 
A rough statement of our main result is as follows (definitions are in Section~\ref{sec: Hessenberg basics}, Section~\ref{sec: Stanley-Stembridge conjecture and dot action} and Section~\ref{sec: sink sets and induction}); the precise statement is Theorem~\ref{theorem:induction}. The idea is that the coefficients $c_{\lambda,i}$ above, associated to a Hessenberg variety in $\Flags(\C^n)$ for $n\geq 3$, can be computed using the coefficients associated to certain Hessenberg varieties in the flag variety $\Flags(\C^{n-2})$. 

\begin{theorem}\label{theorem:main} 
Let $n\geq 3$ be a positive integer and $h: \{1,2,...,n\} \to \{1,2,...,n\}$ be a Hessenberg function such that the ideal $I_h$ is abelian. Let $\mathsf{S}$ denote a regular semisimple element in the Lie algebra of $\mathfrak{gl}(n,\C)$. Regard the cohomology $H^{2i}(\Hess(\mathsf{S},h))$ as a $\Symm_n$-representation using Tymoczko's dot action. Let $i \geq 0$ be a non-negative integer. 
Then, in 
the representation ring $\mathcal{R}ep(\Symm_n)$ we have the equality 
\begin{equation}\label{eq:main inductive step}
H^{2i}(\Hess(\mathsf{S}, h)) =  c_{(n),i} M^{(n)} +  \sum_{T \in SK_2(\Gamma_h)} \left( \sum_{\substack{\mu \,\vdash (n-2)\\ \mu=(\mu_1,\mu_2)}} c_{\mu, i-\deg(T)}^T M^{(\mu_1+1,\mu_2+1)} \right)
\end{equation}
where the set $SK_2(\Gamma_h)$ is a certain collection of subsets of the vertices of $\Gamma_h$ and the coefficients $c^T_{\mu, i-\deg(T)}$ are the coefficients as in~\eqref{eq: decomp intro} associated to a Hessenberg function $h_T: \{1,2,\ldots,n-2\} \to \{1,2,\ldots,n-2\}$ for a Hessenberg variety in $\Flags(\C^{n-2})$. 
\end{theorem}

The technical details of the induction argument leading to Theorem~\ref{theorem:main} 
requires the use -- among other things -- of Brosnan and Chow's proof of the Shareshian-Wachs conjecture, as well as the second author's combinatorial characterization of the Betti numbers of regular Hessenberg varieties. In fact, the technical core of the paper consists of two inductive formulas for the Poincar\'e polynomials of regular abelian Hessenberg varieties.  These formulas are stated in Proposition~\ref{prop: reg step} and Proposition~\ref{prop:induction step} and are of independent interest. 

It is quite straightforward to prove the graded Stanley-Stembridge conjecture for the abelian case based on our inductive formula in Theorem~\ref{theorem:main}, and we record this argument in Corollary~\ref{corollary: graded SS for abelian}. Our result generalizes previous results.  Indeed, in the case when $h$ satisfies $h(3)=\cdots=h(n)=n$, Shareshian and Wachs obtained results on the corresponding chromatic quasisymmetric function which, given Brosnan and Chow's proof of the Shareshian-Wachs conjecture, implies Corollary~\ref{corollary: graded SS for abelian} for that case. Separately, Teff \cite[Theorem 4.20]{Teff2013a} proved the case when $h$ corresponds to a maximal standard parabolic Lie subalgebra $\mathfrak{p}$ of $\mathfrak{gl}(n,\C)$. Both instances are special cases of our result, as we explain in Section~\ref{sec: abelian Hessenberg varieties}.  We also note that Gebhard and Sagan have proved the Stanley-Stembridge conjecture for a collection of graphs called $K_{\alpha}$-chains \cite[Corollary 7.7]{GebhardSagan2001}.  Their result does not subsume, nor is subsumed, by the case considered in this manuscript, but it is of independent interest.

As part of our arguments, we also define the \textbf{height} of an ideal of negative roots using the lower central series of an ideal in a Lie algebra. An ideal is abelian precisely when the height is either $1$ or $0$, so we can interpret Theorem~\ref{theorem:main} as a ``base case'' for an argument for the full graded Stanley-Stembridge conjecture using induction on the height of the ideal $I_h$. We intend to explore this further in future work.

As already mentioned, the graded Stanley-Stembridge conjecture implies the classical Stanley-Stembridge conjecture simply by summing over all $i$, or, in the language of chromatic quasisymmetric polynomials, by ``setting $t$ equal to $1$''.  We record this fact in Proposition~\ref{prop: classical from graded}.  We note here that the ``abelian case'' considered in Theorem~\ref{theorem:induction} (and Corollary~\ref{corollary: graded SS for abelian}) corresponds, in combinatorial language, to the case in which the vertices of the graph $\Gamma_h$ can be partitioned into two disjoint cliques. The fact that the coefficients $c_{\lambda} = \sum_{i\geq 0} c_{\lambda, i}$ are non-negative in this case was originally stated by Stanley in \cite[Corollary 3.6]{Stanley1995} as a corollary to \cite[Theorem 3.4]{Stanley1995}; moreover, this fact is also equivalent to \cite[Remark 4.4]{StanleyStembridge1993}. However, \cite[Theorem 3.4]{Stanley1995} is incorrect as stated \cite{Stanley-personal}, and the equivalence of \cite[Remark 4.4]{StanleyStembridge1993} and \cite[Corollary 3.6]{Stanley1995} is not explicit in \cite{StanleyStembridge1993, Stanley1995}. Thus, our Corollary~\ref{corollary: graded SS for abelian} (together with Proposition~\ref{prop: classical from graded}) records a new and explicit proof of this fact.

We now give a brief overview of the contents of the paper. Section~\ref{sec:background} is devoted to background material. Specifically, Section~\ref{sec: Hessenberg basics} is a crash course on Hessenberg varieties. Section~\ref{sec: Stanley-Stembridge conjecture and dot action} establishes the terminology for discussing the $\Symm_n$-representations $H^{2i}(\Hess(\mathsf{S},h))$, and gives a more detailed account of the relation between the Stanley-Stembridge conjecture and our results. Section~\ref{sec: graphs and orientations} recalls the language of incomparability graphs in the setting of Hessenberg functions and states a result of Stanley connecting acyclic orientations on this graph to the $\Symm_n$-representations above. Section~\ref{sec: Ph tableaux} recounts Gasharov's definition of $P_h$-tableau and a result relating these $P_h$-tableau to the same $\Symm_n$-representations above. We then begin our work in earnest in Section~\ref{sec: abelian Hessenberg varieties} where we define abelian Hessenberg varieties and briefly discuss the relation between this notion and the cases of Hessenberg varieties previously studied in the literature. In Section~\ref{sec: sink sets and induction} we focus attention on the sink sets of an acyclic orientation of an incomparability graph, and introduce the notion of sink-set size. In Section~\ref{sec: sink sets, ideals and representations} we link the subjects of Sections~\ref{sec: abelian Hessenberg varieties} and \ref{sec: sink sets and induction} using a new invariant of an ideal called the height.  Sections~\ref{section: inductive formula} and~\ref{section: the proofs} form the technical core of the paper, where we state and prove our main results. Finally, Section~\ref{section: conjecture} states a conjecture which, if true, would represent a first step towards generalizing the techniques in this paper to prove the full Stanley-Stembridge conjecture for all possible heights. 

\bigskip

\noindent \textbf{Acknowledgements.} The research for this manuscript was partially conducted at the Osaka City University Advanced Mathematical Institute, and we thank the Institute for its hospitality. 
 We are also grateful to Timothy Chow, Seung Jin Lee, Mikiya Masuda, John Shareshian, Richard Stanley, and Stephanie Van Willigenburg for valuable comments.

\section{The setup and background}\label{sec:background}

Let $n$ be a positive integer. We denote by $[n]$ the set of positive integers $\{1,2,\ldots,n\}$. 
We work in type A throughout, so $GL(n,\C)$ is the group of invertible $n\times n$ complex matrices and $\mathfrak{gl}(n,\C)$ is the Lie algebra of $GL(n,\C)$ consisting of all $n\times n$ complex matrices.

\subsection{Hessenberg varieties}\label{sec: Hessenberg basics}

Hessenberg varieties in Lie type A are subvarieties of the (full) flag variety
$\Flags(\C^n)$, which is the collection of sequences of nested linear subspaces of $\C^n$:
\[
\Flags(\C^n) := 
\{ V_{\bullet} = (\{0\} \subset  V_1 \subset  V_2 \subset  \cdots V_{n-1} \subset 
V_n = \C^n) \hsm \mid \hsm \dim_{\C}(V_i) = i \ \textrm{for all} \ i=1,\ldots,n\}. 
\]
A Hessenberg variety in $\Flags(\C^n)$ is specified by two pieces of data: a Hessenberg function and a choice of an element in $\mathfrak{gl} (n,\C)$. We have the following. 

\begin{definition}\label{definition:Hessenberg function} 
A \textbf{Hessenberg function} is a function $h: [n] \to [n]$ such that $h(i) \geq i$ for all $i \in [n]$ and $h(i+1) \geq h(i)$ for all $i \in [n-1]$. We frequently write a Hessenberg function by listing its values in sequence,
i.e., $h = (h(1), h(2), \ldots, h(n))$.
\end{definition} 
We now introduce some terminology associated to a given Hessenberg function. 

\begin{definition}\label{definition:Hessenberg subspace} 
Let $h: [n] \to [n]$ be a Hessenberg function. The associated \textbf{Hessenberg space} is the linear subspace $H$ of $\mathfrak{gl} (n,\C)$ specified as follows:
\begin{align}\label{eq:Hessenberg subspace} 
H := \{ A = (a_{ij})_{i,j\in[n]} \in \mathfrak{gl} (n,\C) \mid a_{ij} = 0 \textup{ if } i > h(j) \} 
= \mathrm{span}_{\C} \{E_{ij} \hsm \vert \hsm i, j \in [n] \textup{ and } i \leq h(j) \}
\end{align} 
where $E_{i,j}$ is the usual elementary matrix with a $1$ in the $(i,j)$-th entry and
$0$'s elsewhere. 
\end{definition}

It is important to note that $H$ is frequently \emph{not} a Lie subalgebra of $\mathfrak{gl} (n,\C)$. However, it \emph{is} stable under the conjugation action of the usual maximal torus $T$ (of invertible diagonal matrices) in $GL(n,\C)$, and the $E_{ij}$ appearing in~\eqref{eq:Hessenberg subspace} are exactly the $T$-eigenvectors. It is also straightforward to see that 
  \begin{equation}\label{eq:Hess stable under b} 
[\mathfrak{b}, H] \subseteq H
\end{equation}
where $[ \cdot, \cdot ]$ denotes the usual Lie bracket in $\mathfrak{gl} (n,\C)$ and 
$\mathfrak{b} = \Lie(B)$ is the Lie algebra of the Borel subgroup $B$ of upper-triangular matrices in $GL(n,\C)$. 

Let $\mathfrak{h} \subseteq \mathfrak{gl} (n,\C)$ denote the Cartan subalgebra of diagonal matrices, and let $t_i$ denote the coordinate on $\mathfrak{h}$ reading off the $(i,i)$-th matrix entry along the diagonal. Denote the root system of $\mathfrak{gl}(n,\C)$ by $\Phi$. Then the positive roots $\Phi^+$ of $\mathfrak{gl} (n,\C)$ are $\Phi^+ = \{ t_i - t_j \hsm \vert \hsm 1 \leq i < j \leq n\}$ where $\gamma = t_i - t_j \in \Phi^+$ corresponds to the root space spanned by $E_{ij}$, denoted $\mathfrak{g}_{\gamma}$. Similarly, the negative roots $\Phi^-$ of $\mathfrak{gl} (n,\C)$ are $\Phi^- = \{ t_i - t_j \hsm \vert \hsm 1 \leq j < i \leq n\}$. We denote the simple positive roots in $\Phi^+$ by $\Delta = \{\alpha_i := t_i - t_{i+1} \; \vert \; 1 \leq i \leq n-1\}$.  

Note that the pairs $(i,j)$ with $i > j$ and $i \leq h(j)$ correspond
precisely to those negative roots $\gamma \in \Phi^-$ whose associated root spaces $\g_{\gamma}$ 
are contained in $H$. Motivated by this, we fix the following notation: 
\[
\Phi_h^- := \{  t_i - t_j \in \Phi^- \hsm \vert \hsm E_{ij} \in H \} =
\{ t_i - t_j  \hsm \vert \hsm i > j \textup{ and } i \leq h(j) \}
\]
and 
\[
\Phi_h := \Phi_h^- \sqcup \Phi^+ = \{  t_i - t_j \in \Phi \hsm \vert \hsm i \leq h(j) \}.
\]
It is clear that $h$ is uniquely determined by either $\Phi_h^-$ or $\Phi_h$. 

Recall that an \textbf{ideal} (also called an \textbf{upper-order ideal}) $I$ of $\Phi^-$ is defined to be a collection of (negative) roots such that if $\alpha\in I$, $\beta\in \Phi^-$, and $\alpha+\beta\in \Phi^-$, then $\alpha+\beta\in I$.
The relation~\eqref{eq:Hess stable under b} immediately implies that 
\begin{equation*}\label{eq:definition Ih} 
I_h:= \Phi^- \setminus \Phi_h^-
\end{equation*} 
is an ideal in $\Phi^-$.  We call it the \textbf{ideal corresponding to $h$}. In fact, the association taking a Hessenberg function to its corresponding ideal $I_h$ defines a bijection from the set of Hessenberg functions to ideals in $\Phi^-$, as noted by Sommers and Tymoczko in \cite[Section 10]{SomTym06}.

It is conceptually useful to express the sets $\Phi_h, \Phi_h^-,$ and $I_h$ pictorially.   We illustrate this by an example. 

\begin{example} \label{ex: pictures}
Let $n=6$. Figure~\ref{picture for h} contains the pictures corresponding to the Hessenberg function $h=(3,4,5,6,6,6)$. The leftmost square grid contains a star in the $(i,j)$-th box exactly if the $(i,j)$-th matrix entry is allowed to be non-zero for $A \in H$, or equivalently, either $i=j$, or, the corresponding root $t_i - t_j$ of $\mathfrak{gl} (n,\C)$ is contained in $\Phi_h$. The center square grid contains a star in the $(i,j)$-th box precisely if the corresponding root is contained in $\Phi_h^-$. Finally, the rightmost grid contains a star in the $(i,j)$-th box if and only if the corresponding root is contained in $I_h$, i.e., it is the complement of $\Phi_h$. This illustrates why some authors refer to $I_h$ as (the roots corresponding to) the ``opposite Hessenberg space''. 

\begin{figure}[h]
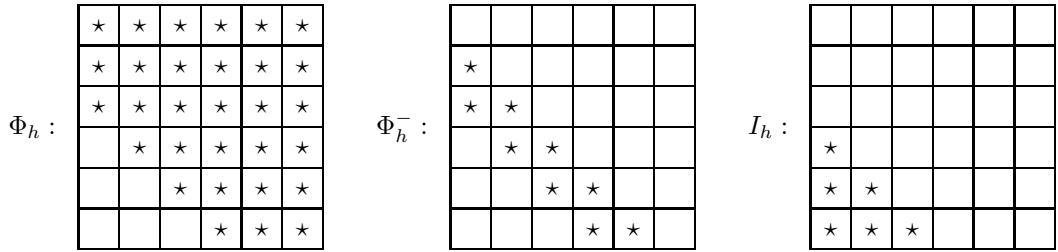


\[\ytableausetup{centertableaux}
\Phi_h: \;\; \begin{ytableau} \star & \star & \star & \star & \star & \star \\ \star & \star & \star & \star & \star & \star \\  \star & \star & \star & \star & \star & \star \\ \empty & \star & \star & \star & \star & \star \\  \empty & \empty & \star & \star & \star & \star\\ \empty & \empty & \empty & \star & \star & \star    \end{ytableau} 
\quad\quad
\Phi_h^-: \;\; \begin{ytableau} \empty & \empty & \empty & \empty & \empty & \empty \\ \star & \empty & \empty & \empty & \empty & \empty \\  \star & \star & \empty & \empty & \empty & \empty \\ \empty & \star & \star & \empty & \empty & \empty \\  \empty & \empty & \star & \star & \empty & \empty\\ \empty & \empty & \empty & \star & \star & \empty    \end{ytableau} 
\quad\quad
I_h: \;\; \begin{ytableau} \empty & \empty & \empty & \empty & \empty & \empty \\ \empty & \empty & \empty & \empty & \empty & \empty \\  \empty & \empty & \empty & \empty & \empty & \empty \\ \star & \empty & \empty & \empty & \empty & \empty \\  \star & \star & \empty & \empty & \empty & \empty\\ \star & \star & \star & \empty & \empty & \empty    \end{ytableau} 
\]
\caption{The pictures of $\Phi_h$, $\Phi_h^-$, and $I_h$ for $h=(3,4,5,6,6,6)$.}
\label{picture for h}
\end{figure}
\end{example}

Let $h:[n]\to[n]$ be a Hessenberg function and $\mathsf{X}$ be an $n\times n$ matrix in $\mathfrak{gl} (n,\C)$, which we also consider as a linear operator $\C^n \to \C^n$. Then the \textbf{Hessenberg variety} $\Hess(\mathsf{X},h)$ associated to $h$ and $\mathsf{X}$ is defined to be  
\begin{align}\label{eq:def-general Hessenberg}
\Hess(\mathsf{X},h) := \{ V_{\bullet}  \in \Flags(\C^n) \;
\vert \;   \mathsf{X} V_i \subset 
V_{h(i)} \text{ for all } i\in[n]\} \subset  \Flags(\C^n).
\end{align}

In this paper we focus on certain special cases of Hessenberg varieties. Let $\lambda = (\lambda_1, \lambda_2, \ldots, \lambda_n)$ be a composition of $n$ in the sense that $\lambda_1+\lambda_2+\cdots + \lambda_n=n$ and $\lambda_i \geq 0$ for all $i$. A linear operator is \textbf{regular of Jordan type $\lambda$} if its standard Jordan canonical form has block sizes given by $\lambda_1$,  $\lambda_2$, etc., and all of its eigenvalues are distinct.  Note that if $g \in \text{GL}(n,\C)$, then $\Hess(\mathsf{X},h)$ and $\Hess(g\mathsf{X}g^{-1},h)$ can be identified via the action of $\text{GL}(n,\C)$ on $\Flags(\C^n)$ \cite{Tym06}.  For concreteness in what follows, for a given $\lambda$ as above we set the notation  
\[
\mathsf{X}_\lambda \textup{ is a (fixed) matrix in standard Jordan canonical form, which is regular of Jordan type $\lambda$} 
\]
and we refer to the corresponding Hessenberg variety $\Hess(\mathsf{X}_\lambda, h)$ as a \textbf{regular Hessenberg variety}. 

Two special cases are of particular interest. Namely, if $\lambda=(n,0,\ldots,0)=(n)$, then we may take the corresponding regular operator to be the regular nilpotent operator which we denote by $\mathsf{N}$, i.e., $\mathsf{N}$ is the matrix whose Jordan form consists of exactly one Jordan block with corresponding eigenvalue equal to $0$. The regular Hessenberg variety $\Hess(\mathsf{N},h)$ is called a \textbf{regular nilpotent Hessenberg variety}. Similarly let $\mathsf{S}$ denote a regular semisimple matrix in $\mathfrak{gl} (n,\C)$, i.e., a matrix which is diagonalizable with distinct eigenvalues. This corresponds to the other extreme case, namely, $\lambda = (1,1,1,\ldots,1)$. We call $\Hess(\mathsf{S},h)$ a
\textbf{regular semisimple Hessenberg variety}.

\subsection{The Stanley-Stembridge conjecture in terms of Tymoczko's dot action representation} \label{sec: Stanley-Stembridge conjecture and dot action}

As already discussed in the Introduction, the main motivation of this manuscript is to study a graded version of the Stanley-Stembridge conjecture (Conjecture~\ref{conj:Stanley-Stembridge} below), stated in terms of the $\Symm_n$-representation on the cohomology rings of regular semisimple Hessenberg varieties defined by Tymoczko \cite{Tym08}. 
Tymoczko's \textbf{dot action} preserves the grading on these cohomology rings (which is concentrated in even degrees). 
The structure of this section is as follows. We first review basic facts and establish notation for partitions and $\Symm_n$-representations. 
We then give in Conjecture~\ref{conj:Stanley-Stembridge} 
a precise statement of the graded Stanley-Stembridge conjecture. We also state the original Stanley-Stembridge conjecture, and briefly recount how a solution to Conjecture~\ref{conj:Stanley-Stembridge} implies the original Stanley-Stembridge conjecture (cf. discussion in \cite{BrosnanChow2015, ShareshianWachs2016}).  The rest of the section is a brief review of standard representation theory facts and a statement of a fundamental result of Brosnan and Chow.

A \textbf{partition} of $n$ is a sequence $\lambda=(\lambda_1,\lambda_2,\ldots,\lambda_n) \in \Z^n$ satisfying $\lambda_1+\lambda_2+\cdots+\lambda_n=n$ and $\lambda_1 \geq \lambda_2 \geq \cdots \geq \lambda_n  \geq 0$. If $\lambda$ is a partition of $n$ we write $\lambda \vdash n$. 
We say a partition $\lambda \vdash n$ \textbf{has $k$ parts} and write $\mathrm{parts}(\lambda)=k$ if $\lambda_k \neq 0$ and $\lambda_{k+1}=\cdots = \lambda_n=0$. (Alternatively, $\lambda$ has $k$ parts if and only if the Young diagram corresponding to $\lambda$ has precisely $k$ rows.) For simplicity if $\mathrm{parts}(\lambda)=k$ then we write $\lambda=(\lambda_1,\ldots,\lambda_k)$ instead of $\lambda=(\lambda_1,\cdots,\lambda_k, 0, 0,\cdots,0)$. 
Moreover, for $\nu \vdash n$ a partition of $n$, we let $\Symm_\nu \subseteq \Symm_n$ denote the \textbf{Young subgroup} of $\Symm_n$ corresponding to $\nu$. Concretely, if $\nu = (\nu_1, \nu_2, ..., \nu_k)$ has $k$ parts then $\Symm_\nu$ is the subgroup 
\[
\Symm_{1,\ldots,\nu_1} \times \Symm_{\nu_1+1, \ldots, \nu_1+\nu_2} \times \cdots \times 
\Symm_{(\sum_{\ell=1}^{k-1} \nu_{\ell})+1, \ldots, n} \subseteq \Symm_n
\]
where $\Symm_{i, i+1,..., j}$ denotes the permutations of the set $\{i, i+1,..., j\}$ for each $1\leq i <j \leq n$.

Following Fulton \cite[Part II, Section 7.2]{Ful97}, we denote by $M^\lambda$ the complex vector space with basis the set of tabloids $\{T\}$ of shape $\lambda$, where $\lambda$ is a partition of $n$. Since $\Symm_n$ acts on the set of tabloids, $M^\lambda$ is a $\Symm_n$-representation. Our main theorem concerns the decomposition of $H^{2i}(\Hess(\mathsf{S},h))$ into $M^\lambda$'s, but to study this, we first decompose $H^{2i}(\Hess(\mathsf{S},h))$ into irreducible representations. Denote by $\mathcal{S}^{\lambda}$ the Specht module corresponding to $\lambda$. It is well-known that each $\mathcal{S}^{\lambda}$ is irreducible and that any irreducible $\Symm_n$-representation is isomorphic to $\mathcal{S}^{\lambda}$ for some $\lambda$ \cite[Section 7.2, Proposition 1]{Ful97}. Thus we conclude that there exist non-negative integers $d_\lambda$ and $d_{\lambda,i}$ such that 
\begin{equation}\label{eq:decomp into Specht} 
H^*(\Hess(\mathsf{S},h)) \cong \bigoplus_{\lambda\vdash n} d_\lambda \mathcal{S}^{\lambda}
\quad \textup{ and } \quad H^{2i}(\Hess(\mathsf{S},h)) \cong \bigoplus_{\lambda \vdash n} d_{\lambda,i} \mathcal{S}^{\lambda} 
\end{equation}
as $\Symm_n$-representations. Note $d_{\lambda} = \sum_{i \geq 0} d_{\lambda,i}$ for any $\lambda$.

There is a well-known formula, called Young's rule, for the decomposition of $M^\lambda$ into Specht modules. We need some terminology. Recall that the \textbf{dominance order} on partitions 
\cite[page 26]{Ful97} is defined as 
\[
\lambda \trianglelefteq \nu \textup{ if and only if } \lambda_1 + \cdots + \lambda_i \leq \nu_1 + \nu_2 + \cdots + \nu_i \textup{ for all } i.
\]
The following lemmas are straightforward. Let $\leq_{lex}$ denote the usual lexicographic order.

\begin{lemma}\label{lemma:dominance and lex} 
Let $\lambda, \nu$ be partitions of $n$. If $\lambda \trianglelefteq \nu$ then $\lambda \leq_{lex} \nu$. 
\end{lemma} 

\begin{lemma}\label{lemma:dominance order and length} 
Let $\lambda, \nu$ be partitions of $n$. If $\lambda \trianglelefteq \nu$ then $\mathrm{parts}(\nu) \leq \mathrm{parts}(\lambda)$. 
\end{lemma} 

We define the following total order on partitions: 
\begin{equation}\label{eq: def total order on partitions} 
\lambda \preccurlyeq \nu \Leftrightarrow_{\textup{def}} \left( \mathrm{parts}(\nu) < \mathrm{parts}(\lambda) \right) \textup{ or } ( \mathrm{parts}(\nu) = \mathrm{parts}(\lambda) \textup{ and } \lambda \leq_{lex} \nu ). 
\end{equation}
In words, this total order first compares partitions based on the number of parts, and then breaks ties using the usual lex order. The following is immediate from the above two lemmas. 

\begin{lemma}\label{lemma: the total order we want} 
The order $\preccurlyeq$ is a refinement of the dominance order, i.e., for $\lambda, \nu \vdash n$, if $\lambda \trianglelefteq \nu$ then $\lambda \preccurlyeq \nu$. 
\end{lemma} 

For a pair of partitions $\lambda$ and $\nu$ of $n$, the \textbf{Kostka number} $K_{\nu \lambda }$ is defined \cite[Part I, Section 2.2]{Ful97} to be the number of semistandard Young tableau of shape $\nu$ and weight/content $\lambda$.

\begin{fact}\label{fact: Kostka matrix entries}  \cite[Part I, Section 2.2, Exercise 2]{Ful97} 
For $\lambda, \nu$ partitions of $n$, we have $K_{\nu\lambda} \neq 0$ if and only if $\lambda \trianglelefteq \nu$. Moreover, $K_{\lambda \lambda} = 1$ for all partitions $\lambda \vdash n$. 
\end{fact} 

\noindent Let $K := (K_{\nu \lambda})$ denote the \textbf{Kostka matrix} with entries the Kostka numbers with partitions listed in decreasing order with respect to $\preccurlyeq$. 
By Lemma~\ref{lemma: the total order we want} and Fact~\ref{fact: Kostka matrix entries} above, $K$ is upper-triangular and has $1$'s along the diagonal, and in particular, is invertible over $\Z$. 
Young's rule 
 \cite[Section 7.3, Corollary 1]{Ful97} states that 
\begin{equation}\label{eq:Mlambda in terms of Slambda} 
M^{\lambda} \cong \mathcal{S}^{\lambda} \oplus \left( \bigoplus_{\nu \triangleright \lambda} (\mathcal{S}^\nu)^{\oplus K_{\nu\lambda}} \right)
\end{equation}
as $\Symm_n$-representations. 
Since the Kostka matrix is invertible over $\Z$, ~\eqref{eq:Mlambda in terms of Slambda} implies the $\{M^\lambda\}$ form 
a $\Z$-basis for the representation ring $\mathcal{R}ep(\Symm_n)$ of $\Symm_n$. Therefore there exist unique integers $c_\lambda$ and $c_{\lambda,i}$ such that 
\begin{equation}\label{eq: decomp into Mlambda}
H^*(\Hess(\mathsf{S},h)) = \sum_{\lambda\vdash n} c_\lambda M^\lambda \quad \textup{ and } 
\quad H^{2i}(\Hess(\mathsf{S},h)) = \sum_{\lambda \vdash n} c_{\lambda,i} M^{\lambda} 
\end{equation}
as elements in $\mathcal{R}ep(\Symm_n)$. Note that, a priori, the coefficients $c_\lambda$ and $c_{\lambda,i}$ may be negative. We also have $c_{\lambda} = \sum_{i \geq 0} c_{\lambda,i}$ for all $\lambda \vdash n$. We can now formulate the \textbf{graded Stanley-Stembridge conjecture} which motivates this manuscript; the terminology will be justified below.

\begin{conjecture}\label{conj:Stanley-Stembridge} 
Let $n$ be a positive integer and $h: [n] \to [n]$ be a Hessenberg function. Then 
the integers $c_{\lambda,i}$ appearing in~\eqref{eq: decomp into Mlambda} are non-negative. 
\end{conjecture}

The main theorem of this manuscript (Theorem~\ref{theorem:induction}) allows us to deduce the above conjecture in the special case that $h$ is abelian (cf. Definition~\ref{def: abelian}). Before proceeding we take a moment to explain how a proof of Conjecture~\ref{conj:Stanley-Stembridge} implies the classical Stanley-Stembridge conjecture. Since this story has been recorded elsewhere (e.g. \cite{BrosnanChow2015, ShareshianWachs2016}) we will be brief. We begin with a statement of the original Stanley-Stembridge conjecture. An incomparability graph of a unit interval order is a finite graph $\Gamma=(V,E)$ whose vertices are (distinct) closed unit intervals on the real line, with a single edge joining unit intervals with non-empty intersection. For any finite graph $\Gamma=(V,E)$, a \textbf{coloring} of $\Gamma$ is a function $\kappa: V \to \{1,2,3,\ldots\}$ assigning the ``color'' $\kappa(v)$ to each $v \in V$ and $\kappa$ is proper if for every edge $e=\{u,v\}$ in $E$, $\kappa(u) \neq \kappa(v)$. The \textbf{chromatic symmetric function} $X_{\Gamma}(x_1,x_2,\ldots)$ is defined as 
\[
X_\Gamma(\underline{x}) = X_{\Gamma}(x_1,x_2,\cdots) = \sum_{\textup{proper } \kappa: V \to \{1,2,\ldots\}} x^\kappa
\]
where $x^\kappa := \prod_{v \in V} x_{\kappa(v)}$. It is not hard to see that $X_{\Gamma}$ is symmetric in the variables $\{x_i\}$. A symmetric function is said to be \textbf{e-positive} if it can be expressed as a non-negative linear combination of the elementary symmetric functions $e_{\lambda}$. The following is the Stanley-Stembridge conjecture, which is related to many other deep conjectures, e.g. about immanants. 

\begin{conjecture} (Stanley-Stembridge conjecture) 
Let $\Gamma=(V,E)$ be the incomparability graph of a unit interval order. Then $X_{\Gamma}(\underline{x})$ is e-positive.
\end{conjecture} 

\noindent (In fact, the Stanley-Stembridge conjecture is stated more generally, but Guay-Paquet showed in \cite{Guay-Paquet2013} that the above special case implies the general version.) Shareshian and Wachs linked the Stanley-Stembridge conjecture in \cite{ShareshianWachs2016} to the theory of Hessenberg varieties as follows. For the discussion below, we assume the vertex set $V$ of $\Gamma$ is a finite subset of $\{1,2,\ldots\}$. Shareshian and Wachs consider a refinement of Stanley's chromatic symmetric polynomial by defining 
\[
X_{\Gamma}(\underline{x}, t) := \sum_{\textup{ proper } \kappa: V \to \{1,2,3,\ldots\}} t^{\asc{\kappa}} x^{\kappa}
\]
where 
\[
\mathrm{asc}(\kappa) := \lvert \{ e = \{i,j\} \in E \, \vert \, i < j \textup{ and } \kappa(i) < \kappa(j) \} \rvert.
\]
This polynomial is called the chromatic quasisymmetric function.  Evidently, evaluating $X_{\Gamma}(\underline{x},t)$ at $t=1$ recovers Stanley's $X_{\Gamma}(\underline{x})$. Shareshian and Wachs further conjectured that the coefficients of the $t^i$ in $X_{\Gamma}(\underline{x},t)$ are related to Hessenberg varieties as follows. 
Specifically, they show \cite[Proposition 4.1]{ShareshianWachs2016} that any incomparability graph $\Gamma$ of a unit interval order arises as the incomparability graph of an appropriately chosen Hessenberg function $h: \{1,2,\ldots,n\} \to \{1,2,\ldots,n\}$. Let $\mathrm{ch}: \mathcal{R}ep(\Symm_n) \to \Lambda_{\Q} := \Lambda \otimes \Q$ denote the characteristic map from the representation ring of $\Symm_n$ to the ring of symmetric functions 
in the variables $\underline{x} = (x_1, x_2, \ldots)$ with rational coefficients. It is well-known that $\mathrm{ch}(M^{\lambda}) = h_\lambda$ and $\mathrm{ch}(\mathcal{S}^{\lambda}) = s_{\lambda}$, where the $h_\lambda$ (respectively $s_{\lambda}$) are the complete symmetric (respectively Schur) polynomials. Also let $\omega: \Lambda_{\Q} \to \Lambda_{\Q}$ denote the standard (``Frobenius'') involution on the space of symmetric polynomials which takes $h_\lambda$ to the elementary symmetric polynomials $e_{\lambda}$ and vice versa. We now state the conjecture of Shareshian and Wachs \cite{ShareshianWachs2016} which is now a theorem thanks to the work of Brosnan and Chow \cite{BrosnanChow2015}. 

\begin{theorem}\label{theorem:BrosnanChow for SW} (\cite[Theorem 78]{BrosnanChow2015}) 
Let $\Gamma$ be the incomparability graph of a unit interval order and $h$ be the corresponding Hessenberg function. Then the coefficient of $t^i$ in the chromatic quasisymmetric function $X_{\Gamma}(\underline{x},t)$ is $\omega(\mathrm{ch}(H^{2i}(\Hess(\mathsf{S},h))))$. 
\end{theorem}

We can now make precise the argument that deduces the classical Stanley-Stembridge conjecture
from the graded Stanley-Stembridge conjecture (Conjecture~\ref{conj:Stanley-Stembridge}). 

\begin{proposition}\label{prop: classical from graded} 
Let $\Gamma$ and $h$ be as above. If $H^{2i}(\Hess(\mathsf{S},h)) = \sum_{\lambda \vdash n} c_{\lambda,i} M^{\lambda} \in \mathcal{R}ep(\Symm_n)$ and the $c_{\lambda,i}$ are non-negative for all $\lambda \vdash n$, then $X_{\Gamma}(\underline{x})$ is $e$-positive. 
\end{proposition}

\begin{proof} 
By construction of the map $\mathrm{ch}$, we have $\mathrm{ch}(H^{2i}(\Hess(\mathsf{S},h))) = \mathrm{ch}(\sum c_{\lambda,i} M^{\lambda}) = \sum c_{\lambda,i} \mathrm{ch}(M^{\lambda}) = \sum c_{\lambda,i} h_{\lambda}$ is a non-negative linear combination of the $h_{\lambda}$. Thus $\omega(\mathrm{ch}(H^{2i}(\Hess(\mathsf{S},h)))) = \sum c_{\lambda, i} \omega(h_{\lambda}) = 
\sum c_{\lambda,i} e_{\lambda}$ is a non-negative combination of the $e_{\lambda}$. Thus the coefficient of $t^i$ in the chromatic quasi-symmetric polynomial $X_{\Gamma}(\underline{x},t)$ is $e$-positive; by evaluation at $t=1$, the same is true for $X_{\Gamma}(\underline{x})$. 
\end{proof}

The remainder of this section will be a review of some standard facts in the representation theory of $\Symm_n$ as well as a fundamental result of Brosnan and Chow on Hessenberg varieties.

The following lemma is straightforward and probably well-known. 
Let $V = \oplus_{\ell \geq 0} V_\ell$ be a graded  $\Symm_n$-representation, i.e. $\Symm_n$ preserves each subspace $V_\ell$. Let $d_{\lambda,V}$ and $c_{\lambda,V}$ (respectively $d_{\lambda, V_{\ell}}, c_{\lambda, V_{\ell}}$) denote the integers associated to $V$ (respectively $V_\ell$ for each $\ell \geq 0$) given by the decomposition of $V$ (respectively $V_{\ell}$) into $\mathcal{S}^{\lambda}$'s and $M^{\lambda}$'s as elements of $\mathcal{R}ep(\Symm_n)$.

\begin{lemma}\label{lemma: dlambda zero implies all other zero} 
Let $k$ be a positive integer. Suppose $d_{\lambda, V} = 0$ for all $\lambda \vdash n$ with more than $k$ parts. Then 
\begin{enumerate} 
\item $d_{\lambda, V_{\ell}} = 0$ for all $\lambda \vdash n$ with more than $k$ parts and for all $\ell \geq 0$, 
\item $c_{\lambda, V}=0$ for all $\lambda \vdash n$ with more than $k$ parts, and 
\item $c_{\lambda, V_{\ell}} = 0$ for all $\lambda \vdash n$ with more than $k$ parts and for all $\ell \geq 0$. 
\end{enumerate}
\end{lemma}

\begin{proof} 
Since $d_{\lambda, V} =\sum_{\ell} d_{\lambda, V_{\ell}}$, if $d_{\lambda, V}=0$ then $d_{\lambda, V_{\ell}} = 0$ for all $\ell$. This proves (1). 
To prove (2), recall first that Young's rule implies $\mathbf{d}_{V} = K \mathbf{c}_{V}$. We have already seen that $K$ is an upper-triangular matrix with $1$s along the diagonal. Hence $K^{-1}$ has the same properties.  Also, since the total order $\preccurlyeq$ orders partitions by the number of parts, the given hypothesis on the $d_{\lambda, V}$'s implies that the vector $\mathbf{d}_V$ has coordinates all equal to $0$ below a certain point. Since $\mathbf{c}_V = K^{-1}\mathbf{d}_V$, we conclude that $\mathbf{c}_V$ must have the same property. 
The last claim follows from (1) by an argument identical to the discussion above, applied to $V_\ell$ instead of $V$. 
\end{proof}

Next we recall some facts which we later use to show that two representations are isomorphic. 
Given any finite-dimensional representation $V$ of $\Symm_n$ and any partition $\nu$ of $n$, we may consider $V^{\Symm_{\nu}}$, the $\Symm_{\nu}$-stable subspace of $V$. The following is well-known (see e.g. \cite[Proposition 8]{BrosnanChow2015}). 

\begin{proposition}\label{prop: Brosnan Chow}
Let $V$ and $W$ be finite-dimensional representations of $\Symm_n$. Then $V$ and $W$ are isomorphic as $\Symm_n$-representations if and only if 
\[
\dim \,V^{\Symm_{\nu}} = \dim \,W^{\Symm_{\nu}} \textup{ for all partitions } \nu \vdash n.
\]
\end{proposition} 

In the setting of the representation ring, there is a similar statement. For two partitions $\lambda,\nu$ of $n$, we set the notation 
\begin{equation}\label{eq:def Nlambdanu}
N_{\lambda, \nu} := \dim (M^{\lambda})^{\Symm_\nu}.
\end{equation}
Let $N = (N_{\lambda,\nu})$ denote the matrix with these entries. From \cite{Stanley-EnumCombVol2} we know that 
\[
N_{\lambda, \nu} = \sum_{\mu \vdash n} K_{\mu, \lambda} K_{\mu, \nu} \quad 
\textup{ and hence } \quad N = K^T K 
\]
where $K=(K_{\lambda \mu})$ is the Kostka matrix and $K^T$ denotes the transpose of $K$. In particular, since $K$ is invertible over $\Z$, it follows that $N$ is invertible over $\Z$. Now suppose we have two elements $\sum a_\lambda M^{\lambda}, \sum b_{\lambda} M^{\lambda}$ in $\mathcal{R}ep(\Symm_n)$ where $a_\lambda, b_{\lambda} \in \Z$ for all $\lambda \vdash n$. By definition, $\sum a_\lambda M^\lambda = \sum b_\lambda M^\lambda$ if and only if $a_\lambda = b_\lambda$ for all $\lambda \vdash n$, or equivalently,  $\mathbf{a} = \mathbf{b}$ where $\mathbf{a} = (a_{\lambda})_{\lambda \vdash n}$ and $\mathbf{b} = (b_{\lambda})_{\lambda \vdash n}$ are (column) vectors with entries $a_{\lambda}, b_{\lambda}$ respectively. Since $N$ is invertible, the following analogue of Proposition~\ref{prop: Brosnan Chow} for $\mathcal{R}ep(\Symm_n)$ is immediate. 

\begin{proposition}\label{prop: rep ring analogue} 
Let $\sum_{\lambda \vdash n} a_\lambda M^{\lambda}, \sum_{\lambda \vdash n} b_{\lambda} M^{\lambda}$ be elements in $\mathcal{R}ep(\Symm_n)$. The following are equivalent:
\begin{enumerate} 
\item   $\sum_{\lambda \vdash n} a_\lambda M^\lambda = \sum_{\lambda \vdash n} b_\lambda M^\lambda$ in $\mathcal{R}ep(\Symm_n)$,
\item $N \mathbf{a} = N \mathbf{b}$ 
\item $\sum_{\lambda \vdash n} a_\lambda \dim(M^\lambda)^{\Symm_\nu} = \sum_{\lambda \vdash n} b_\lambda \dim(M^\lambda)^{\Symm_\nu}$ for all $\nu \vdash n$.
\end{enumerate} 
\end{proposition}

The above discussion shows that proving equality in the representation ring can be viewed as a linear algebra problem. In the case in which these vectors $\mathbf{a} = (a_{\lambda})_{\lambda \vdash n}$ and $\mathbf{b} = (b_{\lambda})_{\lambda \vdash n}$ have coordinates equal to $0$ below a certain point, it will be convenient to further simplify the problem. We now make this more precise. Fix a positive integer $k$. 
Let $\pi_k(\mathbf{a}), \pi_k(\mathbf{b})$, $\pi_k(K), \pi_k(N)$ denote the submatrices obtained from $\mathbf{a}, \mathbf{b}$, $K, N$ by taking the only those entries corresponding to partitions with $\leq k$ parts. (Intuitively, 
this corresponds to taking the ``top parts'' of $\mathbf{a}, \mathbf{b}$ and the ``upper-left corners'' of $K$ and $N$.) 

\begin{lemma}\label{lemma: restrict to submatrix} 
Let $\sum a_\lambda M^{\lambda}, \sum b_{\lambda} M^{\lambda}$ be elements in $\mathcal{R}ep(\Symm_n)$. Let $k$ be a positive integer. Suppose that $a_{\lambda} = b_{\lambda} = 0$ for all $\lambda \vdash n$ with more than $k$ parts. Then the following are equivalent. 
\begin{enumerate} 
\item $\sum a_\lambda M^{\lambda} = \sum b_{\lambda} M^{\lambda}$
\item $\pi_k(N) \pi_k(\mathbf{a}) = \pi_k(N) \pi_k(\mathbf{b})$
\item $\sum a_\lambda \dim(M^\lambda)^{\Symm_\nu} = \sum b_\lambda \dim(M^\lambda)^{\Symm_\nu}$ for all $\nu \vdash n$ with $\leq k$ parts. 
\end{enumerate} 
\end{lemma}

\begin{proof}[Sketch of the proof.]
It is straightforward to see that the essential claim is that $\pi_k(N)$ (the ``upper-left corner'' of $N$) is invertible. Recall that $N=K^T K$ and $K$ is upper-triangular. Hence $\pi_k(N) = \pi_k(K^T) \pi_k(K)$. Moreover, $\pi_k(K^T) = \pi_k(K)^T$ and since $K$ is upper-triangular with $1$'s along the diagonal, the upper-left corner $\pi_k(K)$ has the same properties, and in particular is invertible. Thus $\pi_k(N)$ is also invertible. 
\end{proof}

Finally, we recall a fundamental result of Brosnan and Chow which  identifies the dimension of the subspaces $H^*(\Hess(\mathsf{S},h))^{\Symm_{\nu}}$ with the dimension of the cohomology of a regular Hessenberg variety and which we use repeatedly below.

\begin{theorem}\label{thm: BrosnanChow main thm}  ( \cite[Theorem 76]{BrosnanChow2015})
Let $n$ be a positive integer and $h: [n] \to [n]$ a Hessenberg function. Let $\nu \vdash n$ be a partition of $n$, $\mathsf{X}_\nu$ a regular operator of Jordan type $\nu$, and $\mathsf{S}$ a regular semisimple operator. Then for each non-negative integer $i$, 
\[
\dim (H^{2i}(\Hess(\mathsf{S}, h)))^{\Symm_\nu} = \dim H^{2i}(\Hess(\mathsf{X}_{\nu}, h)).
\]
\end{theorem}


\subsection{Incomparability graphs, acyclic orientations, and Stanley's theorem} \label{sec: graphs and orientations}

In this section we recall some graph-theoretic data which can be constructed from a Hessenberg function. 
Let $n$ be a positive integer and suppose $h: [n] \to [n]$ is a Hessenberg function. 

\begin{definition}\label{definition:incomparability graph} 
We define the \textbf{incomparability graph} $\Gamma_h = (V(\Gamma_h), E(\Gamma_h))$ associated to $h$ 
as follows.  
The vertex set $V(\Gamma_h)$ is $[n] = \{1,2,\ldots,n\}$. The edge set $E(\Gamma_h)$ is defined as follows: $\{i,j\}\in E(\Gamma_h)$ if $1 \leq j < i \leq n$ and $i \leq h(j)$.  
\end{definition}

The incomparability graph is a visual representation of the set of roots $\Phi_h^-$.  In particular, it is straightforward to see that the roots of $\Phi_h^-$ are in bijection with the edges of $\Gamma_h$.  

\begin{example}\label{ex: incomparability graphs} The incomparability graphs for $h=(2,4,4,4)$ and $h=(3,4,5,5,5)$ are given below.
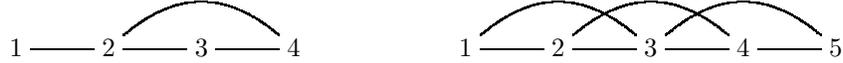
\vspace*{.15in}
\[\xymatrix{1 \ar@{-}[r] & 2 \ar@{-}[r] \ar@{-}@/^1.5pc/[rr] & 3 \ar@{-}[r] & 4 & & 
1 \ar@{-}[r]\ar@{-}@/^1.5pc/[rr]  & 2 \ar@{-}[r]\ar@{-}@/^1.5pc/[rr]  & 3 \ar@{-}[r] \ar@{-}@/^1.5pc/[rr]  & 4 \ar@{-}[r]  & 5
}\]
\end{example}

Recall that an \textbf{orientation} $\omega$ of (the edges of) a graph is an assignment of a direction (i.e. orientation) to each edge $e \in E(\Gamma_h)$. Equivalently, $\omega$ assigns to each edge $e$ a source and a target; we notate the source (respectively target) of $e$ according to the orientation $\omega$ by $\src_\omega(e)$ (respectively $\tgt_\omega(e)$).  A (directed) \textbf{cycle} is a sequence of vertices starting and ending at the same vertex whose edges are oriented consistently with the order of the vertices in the sequence.
We say that an orientation $\omega$ is \textbf{acyclic} if there are no (directed) cycles in the corresponding oriented graph.  Let 
\[
{\mathcal{A}}(\Gamma_h) := \{ \hsm \omega \hsm \vert \hsm \omega \textup{ is an acyclic orientation of } \Gamma_h \}
\]
denote the set of all acyclic orientations of $\Gamma_h$. 
Moreover, given an orientation $\omega$, a \textbf{sink} associated to $\omega$ is a vertex $v$ of the graph such that, for all edges $e$ incident to the vertex $v$, the edge ``points towards $v$'', i.e., $\tgt_\omega(e)=v$ for all edges $e$ incident to $v$.  
It will turn out to be extremely important to pay close attention to the \emph{number} of sinks associated to a given orientation. Thus we define
\begin{equation}\label{eq:def AkGamma_h}
\mathcal{A}_k(\Gamma_h) := \{ \hsm \omega \in \mathcal{A}(\Gamma_h) \hsm \vert \hsm \omega
\textup{ has exactly $k$ sinks}  \}.
\end{equation}
Since every acyclic orientation has at least one sink \cite[Section 8.6, Exercise 4]{Rahman2017}, we have $\mathcal{A}(\Gamma_h) =\bigsqcup_{k\geq 1} \mathcal{A}_k(\Gamma_h)$.  

The following is a result of Shareshian and Wachs \cite[Theorem 5.3]{ShareshianWachs2016} which generalizes a theorem of Stanley.  Following their terminology, for an orientation $\omega$ of $\Gamma_h$, we let 
\begin{equation}\label{eq:definition asc}
\asc(\omega) := \lvert \{  e=\{a,b\} \in E(\Gamma_h) \hsm \vert \hsm \src_\omega(e)=a, \tgt_\omega(e)=b, \textup{ and } a<b \} \rvert.
\end{equation}
In other words, if $\Gamma_h$ is drawn as in Example \ref{ex: incomparability graphs} with the labels of the vertices increasing from left to right, then $\asc(\omega)$ is the total number of edges which ``point to the right''. 

\begin{theorem} \label{thm: Stanley} \cite[Theorem 5.3]{ShareshianWachs2016} 
Let $n$ be a positive integer and $h: [n] \to [n]$ a Hessenberg function. Let $c_{\lambda,i}$ denote the coefficients appearing in~\eqref{eq: decomp into Mlambda}. Then 
\[
\sum_{\lambda \vdash n,\; \mathrm{parts}(\lambda) = k} c_{\lambda,i} = 
\lvert \{ \omega \hsm \vert \hsm \omega \in \mathcal{A}_k(\Gamma_h) \textup{ and } \asc(\omega)=i \} \rvert.
\]
\end{theorem} 

Since there is only one partition of $n$ with exactly $1$ part, namely $\lambda=(n)$, and because the representation $M^{(n)}$ corresponding to this partition is the trivial representation \cite{Ful97}, we may immediately conclude the following, which will be important to us later on. 

\begin{corollary}\label{corollary: triv coeff is nonneg}
Under the conditions in the above theorem, the multiplicity of the trivial representation in $H^{2i}(\Hess(\mathsf{S},h))$ is the number of acyclic orientations $\omega$ of $\Gamma_h$ with exactly $1$ sink such that $\asc(\omega)=i$. Equivalently, for $\lambda = (n)$ the trivial partition, we have 
\[
c_{(n),i} = \lvert \{ \omega \hsm \vert \hsm \omega \in \mathcal{A}_1(\Gamma_h) \textup{ and } \asc(\omega)=i \} \rvert. 
\]
\end{corollary}  

The following is also immediate from Theorem~\ref{thm: Stanley} by summing over $k$. 

\begin{corollary}\label{corollary: sum clambda for all i} 
Under the conditions in the above theorem we have 
\[
\sum_{\lambda \vdash n} c_{\lambda, i} = \lvert \{ \omega \, \vert \, \omega \in \mathcal{A}(\Gamma_h) \textup{ and } \asc(\omega) = i \} \rvert.  
\]
\end{corollary}

\subsection{$P_h$-tableaux and a result of Gasharov}\label{sec: Ph tableaux}

Recall that the goal of the present manuscript is to prove a result about the coefficients $c_{\lambda}$ and $c_{\lambda,i}$ appearing in~\eqref{eq: decomp into Mlambda}, for certain cases of Hessenberg functions $h$. We have also seen that the coefficients $d_{\lambda}$ from~\eqref{eq:decomp into Specht} are intimately related to the $c_{\lambda}$. In preparation for the arguments in the following sections, we now take a moment to recall a combinatorial object called a $P_h$-tableau, and the result of Gasharov which computes the $d_{\lambda}$'s in terms of $P_h$-tableau.

\begin{definition}\label{def: P-tab} Fix a Hessenberg function $h: [n] \to [n]$.  A \textbf{$P_h$-tableau of shape $\lambda$} is a filling of a Young diagram of shape $\lambda \vdash n$ with the integers of $[n]$ such that 
\begin{enumerate}
\item each integer $1,2,..., n$ appears exactly once,
\item if $i\in [n]$ appears immediately to the right of $j\in [n]$ then $i>h(j)$, and
\item if $i\in [n]$ appears immediately below $j\in [n]$ then $j \leq h(i)$.
\end{enumerate}
\end{definition} 

\begin{example}  Let $n=5$ and let $h=(2,3,4,5,5)$. 
Then there are nine $P_h$-tableau of shape $(2,2,1)$:  
\[
\young(13,24,5)\quad \young(14,25,3) \quad \young(13,25,4)\quad \young(14,35,2)\quad 
\young(15,24,3) \quad \young(24,13,5) \quad \young(24,15,3)\quad \young(25,14,3) \quad \young(35,24,1)
\]  
\end{example}

Recall that every partition $\lambda \vdash n$ has a dual partition $\lambda^{\vee}$ whose Young diagram is the transpose of the Young diagram of $\lambda$.  The following theorem, which gives a positive, combinatorial formula for the coefficients $d_{\lambda}$, is due to Gasharov \cite{Gasharov1996}.  There is also a graded version of the theorem, due to Shareshian and Wachs \cite[Theorem 6.3]{ShareshianWachs2016}, but we will only need the ungraded version below.

\begin{theorem}\label{thm: irreducible coefficients} Let $n$ be a positive integer and let $h: [n]\to [n]$ be a Hessenberg function. Let $d_{\lambda}$ denote the coefficients  appearing in~\eqref{eq:decomp into Specht}.  Then 
\[
d_\lambda = \lvert \{ \textup{  $P_h$-tableaux of shape $\lambda^{\vee}$ } \} \rvert.
\]
\end{theorem}


\section{Abelian Hessenberg varieties} \label{sec: abelian Hessenberg varieties}

In the previous sections we outlined the motivation behind this paper and recalled some background We are finally ready to begin our own arguments in earnest, and the first task is to establish the terminology (and hypothesis) which allows us to make our arguments -- namely, the definition of an abelian ideal and an abelian Hessenberg variety. We also briefly discuss how our special case relates to other situations that have been studied previously in the literature.

In Section~\ref{sec:background} we defined an ideal of $\Phi^-$ associated to a Hessenberg function $h$. We now introduce the definition which is central to this manuscript. 

\begin{definition}\label{def: abelian} 
We say that an ideal $I\subseteq \Phi^-$ is \textbf{abelian} if $\alpha+\beta\notin \Phi^-$ for all $\alpha, \beta\in I$.  \end{definition} 

The notion of abelian ideals is not new in the context of Lie theory. However, as far as we are aware, its use in the study of Hessenberg varieties is new.  The following definition is not essential to this paper but we include it because it frequently arises in the literature. 
 
\begin{definition} 
 Let $I$ be an ideal in $\Phi^-$. We say that $I$ is \textbf{strictly negative} if $-\Delta\cap I$ is empty. 
\end{definition}

Note that if $I=I_h$ is the ideal of $\Phi^-$ associated to a Hessenberg function $h$, then $-\Delta \cap I$ is empty if and only if $-\Delta \subseteq \Phi_h$.  The following is well-known, which partly explains why it is common practice in the study of Hessenberg varieties to assume that $I_h$ is strictly negative. 

\begin{lemma}\cite[Theorem 3.4]{Precup2015}
Let $h$ be a Hessenberg function and $\mathsf{X} \in \mathfrak{gl}(n,\C)$ be a semisimple matrix. Then the corresponding semisimple Hessenberg variety $\Hess(\mathsf{X},h)$ is connected if and only if $I_h$ is strictly negative. 
\end{lemma}

\begin{example}  In the case $n=4$, there are $8$ abelian ideals in $\Phi^-$. The reader may check that these correspond to the Hessenberg functions $(1,4,4,4), (2,2,4,4), (2,3,4,4), (2,4,4,4), (3,3,3,4), (3,3,4,4), (3,4,4,4)$ and $(4,4,4,4)$. Among these, those that are strictly negative are
\[
(4,4,4,4), (3,4,4,4), (3,3,4,4), (2,4,4,4), (2,3,4,4).
\]
and their corresponding ideals $I_h$ are, respectively, 
\[
	\emptyset, \{ t_4-t_1 \}, \{ t_4-t_1, t_4-t_2\}, \{ t_4-t_1, t_3-t_1 \}, \{ t_4-t_1, t_4-t_2, t_3-t_1 \}.
\]

\end{example}

The following extends the notion of abelian ideals to their corresponding Hessenberg varieties. 

\begin{definition}  We say that the Hessenberg variety $\Hess(\mathsf{X},h)$ and the corresponding Hessenberg function $h$ are \textbf{abelian}, if $I_h$ is abelian.
\end{definition}

We thank Timothy Chow for the following remark. 

\begin{remark}\label{rem: alt-def}  There is also a purely combinatorial characterization of abelian Hessenberg functions as follows.   Define the \textit{index} $\mathrm{index}(h)$ of a Hessenberg function to be the largest integer $i$ such that $h(i)<n$.  Then $h$ is abelian if and only if $h(1)\geq \mathrm{index}(h)$. Indeed, if $i$ is the index of $h$ and $h(1)<i$, then $t_i-t_1, t_n-t_i\in I_h$ since $h(1)<i\leq h(i)<n$ and $(t_i-t_1)+ (t_n-t_i) = t_n-t_1\in \Phi^-$ implying $I_h$ is not abelian.  On the other hand, if $I_h$ is not abelian, then there exists roots $t_j -t_k,t_k-t_{\ell}\in I_h$, so $j>k>\ell$ and $j>h(k)$ and $k>h(\ell)$.  Now, since $h(k)<j \leq n$, the index $i$ of $h$ is at least $k$, so we conclude that $h(1)\leq h(\ell)<k\leq i$.
\end{remark}

Abelian ideals of $\Phi^-$ (or equivalently, of $\Phi^+$) are the source of many combinatorial and Lie-theoretic formulas.  The number of abelian ideals in the negative roots $\Phi^-$ of $\mathfrak{gl}(n,\C)$ grows exponentially in $n$. This is a special case of a result by D. Peterson, as recorded by Kostant in \cite[Theorem 2.1]{Kostant1998}.

\begin{proposition} (Peterson, \cite[Theorem 2.1]{Kostant1998}) 
 Let $\g$ be any semisimple Lie algebra and let $\Phi^-$ denote its set of negative roots. Then there are exactly $2^{\textup{rk}(\mathfrak{g})}$ abelian ideals in $\Phi^-$, where $\textup{rk}(\mathfrak{g})$ denotes the rank of $\mathfrak{g}$. 
 \end{proposition}

\begin{remark} 
We can also ask: how many ideals in $\Phi^-$ are both abelian and strictly negative? Suppose $h: [n] \to [n]$ is a Hessenberg function such that its corresponding ideal $I = I_h$ is abelian. Suppose in addition that $I_h$ is not strictly negative. Then the corresponding Hessenberg space $H$ is a maximal, standard parabolic Lie subalgebra of $\mathfrak{gl}(n,\C)$.  In particular, 
\[
\left\lvert \{ \textup{ abelian ideals of $\Phi^-$ which are strictly negative } \} \right\rvert 
= 2^{n-1} - (n-1).
\]
\end{remark}

The set of abelian, strictly-negative Hessenberg varieties contains examples studied previously. For instance, suppose $n$ is positive and $n \geq 3$.  For any strictly negative Hessenberg function of the form $h = (m_1, m_2, n,n,\ldots,n)$, Shareshian and Wachs proved results on the associated chromatic quasisymmetric polynomial which, when paired with Brosnan and Chow's Theorem~\ref{theorem:BrosnanChow for SW}, proves Conjecture~\ref{conj:Stanley-Stembridge} for that case. It is not difficult to see that such a Hessenberg function is abelian using the alternative definition of an abelian Hessenberg function given in Remark~\ref{rem: alt-def}.

In addition, the representation $H^*(\Hess(\mathsf{S},h))$ for the Hessenberg function $h$ associated to any standard parabolic subalgebra $\mathfrak{p}$ was determined to be a direct sum of $M^{\lambda}$'s by Teff \cite[Theorem 4.20]{Teff2013a}. When $\mathfrak{p}$ is a maximal parabolic subalgebra, the corresponding Hessenberg function is abelian so this is another special case of our result.

\begin{example}  Consider $h=(3,4,5,6,6,6)$, as in Example \ref{ex: pictures}.  This Hessenberg function corresponds to the abelian ideal
\[
	I_h = \{ t_3-t_1, t_4-t_1, t_4-t_2, t_5-t_1, t_5-t_2, t_5-t_3 \}.
\]
However $h$ is not of the form $(m_1, m_2, 6,6,6,6)$, nor is it the Hessenberg function corresponding to a standard parabolic subalgebra.  Therefore $h$ is an example of a Hessenberg function for which our proof of Conjecture~\ref{conj:Stanley-Stembridge} is new.
\end{example}


\section{Sink sets, maximum sink set size, and an inductive description of acyclic orientations} \label{sec: sink sets and induction}

The results in Section~\ref{sec: graphs and orientations} make it evident that the set of acyclic orientations, and the cardinalities of the sink sets associated to them, play a crucial role in determining the coefficients $c_{\lambda,i}$. The contribution of this manuscript is to further develop this circle of ideas by analyzing the sink sets themselves. We also pay close attention to those sinks sets which are of maximal cardinality. Below, we make these ideas more precise.

\subsection{Sink sets and induced subgraphs} \label{sec: Sink sets and induced subgraphs}

Our first lemma gives several equivalent characterizations of a subset of $V(\Gamma_h)$ which can appear as a set of sinks for some acyclic orientation. We prepare some terminology. 
First, for a fixed $\Gamma_h$ and acyclic orientation $\omega$ of $\Gamma_h$, let 
\[
\sk(\omega) := \{ v \in V(\Gamma_h) \hsm \vert \hsm v \textup{ is a sink of } \omega \}
\]
denote the set of sinks of $\omega$. We say $\sk(\omega)$ is a \textbf{sink set}. 
Recall that an \textbf{independent set} of vertices in $\Gamma_h$ is a subset of $V(\Gamma_h)$ such that no two of them are connected by an edge in $\Gamma_h$. 
We have the following.

\begin{lemma}\label{lemma: equivalency sink and independence} 
Let $h: [n] \to [n]$ be a Hessenberg function and $\Gamma_h$ be the associated incomparability graph. Let $T = \{\ell_1 < \ell_2 < \cdots < \ell_k\}$ be a subset of $V(\Gamma_h)$ for $k$ a positive integer. Then the following are equivalent: 
\begin{enumerate} 
\item $T$ is a sink set, i.e., there exists an acyclic orientation $\omega \in \mathcal{A}_k(\Gamma_h)$ such that $T = \sk(\omega)$, 
\item $\ell_{i+1} > h(\ell_i)$ for all $i \in [k-1]$, and 
\item $T$ is an independent set in $\Gamma_h$. 
\end{enumerate} 
In particular, 
\[
\max \{ \lvert T \rvert \hsm \vert \hsm T \subseteq V(\Gamma_h) \textup{ and $T$ is independent } \} = \max \{ \lvert T \rvert \hsm \vert \hsm T \subseteq V(\Gamma_h) \textup{ and $T$ is a sink set }\}.
\]
\end{lemma} 

\begin{proof}  We first show that (1) implies (2). Suppose $T$ is a sink set. We wish to show $ \ell_{i+1}> h(\ell_i)$ for all $i$, $1 \leq i \leq k-1$. If $k=1$, the condition is vacuous and there is nothing to check. If $k>1$, suppose for a contradiction that there exists $i \in [k-1]$ with $\ell_{i+1} \leq h(\ell_i)$.  Then by construction of $\Gamma_h$ there exists an edge $e$ between $\ell_i$ and $\ell_{i+1}$. For any orientation $\omega$ of $\Gamma_h$, we must have either $\tgt_{\omega}(e)=\ell_i$ or $\tgt_{\omega}(e)=\ell_{i+1}$, and not both. Thus $\ell_i$ and $\ell_{i+1}$ cannot be simultaneously contained in $\sk(\omega)$, contradicting the fact that $T$ is a sink set. Thus (1) implies (2). 

Next we prove that (2) implies (3). Note that since $\ell_1 < \ell_2 < \cdots < \ell_k$ by assumption, if $\ell_{i+1} > h(\ell_i)$ for all $i \in [k-1]$ then it follows that $\ell_b>h(\ell_a)$ for any pair $a<b$, $a, b \in [k]$. By construction of $\Gamma_h$ this implies there are no edges in $\Gamma_h$ connecting any two of the vertices in $T$. Hence, by definition, $T$ is an independent set.

Finally we prove (3) implies (1). Suppose $T = \{\ell_1 < \ell_2 < \cdots < \ell_k\}$ is an independent set. We explicitly construct an acyclic orientation $\omega$ of $\Gamma_h$ with sink set precisely $T$ as follows. We first consider the set of edges $e$ in $\Gamma_h$ which are incident to a vertex in $T$. Note that any such $e$ is incident to only one vertex, say $\ell_i$, in $T$, because $T$ is independent. We assign an orientation to any such $e$ by requiring $\tgt_\omega(e)=\ell_i$. 
Next consider all edges in $\Gamma_h$ which are not incident to any vertex in $T$. To any such edge $e=\{v, v'\}$ where $v < v'$ we assign the orientation which makes the edge ``point to the left''; more precisely, $\tgt_{\omega}(e)=v$. The above clearly defines an orientation $\omega$ on $\Gamma_h$. To finish the argument we must prove that it is acylic. To see this, first note that no cycle can contain any vertex in $T$, since all edges have to ``point in'' to such a vertex (and for a cycle, at least one edge has to ``point out'' and one edge has to ``point in''). So the only possibility is that a cycle contains only vertices that are not in $T$. But all edges of the form $e=\{v,v'\}$ for $v<v'$ not contained in $T$ were oriented 
``to the left'', implying that there cannot be a cycle in this case either. Thus $\omega$ is acyclic as desired.

The last claim of the lemma follows immediately from the equivalence of (1) and (3). 
\end{proof}

As already mentioned, we wish to focus on the sink sets themselves, not just the acyclic orientations which give rise to them. Let $\mathcal{P}(V(\Gamma_h))$ be the power set of $V(\Gamma_h)$. 
We let 
\[
\SK(\Gamma_h) := \{ \hsm \sk(\omega) \hsm \vert \hsm \omega \in {\mathcal{A}}(\Gamma_h) \} \subseteq \mathcal{P}(V(\Gamma_h))
\]
denote the set of all subsets of $V(\Gamma_h)$ which can arise as the sink set of some acyclic orientation. Similarly, we let $\SK_k(\Gamma_h)$ denote the subset of $SK(\Gamma_h)$ consisting of sink sets of cardinality $k$. By definition, we have 
\begin{equation}\label{eq: sink set decomposition} 
\mathcal{A}_k(\Gamma_h) = \bigsqcup_{T \in SK_k(\Gamma_h)} \{ \omega \in \mathcal{A}_k(\Gamma_h) \hsm \vert \hsm \sk(\omega) = T \}.
\end{equation} 
We call this the \textbf{sink set decomposition} and it is conceptually central to our later arguments. 
Indeed, recall that the statements of Theorems~\ref{thm: Stanley}, Corollary~\ref{corollary: triv coeff is nonneg}, and Corollary~\ref{corollary: sum clambda for all i} relate the coefficients $c_{\lambda,i}$ to cardinalities of certain subsets of $\mathcal{A}_k(\Gamma_h)$ for various $k$. Thus, the set of sink sets $SK_k(\Gamma_h)$ provides a way to further decompose $\mathcal{A}_k(\Gamma_h)$, and refine our understanding of the $c_{\lambda,i}$.  More concretely, we now define a graph $\Gamma_h - T$ on $n-|T|$ vertices using the data of the original graph $\Gamma_h$ together with the data of a sink set $T$. This construction will be critical for the induction argument in Sections~\ref{section: inductive formula} and~\ref{section: the proofs}. 

Let $k \geq 1$ and suppose $T \in SK_k(\Gamma_h)$. Intuitively, the graph $\Gamma_h - T$ is obtained from $\Gamma_h$ by ``deleting'' $T$. Moreover, we will identify this graph as the incomparability graph of a certain Hessenberg function $h_T$. We proceed in steps. The underlying set of vertices $V(\Gamma_h - T)$ is defined to be $V(\Gamma_h) \setminus T$. We also define the edge set $E(\Gamma_h - T)$ as follows: two vertices in $V(\Gamma_h - T)$ are connected in $\Gamma_h - T$ if and only if there exists an edge connecting them in $E(\Gamma_h)$. In graph theory $\Gamma_h - T$ is called the \textbf{induced subgraph} corresponding to the vertices $V(\Gamma_h)\setminus T\subseteq V(\Gamma_h)$. 

In what follows it will sometimes be useful to label the vertices of $\Gamma_h - T$ by the integers $\{1,2,\ldots,n-k\} = [n-k]$ rather than $V(\Gamma_h) -T = [n] \setminus T$. Intuitively, this is straightforward: we simply re-label the vertices $V(\Gamma_h-T)=V(\Gamma_h) \setminus T$ by $[n-k]$ using the ordering induced by the ordering on the original $V(\Gamma_h)$. More precisely, for each $1\leq j \leq n$, let $j'$ denote the number of vertices $i\in T$ such that $i\leq j$.  The function $\phi_T: [n] \to [n-k]$ defined by $\phi_T(j)= j-j'$ is surjective from $[n]$ to $[n-k]$ and restricts to a bijection between $[n]\setminus T$ and $[n-k]$.

\begin{example}\label{example: smaller graph} 
Consider the graph $\Gamma_h$ for $h=(3,4,5,5,5)$ and let $T = \{2,5\}$. Then $T$ is indeed a sink set, for the following acyclic orientation of $\Gamma_h$. 
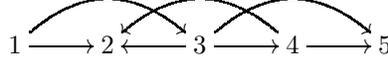
\vspace*{.15in}
\[\xymatrix{1 \ar[r]\ar@/^1.5pc/[rr]  & 2 & 3 \ar[l]\ar[r] \ar@/^1.5pc/[rr]  & 4 \ar@/_1.5pc/[ll]  \ar[r]  & 5
}\]
We also draw the (unoriented) graphs $\Gamma_h$ and $\Gamma_h - T$ in the figure below. In the figure for $\Gamma_h$, the vertices of $T$ and all incident edges to $T$ are highlighted in red. The red edges and vertices are then deleted to obtain $\Gamma_h - T$ (with re-labelled vertices).
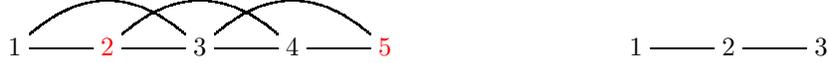
\vspace*{.15in}
\[\xymatrix{1 \ar@[red]@{-}[r] \ar@{-}@/^1.5pc/[rr]  & {\color{red}2} \ar@[red]@{-}[r]\ar@[red]@{-}@/^1.5pc/[rr]  & 3 \ar@{-}[r] \ar@[red]@{-}@/^1.5pc/[rr]  & 4 \ar@[red]@{-}[r]  & {\color{red}5} & & &
1 \ar@{-}[r] & 2 \ar@{-}[r] & 3
}\]
In this case, $\phi_T(1)=1$, $\phi_T(3)=2$, and $\phi_T(4)=3$.
\end{example}

Using the above bijection $\phi_T: [n] \setminus T \to [n-k]$ we may define a function $h_T: [n-k] \to [n-k]$ by setting 
\begin{equation}\label{eq: definition hT} 
h_T(\phi_T(i)) := \phi_T(h(i)) 
\end{equation}
for all $i \in [n] \setminus T$. Our next claim is that the smaller graph $\Gamma_h -T$ is in fact the incomparability graph of $h_T$, so $\Gamma_h - T \cong \Gamma_{h_T}$.  Note that in Example~\ref{example: smaller graph}, the above construction yields the Hessenberg function $h_T = (2,3,3)$.

\begin{lemma}\label{lem: induction graph}  Let $T\in \SK_k(\Gamma_h)$ and let 
$\Gamma_h -T$ and $h_T$ be defined as above. Then $h_T$ is a Hessenberg function, and $\Gamma_h - T = \Gamma_{h_T}$. 
\end{lemma} 

\begin{proof} 
First we show that $h_T$ is a Hessenberg function. Since $\phi_T$ and its inverse $\phi_T^{-1}: [n-k] \to [n] \setminus T$ are non-decreasing and $h$ is non-decreasing by definition of Hessenberg functions, it follows that $h_T$ is also non-decreasing. Next, the facts that $h(i)\geq i$ and $\phi_T$ is non-decreasing imply that 
\[
h_T(\phi_T(i)) = \phi_T(h(i)) \geq \phi_T(i)
\]
for all $i \in [n] \setminus T$. Therefore $h_T$ is a Hessenberg function.  

Second, we wish to show that $\Gamma_h - T$ is the incomparability graph of $h_T$. 
To do this, fix $i, j \in [n] \setminus T$ and $i<j$. It suffices to show that in $\Gamma_h - T$, there exists an edge between $\phi_T(i)$ and $\phi_T(j)$ if and only if $\phi_T(j) \leq h_T(\phi_T(i)) = \phi_T(h(i))$. To see this, note there exists an edge between $\phi_T(i)$ and $\phi_T(j)$ in $\Gamma_h - T$ if and only if there exists an edge between $i$ and $j$ in $\Gamma_h$ by definition of $\Gamma_h-T$. This holds if and only if $j \leq h(i)$ by definition of $\Gamma_h$. This implies 
$\phi_T(j) \leq \phi_T(h(i))$ since $\phi_T$ is non-decreasing. Finally, this holds if and only if 
$\phi_T(j) \leq h_T(\phi_T(i))$ by definition of $h_T$. It now suffices to show that for $i, j \in [n] \setminus T$, we have $\phi_T(j) \leq \phi_T(h(i)) \Rightarrow j \leq h(i)$. Note that $h(i)$ may not be in $[n] \setminus T$. Suppose for a contradiction that $j > h(i)$. Then we must have $\phi_T(h(i)) \leq \phi_T(j)$. This means $\phi_T(j) = \phi_T(h(i))$. Since $\phi_T$ is injective on $[n] \setminus T$ and $j \in [n] \setminus T$ this means $h(i) \in T$. Moreover,  it follows from the definition of $\phi_T$ that $h(i) > j$. This contradicts the initial assumption that $j > h(i)$, so we conclude $j \leq h(i)$ as desired. 
\end{proof}

We have already observed that the edges of an incomparability graph $\Gamma_h$ associated to a Hessenberg function are in one-to-one correspondence with the set of negative roots in $\Phi^-_h$. 
Our construction of a `smaller' graph $\Gamma_{h_T}\cong \Gamma_h - T$, suggests that there should be a correspondence between negative roots in $\Phi_{h_T}^-$ and a certain subset of $\Phi_h^-$ which is determined by $T$. We now make this precise. 
By Lemma~\ref{lem: induction graph}, we may describe the roots $\Phi_{h_T}^-$ and the ideal $I_{h_T}$ corresponding to $h_T$ using those of $h$ as follows:
\[
	\Phi_{h_T}^- = \{ t_{\phi_T(i)} -  t_{\phi_T(j)} \hsm \vert \hsm t_i-t_j \in \Phi_h^- \textup{ and } i,j \notin T \}
\]
and
\[
	I_{h_T} = \{ t_{\phi_T(i)} - t_{\phi_T(j)}  \hsm \vert \hsm t_i-t_j\in I_h \textup{ and } i,j \notin T \}.
\]
In our computations below, it will also be convenient to consider the subset of negative roots in $\Phi_h^-$ and $I_h$ which correspond to $\Phi_{h_T}^-$ and $I_{h_T}$, respectively, under the map $\phi_T$.  We set the notation 
\[
	\Phi_h^-[T]: = \{ t_i-t_j \hsm \vert \hsm t_i - t_j \in \Phi_h^- \textup{ and } i,j \notin T \}
\]
and
\[
	I_h[T]: = \{ t_i-t_j \hsm \vert \hsm t_i - t_j \in I_h \textup{ and } i,j\notin T \}. 
\]
There is an obvious bijection from $\Phi_h^-[T]$ to $\Phi_{h_T}^-$ and $I_h[T]$ to $I_{h_T}$ given by $t_i-t_j \mapsto t_{\phi_T(i)} - t_{\phi_T(j)}$.

Finally, we observe that the construction of the smaller graph $\Gamma_h - T \cong \Gamma_{h_T}$ from the data of $\Gamma_h$ also extends to orientations. Specifically, 
let $\omega\in \mathcal{A}_k(\Gamma_h)$ be any acyclic orientation such that $\sk(\omega)=T$.  Then the orientation $\omega$ naturally induces, by restriction, an orientation on $\Gamma_h-T = \Gamma_{h_T}$ (since the edges of $\Gamma_h - T$ are a subset of those of $\Gamma_h$). We denote this acyclic orientation on $\Gamma_{h_T}$ by $\omega_T$.

\begin{example}\label{ex: omega induces omegaT}  We continue with Example~\ref{example: smaller graph}. In the pictures below, we draw an orientation $\omega$ of $\Gamma_h$ on the left, and its corresponding induced orientation $\omega_T$ of $\Gamma_{h_T}$ on the right. For visualization purposes the sink set $T$ and its incident edges are highlighted in red. 
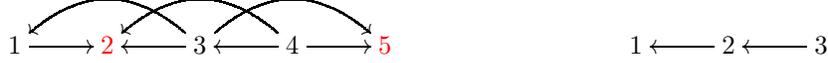
\vspace*{.15in}
\[\xymatrix{1 \ar@[red][r]   & {\color{red}2}  & 3  \ar@[red][l] \ar@/_1.5pc/[ll]  \ar@[red]@/^1.5pc/[rr] &  4  \ar[l] \ar@[red][r] \ar@[red]@/_1.5pc/[ll]   & {\color{red}5} & & &
1 & 2 \ar[l] & 3 \ar[l]
}\]
\end{example}


\subsection{Sink sets of maximal cardinality and an inductive description of acyclic orientations}\label{subsec:sink sets}

The main observation of the present section, recorded in Proposition~\ref{proposition: max sink set induction}, is that if $k$ is maximal, then the sets appearing on the RHS of the sink set decomposition~\eqref{eq: sink set decomposition} are in bijective correspondence with the set of \emph{all} acyclic orientations corresponding to the graphs $\Gamma_{h_T}$ for $T\in \SK(\Gamma_h)$. Moreover, this natural bijection gives a tight relationship between the number of ascending edges $\asc(\omega)$ of the orientation $\omega$ of the original graph $\Gamma_h$ with the number $\asc(\omega_T)$ of the induced orientation on the smaller graph $\Gamma_{h_T}$, where $\omega_T$ is described at the end of Section~\ref{sec: graphs and orientations} above. These ascending edge statistics record the degree -- i.e. the grading in $H^*(\Hess(\mathsf{S},h))$ -- in Theorem~\ref{thm: Stanley} and Corollary~\ref{corollary: triv coeff is nonneg}, so it is this relation which allows us to prove our ``graded'' results in Section~\ref{sec: main proofs}.

We begin by making precise the notion of a sink set of maximum possible size. 

\begin{definition}\label{definition: max sink set size} 
We define the \textbf{maximum sink-set size} $m(\Gamma_h)$ to be the maximum of the cardinalities of the sink sets $\sk(\omega)$ associated to all possible acyclic orientations of $\Gamma_h$, i.e., 
\begin{equation}\label{eq:definition max sink set length}
m(\Gamma_h) = \mathrm{max} \{ \hsm \lvert \sk(\omega) \rvert \hsm \vert \hsm \omega \in {\mathcal{A}}(\Gamma_h) \}.
\end{equation} 
\end{definition}

Note that the maximum clearly exists since $\lvert \sk(\omega) \rvert$ is bounded above by $n$. Furthermore, by Lemma \ref{lemma: equivalency sink and independence}, the maximal sink-set size of $\Gamma_h$ is also the maximum cardinality of an independent set of vertices in $\Gamma_h$.

\begin{example}\label{ex: max sink set size}  Continuing Example~\ref{example: smaller graph}, 
the sink set $T=\{2,5\}$ given in that example is in fact maximal, i.e., $m(\Gamma_h)=2$. Indeed, in this case any set of three vertices must have at least one edge incident with two of them, and thus cannot be independent. Finally, we note that for this orientation we have $\asc(\omega) = 5$, i.e., there are $5$ edges pointing to the right.
\end{example}

Let $m=m(\Gamma_h)$ be the maximum sink-set size for a fixed incomparability graph $\Gamma_h$ and Hessenberg function $h$ as in Definition~\ref{definition: max sink set size}.  
We need some terminology. Suppose $T\in \SK(\Gamma_h)$. Any acyclic orientation $\omega$ with sink set $T$ must have some number of edges oriented to the right, as determined by the vertices in $T$.

\begin{definition}  Suppose $T\in \SK(\Gamma_h)$. We define the \textbf{degree of $T$} to be
\[
	\deg(T): = \min\{\hsm \asc(\omega)  \hsm \vert \hsm \omega\in \mathcal{A}(\Gamma_h),\; \sk(\omega)=T \}.
\]
\end{definition}

The next lemma shows that in practice it is easy to compute $\deg(T)$ for any $T\in \SK(\Gamma_h)$.

\begin{lemma} \label{lem: deg(T) property1}  
Let $T \in SK(\Gamma_h)$. Then 
\begin{equation}\label{eq: characterize deg T}
\deg(T) = \big\lvert \{ e = \{\ell, \ell'\} \in E(\Gamma_h) \, \vert \, \ell' \in T, \, \ell < \ell' \} \big \rvert.
\end{equation}
Moreover, $\lvert \Phi_h^- \rvert \geq \lvert \Phi_{h_T}^- \rvert + \deg(T)$. 
\end{lemma}

\begin{proof}  We begin with the first claim. Let $\omega$ denote the orientation of $\Gamma_h$ constructed in the proof of Lemma~\ref{lemma: equivalency sink and independence}. It can be verified that the edges which point to the right with respect to $\omega$ are precisely those which connect a vertex $\ell' \in T$ to a smaller vertex $\ell < \ell'$. Thus $\asc(\omega)$ is equal to the RHS of~\eqref{eq: characterize deg T} and hence $\deg(T) \leq$ RHS by definition. To prove equality, we claim that for any $\omega' \in \mathcal{A}(\Gamma_h)$ with $\sk(\omega')=T$ we must have $\asc(\omega') \geq \asc(\omega)$. But the fact that $\sk(\omega')=T$ implies that any edge of the form $e=\{\ell, \ell'\}$ for $\ell<\ell'$ and $\ell' \in T$ must satisfy $\tgt_\omega(e) = \ell'$, i.e., $e$ must point to the right. It follows that $\asc(\omega') \geq \asc(\omega)$ as desired.

For the second claim, recall that the edges of $\Gamma_h$ (respectively $\Gamma_{h_T}$) are in bijection with $\Phi_h^-$ (respectively $\Phi_{h_T}^-$). 
By definition of $h_T$ 
we know $|\Phi_h^-|$ is equal to $|\Phi_h^-[T]| = |\Phi_{h_T}^-|$ plus the total number of edges incident to any vertex of $T$.  By~\eqref{eq: characterize deg T}, $\deg(T)$ is less than or equal to the total number of vertices incident to a vertex in $T$ so the claim follows. 
\end{proof}

Before stating our main proposition we illustrate with our running example. 

\begin{example}\label{ex: max sink induction}  Consider the graph $\Gamma_h$ for $h=(3,4,5,5,5)$ as in Example \ref{example: smaller graph}. As already noted in Example \ref{ex: max sink set size}, in this example we have $m(\Gamma_h)=2$ and $T=\{2,5\}$ is a sink set of maximal cardinality.  In the figure below we draw all acyclic orientations $\omega \in \mathcal{A}(\Gamma_h)$ such that $\sk(\omega)=\{2,5\}$.  The sink set $T$ and incident edges are highlighted in red, and the corresponding acyclic orientation of $\Gamma_{h_T}$ is displayed to the right.
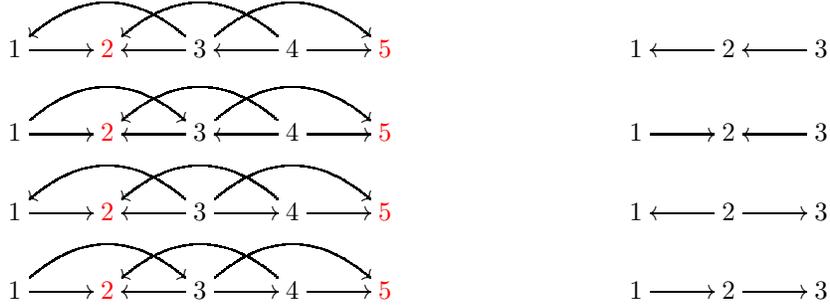
\vspace*{.15in}
\[\xymatrix{1 \ar@[red][r]   & {\color{red}2}  & 3  \ar@[red][l] \ar@/_1.5pc/[ll]  \ar@[red]@/^1.5pc/[rr] &  4  \ar[l] \ar@[red][r] \ar@[red]@/_1.5pc/[ll]   & {\color{red}5} & & &
1 & 2 \ar[l] & 3 \ar[l]
}\]
\vspace*{.05in}
\[\xymatrix{1 \ar@[red][r] \ar@/^1.5pc/[rr]  & {\color{red}2} & 3  \ar@[red][l]  \ar@[red]@/^1.5pc/[rr] &  4  \ar[l] \ar@[red][r] \ar@[red]@/_1.5pc/[ll]   & {\color{red}5} & & &
1 \ar[r] & 2 & 3 \ar[l] 
}\]
\vspace*{.05in}
\[\xymatrix{1 \ar@[red][r]   & {\color{red}2} & 3  \ar@[red][l] \ar@/_1.5pc/[ll]  \ar@[red]@/^1.5pc/[rr] \ar[r] &  4  \ar@[red][r] \ar@[red]@/_1.5pc/[ll]   & {\color{red}5} & & &
1 & 2 \ar[l] \ar[r] & 3
}\]
\vspace*{.05in}
\[\xymatrix{1 \ar@[red][r]  \ar@/^1.5pc/[rr]   & {\color{red}2}  & 3  \ar@[red][l]   \ar@[red]@/^1.5pc/[rr] \ar[r] &  4   \ar@[red][r] \ar@[red]@/_1.5pc/[ll]   & {\color{red}5} & & &
1 \ar[r] & 2 \ar[r] & 3
}\]
In this example, we have $\deg(T)=3$. 
\end{example}

In the example above, we obtain \emph{every} acyclic orientation of $\Gamma_{h_T}$ by restricting from an acyclic orientation of $\Gamma_h$ with sink set $T$. The next proposition shows that this is always the case when $T$ is a sink set of maximal cardinality.

\begin{proposition}\label{proposition: max sink set induction} 
Let $h:[n] \to [n]$ be a Hessenberg function and let $m=m(\Gamma_h)$ be the maximum sink-set size for $\Gamma_h$. 
Let $T\in \SK_m(\Gamma_h)$.  Then the restriction map 
\begin{equation}\label{eq: orientation restriction map}
	\{ \hsm \omega\in \mathcal{A}_m(\Gamma_h)  \hsm \vert \hsm \sk(\omega)=T \} \to \mathcal{A}(\Gamma_{h_T}) , \quad \omega \mapsto \omega_T 
\end{equation}
is a bijection.  Furthermore, for any $\omega \in \mathcal{A}_m(\Gamma_h)$ with $\sk(\omega)=T$ we have $\asc(\omega) = \deg(T) + \asc(\omega_T)$.
\end{proposition}

\begin{proof}

We first claim that the given restriction map is injective. To see this, recall from Lemma~\ref{lem: induction graph} that $\Gamma_{h_T} = \Gamma_h -T$ and $\Gamma_h - T$ is obtained from $\Gamma_h$ by deleting the vertices $T$ and the edges incident to $T$. Now note that any orientation $\omega$ satisfying $\sk(\omega)=T$ must have the property that $\tgt_\omega(e) = v$ for any $v \in T$ and any edge $e$ incident to $v$, by the definition of a sink. In other words, the orientation $\omega$ is determined on edges incident to $T$ by the condition $\sk(\omega)=T$. Thus the restriction map $\omega \mapsto \omega_T$, which forgets the orientations incident to $T$, is injective. Next we claim that the map is surjective. Recall that the equality $\Gamma_{h_T} = \Gamma_h - T$ identifies $V(\Gamma_{h_T})$ with $V(\Gamma_h) \setminus T$ and the edges of $\Gamma_{h_T}$ with the edges in $\Gamma_h$ which are not incident to $T$. Now suppose $\omega' \in \mathcal{A}(\Gamma_{h_T})$. Define an orientation $\omega''$ on $T$ by
orienting all the edges incident to a vertex in $T$ towards that vertex, and orienting the remaining edges using the orientation of $\omega'$.  We claim that $\omega''$ is an acyclic orientation such that $\sk(\omega'')=T$.  This would show that the restriction map is surjective since $(\omega'')_T=\omega'$ by the definition of $\omega''$.

From the construction of $\omega''$ it follows that  $T\subseteq \sk(\omega'')$.  Since $T$ is a sink set of maximal cardinality of $\Gamma_h$, it  also follows that $\sk(\omega'')=T$.  To finish the proof it suffices to show that $\omega''$ is acyclic. Suppose for a contradiction that $\omega''$ contains an oriented cycle.  Since $\omega'\in \mathcal{A}(\Gamma_{h_T})$, such a cycle must include at least one vertex $v$ of $T$.  An oriented cycle can contain a vertex only if that vertex has at least one edge oriented towards that vertex and one edge oriented outwards from that vertex. But since $v$ is a sink, this is impossible. Therefore no such oriented cycle can exist and $\omega''\in \mathcal{A}(\Gamma_h)$ as desired.

Finally, consider $\omega \in \mathcal{A}_m(\Gamma_h)$ such that $\sk(\omega)=T$ and recall that $\asc(\omega)$ counts the number of edges which point to the right in $\omega$.  By Lemma \ref{lem: deg(T) property1}, $\deg(T)$ counts those edges incident to $T$ that are oriented to the right and $\asc(\omega_T)$ counts those edges oriented to the right in the induced subgraph on the vertices $V(\Gamma_h)-T$.  Since these two sets of edges are mutually disjoint and together comprise all the edges of $\Gamma_h$ pointing to the right, $\asc(\omega) = \deg(T)+\asc(\omega_T)$ for all $\omega \in \{ \omega\in \mathcal{A}(\Gamma_h) \hsm \vert \hsm \sk(\omega)=T \}$ as desired.
\end{proof}

\begin{remark}\label{rem: maximal sink set restriction} 
In fact, the argument in the proof above shows that the restriction map~\eqref{eq: orientation restriction map} is a bijection whenever $T$ has the property that there does \emph{not} exist a $T' \in \SK(\Gamma_h)$ with $T \subsetneq T'$, i.e., $T$ is not a strict subset of any other sink set. 
This can occur even if $T$ is not of the maximal possible size, i.e., it can happen that $T$ has the above property even when $\lvert T \rvert < m(\Gamma_h)$. For example, the reader can check that $T = \{3\}$ in the ongoing example for $h=(3,4,5,5,5)$ satisfies this property despite the fact that $\lvert T \rvert = 1 < 2 = m(\Gamma_h)$, and it can be checked by hand that, for this example, ~\eqref{eq: orientation restriction map} is indeed a bijection. 
\end{remark}

\begin{example} 
Proposition~\ref{proposition: max sink set induction} is certainly false in general without the assumption that $T$ is of size $m(\Gamma_h)$. For instance, take $T = \{5\}$ in the example $h=(3,4,5,5,5)$. 
The reader can check that the following acyclic orientation $\omega'$ on $\Gamma_{h_T}$: 
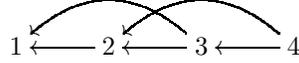
\vspace*{.15in}
\[\xymatrix{1 & {2}  \ar[l] & 3  \ar[l]   \ar@/_1.5pc/[ll]  &  4   \ar[l] \ar@/_1.5pc/[ll]   
}\]
cannot lie in the image of the map~\eqref{eq: orientation restriction map}
since any $\omega \in \mathcal{A}(\Gamma_h)$ with $\omega_T = \omega'$ must contain $1$ as a sink. 
\end{example}


\section{The height of an ideal}\label{sec: sink sets, ideals and representations}

In this section we introduce an integer invariant associated to an ideal $I$ in $\Phi^-$ called the height of $I$. A (nonempty) ideal $I$ is abelian precisely when its height is $1$. In this sense, the case of abelian ideals is the ``base case'' of an inductive argument based on height. For a Hessenberg function $h$, we then observe a connection between the height of $I_h$ and the maximum sink set size $m(\Gamma_h)$ defined in the previous section. This connection, together with the past results of Stanley, Shareshian-Wachs, and Gasharov as well as some standard representation theory, allows us to significantly simplify the process of proving Conjecture~\ref{conj:Stanley-Stembridge} in the abelian case. 

Given an ideal $I\subseteq \Phi^-$, recall that we may form an ideal in the Borel subalgebra of lower triangular matrices in $\mathfrak{gl}(n,\C)$.  Namely, let $\mathcal{I} = \oplus_{\gamma\in I} \mathfrak{g}_{\gamma}$ be the ideal of root spaces corresponding to the roots in $I$.  The \textbf{lower central series} of $\mathcal{I}$ is the sequence of ideals defined inductively by 
\[
	\mathcal{I}_1 = \mathcal{I}, \textup{ and }  \mathcal{I}_j = [\mathcal{I}, \mathcal{I}_{j-1}] \;\textup{ for all $j\geq 2$}.
\]
We define the lower central series of ideals in $\Phi^-$ analogously, by letting $I_j \subseteq \Phi^-$ be the unique ideal of negative simple roots such that $\mathcal{I}_j = \oplus_{\gamma \in I_j} \mathfrak{g}_{\gamma}$.

\begin{definition}\label{def: height} The \textbf{height of an ideal $I\subseteq \Phi^-$}, denoted $ht(I)$, is the length of its lower central series.  More concretely, 
\[
	ht(I):= \max\{ k\geq 1  \hsm \vert \hsm I_k \neq \emptyset \}.
\]
If $I=\emptyset$, then we adopt the convention that $ht(I)=0$. It is not hard to see
that the height is well-defined, i.e., the maximum always exists. 
\end{definition}

Intuitively, the height of an ideal measures how ``non-abelian'' it is.

\begin{example}  Consider the Hessenberg function $h=(2,4,4,5,5)$, with ideal  
\[
	I_h = \{ t_3-t_1, t_4-t_1, t_5-t_1, t_5-t_2, t_5-t_3 \}.
\] 
Then $I_h$ is not abelian since $(t_3-t_1)+(t_5-t_3) = t_5-t_1\in \Phi^-$.  We see that, 
\[
	I_1=I_h, \; I_2 = \{ t_5-t_1 \},\; \textup{ and } \; I_j=  \emptyset \; \textup{ for all $j\geq 3$}
\]
so $ht(I_h)=2$.
\end{example}

To connect the notions of the height of the ideal $I_h$ and maximal sink sets of $\Gamma_h$, we will use certain subsets of roots in $\Phi^-$, defined below.

\begin{definition}\label{def: root subsets}  Let $R\subseteq \Phi^-$. We say $R$ is {\bf a subset of height $k$} if if there exists integers $q_1, q_2,..., q_{k},q_{k+1}\in [n]$ such that $q_1<q_2<\cdots <q_{k}<q_{k+1}$ and $R=\{ t_{q_2}-t_{q_1}, t_{q_3}-t_{q_2},... , t_{q_{k+1}}-t_{q_{k}} \}$.  We let $\mathcal{R}_k(I)$ denote the set of all subsets of height $k$ in $I$, and define $\mathcal{R}(I) := \bigsqcup_{k\geq 0} \mathcal{R}_k(I)$. 
\end{definition}

\begin{remark}  It is straightforward to show that $R\subseteq \Phi^-$ is a subset of height $k$ if and only if there exists $w\in \Symm_n$ such that $w(R)$ is a subset of simple roots corresponding to $k$ consecutive vertices in the Dynkin diagram for $\mathfrak{gl}(n,\C)$.  
\end{remark}

If $R\subseteq \Phi^-$ is a subset of height $k$, then we may write $R=\{ \beta_1, \beta_2, ..., \beta_k \}$ where $\beta_i = t_{q_{i+1}}-t_{q_i}$ for the integers $q_1, q_2,..., q_{k},q_{k+1}\in [n]$ given in the definition above.  Therefore
\begin{eqnarray} \label{eq: root sums}
	\beta_i+\beta_{i+1} + \cdots + \beta_j \in \Phi^- \textup{ for all } 1\leq i \leq j \leq k.
\end{eqnarray}
The next lemma proves a direct relationship between $ht(I)$ and subsets of maximal height in $I$.

\begin{lemma}\label{lemma: computing ht}  Let $I$ be a nonempty ideal in $\Phi^-$.  Then $ht(I) =\max \{\hsm |R|\hsm \vert \hsm R\in \mathcal{R}(I) \}$.
\end{lemma}
\begin{proof}  Let $R$ be a subset of maximal height in $I$, so $|R|=k$ where $k= \max \{\hsm |R|\hsm \vert \hsm R\in \mathcal{R}(I) \}$. By definition, there exists $q_1, q_2,..., q_{k},q_{k+1}\in [n]$ such that $q_1<q_2<\cdots <q_{k}<q_{k+1}$ and $R=\{ \beta_1, \beta_2,..., \beta_k \}$ where $\beta_i = t_{q_{i+1}}-t_{q_i}$ for each $1\leq i \leq k$.  Using the definition of the Lie bracket and \cite[Proposition 8.5]{Humphreys1972} we get that
\[
	\mathfrak{g}_{\beta_1+\cdots+\beta_{i}} = [\mathfrak{g}_{\beta_{i}}, \mathfrak{g}_{\beta_1 + \cdots+\beta_{i-1}}]  \subseteq [\mathcal{I}, \mathcal{I}_{i-1}] = \mathcal{I}_i
\] 
for all $2 \leq i \leq k$.  In particular, $\mathfrak{g}_{\beta_1+\cdots+ \beta_k} \subseteq \mathcal{I}_k$ so $\mathcal{I}_k \neq \emptyset$ and therefore $ht(I)\geq k$. 

Seeking a contradiction, suppose $ht(I)=k' >k$.  Note that $k\geq 1$ since $I$ is nonempty, so we have $k'\geq 2$. We claim that if this is the case, then $\mathcal{R}_{k'}(I)\neq \emptyset$, contradicting the assumption that $k$ is maximal. Recall that $I_i$ denotes the $i$-th ideal in the lower central series of $I=I_1$.  By definition, if $ht(I)=k'$ then $I_{k'}\neq \emptyset$.  Let $\gamma_{k'} \in I_{k'}$ so $ \mathfrak{g}_{\gamma_{k'}} \subseteq \mathcal{I}_{k'} = [\mathcal{I}, \mathcal{I}_{k'-1}]$. By definition of the Lie bracket, there exists $\gamma_{k'-1} \in I_{k'-1}$ and $\alpha_{k'} \in I$ such that $\gamma_{k'}=\alpha_{k'}+ \gamma_{k'-1}$.  Applying the same reasoning, $\mathfrak{g}_{\gamma_{k'-1}} \subseteq \mathcal{I}_{k'-1} = [\mathcal{I}, \mathcal{I}_{k'-2}]$ so there exists $\gamma_{k'-2} \in I_{k'-2}$ and $\alpha_{k'-1} \in I$ such that $\gamma_{k'-1} = \alpha_{k'-1}+\gamma_{k'-2}$.  Continue in this way to obtain $\gamma_i \in I_i$ for each $1\leq i \leq k'$ and $\alpha_i \in I$ for each $2\leq i \leq k'$ such that
\begin{eqnarray}\label{eq: alphas}
	\gamma_{i} =  \alpha_i  + \gamma_{i-1} \textup{ for all $2\leq i \leq k'$}.
\end{eqnarray}
Set $\alpha_1=\gamma_1$ and consider the set $R'=\{ \alpha_1, \alpha_2,..., \alpha_{k'-1}, \alpha_{k'} \}$.  For each $i$ such that $1\leq i \leq k'$, $\alpha_i \in \Phi^-$ so we may write $\alpha_i = t_{a_i} - t_{b_i}$ for some $a_i, b_i \in [n]$ such that $b_i<a_{i}$.  We will prove the following claim.
\begin{itemize}
\item[] Suppose $i$ is an integer such that $2\leq i \leq k'$. Then there exists an ordering $\{ a_1', a_2', ..., a_{i}' \}$ of the set $\{  a_1, a_2, ..., a_{i} \}$ so that if $\{ b_1', b_2',..., b_{i}' \}$ is the corresponding re-ordering of the set $\{ b_1, b_2,..., b_{i} \}$, i.e., the ordering so that 
\[
\{ \alpha_1, \alpha_2,..., \alpha_i \} = \{  t_{a_1'}-t_{b_1'}, t_{a_2'}-t_{b_2'},..., t_{a_{i}'}-t_{b_{i}'} \},
\]  
then $b_{j}'=a_{j-1}'$ for all $2\leq j \leq i$.
\end{itemize}
Given this claim, consider the case in which $i=k'$ and let $q_1=b_1',\,q_2=a_1',\,...,\, q_{k'}=a_{k'-1}'$, and $q_{k'+1} = a_{k'}'$. We immediately get that $q_1<q_2<\cdots<q_{k'}<q_{k'+1}$ since $q_{j}=a_{j-1}' =b_{j}'<a_{j}'=q_{j+1}$ for all $2\leq j \leq k'$ and
\[
	R' = \{ t_{q_2}-t_{q_1}, t_{q_3}-t_{q_2},..., t_{q_{k'+1}} - t_{q_{k'}} \}.
\]
Therefore $R'\in \mathcal{R}_{k'}(I)$, so $\mathcal{R}_{k'}(I)\neq \emptyset$, which is what we wanted to show.

We now prove the claim above, using induction on $i$.  If $i=2$, consider $\alpha_1=t_{a_1}-t_{b_1}$ and $\alpha_2=t_{a_2}-t_{b_2}$.  Equation~\eqref{eq: alphas} implies that
\[
	(t_{a_1}-t_{b_1}) +(t_{a_2}-t_{b_2})  = \alpha_1+\alpha_2= \gamma_1+\alpha_2= \gamma_2\in \Phi^-.
\]
By definition of $\Phi^-$ we must have that either $a_2=b_1$ or $a_1=b_2$.  If $a_2=b_1$, set $a_1'=a_2$ and $a_2'=a_1$.  The corresponding re-ordering of $\{b_1,b_2\}$ is $b_1' = b_2$ and $b_2'=b_1$.  Then $b_2' =a_1'$ as desired.  Similarly, if $a_1=b_2$, set $a_1'=a_1$ and $a_2'=a_2$.  The corresponding re-ordering of $\{ b_1, b_2 \}$ is $b_1'=b_1$ and $b_2'=b_2$.  In this case we also conclude that $b_2'=a_1'$.   Therefore the claim holds for $i=2$.  

Now assume that for some $i$ such that $2\leq i < k'$ there exists a reordering $\{ a_1'', a_2'', ..., a_{i}'' \}$ of the set $\{  a_1, a_2, ..., a_{i} \}$ so that $b_{j}''=a_{j-1}''$ for all $2\leq j \leq i$, where $\{ b_1'', b_2'',..., b_{i}'' \}$ is the corresponding re-ordering of the set $\{ b_1, b_2,..., b_{i} \}$.  By our assumptions,
\begin{eqnarray*}
	(\alpha_1+ \alpha_2 \cdots+ \alpha_i )+ \alpha_{i+1} &=& (t_{a_1''}-t_{b_1''}+t_{a_2''}-t_{a_1''}+\cdots + t_{a_i''}-t_{a_{i-1}''}) + (t_{a_{i+1}}-t_{b_{i+1}})\\
	&=& (t_{a_{i}''} - t_{b_1''}) + (t_{a_{i+1}} - t_{b_{i+1}}).
\end{eqnarray*}
On the other hand, Equation~\eqref{eq: alphas} implies that $(\alpha_1+ \alpha_2 \cdots+ \alpha_i )+ \alpha_{i+1} = \gamma_i+ \alpha_{i+1}=\gamma_{i+1}\in \Phi^-$.  It follows that either $a_{i+1} = b_1''$ or $a_i'' = b_{i+1}$.  If $a_{i+1}=b_1''$, set $a_1' = a_{i+1}$ and $a_j' = a_{j-1}''$ for all $2\leq j \leq i+1$.  The corresponding re-ordering of the set $\{ b_1'', b_2'',..., b_{i}'' \}$ is $b_1' = b_{i+1}$ and $b_j' = b_{j-1}''$ for all $2\leq j \leq i+1$.  It follows that $b_2' = b_{1}''=a_{i+1}=a_{1}'$ and for all $3\leq j \leq i+1$ we get $b_j' = b_{j-1}'' = a_{j-2}'' = a_{j-1}'$.  Next, if $a_i''=b_{i+1}$, set $a_j'= a_{j}''$ for all $1\leq j \leq i$ and $a_{i+1}' = a_{i+1}$.  Using exactly the same reasoning as before (which is even simpler in this case since there are fewer shifts) we get that $b_{j}' = a_{j-1}'$ for all $2\leq j \leq i+1$, proving the induction hypothesis.
\end{proof}

As mentioned above, the height of an ideal gives us another characterization of abelian ideals. 

\begin{proposition}\label{lemma: abelian ht 1}  
Let $I_h\subseteq \Phi^-$ be a nonempty ideal. Then $\mathcal{R}_1(I_h)$ consists of all singleton sets of the roots in $I_h$. Moreover, $I_h$ is abelian if and only if $ht(I_h)=1$.
\end{proposition} 

\begin{proof} Suppose $I_h$ is a nonempty ideal.  Each root in $\Phi^-$ is clearly a subset of height $1$ in $I_h$. This proves the first claim. 
Now suppose $I_h$ is abelian.  This implies that there are no subsets of in $I_h$ of height $2$ or more since any such subset would satisfy the summation conditions of~\eqref{eq: root sums}.  On the other hand, $\mathcal{R}_1(I_h) \neq \emptyset$ by the above argument. Therefore $ht(I_h)=1$ by Lemma \ref{lemma: computing ht}. For the converse, suppose $ht(I_h)=1$. By definition of the lower central series, $\mathcal{I}_2 = [\mathcal{I},\mathcal{I}]=\emptyset$.  This means there cannot exist $\beta_1, \beta_2 \in I_h$ with $\beta_1+\beta_2\in \Phi^-$. Therefore $I_h$ is abelian. 
\end{proof}

\begin{remark}
The height of the ideal $I_h$ is related to the \textit{bounce number} and \textit{bounce path} of a Dyck path, which also arise in the theory of Macdonald polynomials. 
Recall that there is a standard way to associate a Dyck path $\pi_h$ to a Hessenberg function $h$ \cite[Proposition 2.7]{MbirikaTymoczko}. Given a Dyck path $\pi_h$, we can define the 
\textit{bounce path} of $\pi_h$ as in \cite[Definition 3.1]{Haglund2008}. The bounce number of $\pi_h$ is then the number of times its bounce path touches the diagonal. Indeed, it is not hard to see (for instance, using the characterization of abelian Hessenberg functions given in Remark~\ref{rem: alt-def}) that $I_h$ is abelian, i.e. the height of $I_h$  is $1$, if and only if this bounce number is $1$. 
\end{remark}

Our next goal is to make the connection between the height of $I_h$ and the maximal sink set size of $\Gamma_h$ introduced in Section~\ref{subsec:sink sets}. More specifically, we
assign to each sink set $T$ of cardinality $k \geq 2$ a unique subset of roots in $I_h$ of height $k-1$ as follows. Suppose $T \in SK_k(\Gamma_h)$ where $k \geq 2$. Write $T$ as $\{\ell_1, \ell_2, \ldots, \ell_k\} \subseteq [n]$ where $\ell_1 < \ell_2 < \cdots < \ell_k$.
We define a function $SK_k(\Gamma_h) \to \mathcal{R}_{k-1}(I_h)$ by 
\begin{equation}\label{eq: map T to RT} 
T \mapsto R_{T} := \{ \beta_i = t_{\ell_{i+1}}-t_{\ell_i} \hsm \vert \hsm 1\leq i\leq k-1  \}.
\end{equation} 
We need the following. 

\begin{proposition}\label{lem: root subsets} 
The map defined in~\eqref{eq: map T to RT} is well-defined, i.e., for any $T \in SK_k(\Gamma_h)$ with $k \geq 2$, the subset $R_T$ is an element of $\mathcal{R}_{k-1}(I_h)$. Moreover, the map~\eqref{eq: map T to RT} is a bijection. In particular, $m(\Gamma_h)=ht(I_h)+1$.
\end{proposition}

\begin{proof}  
Let $k \geq 2$ and $T \in SK_k(\Gamma_h)$, with $T = \{\ell_1, \ldots, \ell_k\}$ as above. To see that~\eqref{eq: map T to RT} is well-defined we must first show that $R_T = \{\beta_i := t_{\ell_{i+1}} - t_{\ell_i} \hsm \vert \hsm 1 \leq i \leq k-1 \}$ is a subset of $I_h$. By Lemma \ref{lemma: equivalency sink and independence}, $\ell_{i+1}> h(\ell(i))$ for all $i\in [k-1]$.  Therefore each $\beta_i =t_{\ell_{i+1}} - t_{\ell_i} \in I_h$ as desired.  The fact that $R_T$ is a subset of height $k-1$ follows directly from the definition.
We now claim~\eqref{eq: map T to RT} is a bijection. From the definition it is straightforward to see that it is injective, so it suffices to prove that it is also surjective. 
Suppose $R\in \mathcal{R}_{k-1}(I_h)$ is a set of height $k-1$.  By definition there exists $q_1, q_2,..., q_{k-1},q_k\in [n]$ such that $q_1<q_2<\cdots <q_{k-1}<q_k$ and $R= \{ t_{q_2}-t_{q_1}, t_{q_3}-t_{q_2},..., t_{q_k}-t_{q_{k-1}}\}$.  Consider the set $T=\{ q_1, q_2, ..., q_k\}$.  Since $R\subseteq I_h$ we know $q_{i+1}>h(q_i)$ for all $i\in [k-1]$.  Lemma \ref{lemma: equivalency sink and independence} now implies that $T$ is a sink set.  By definition, $R=R_{T}$ is the image of $T$ under~\eqref{eq: map T to RT} so our function is surjective. 
Our last assertion follows directly from the fact that $|\SK_k(\Gamma_h)| = |\mathcal{R}_{k-1}(I_h)|$ for all $k\geq 2$ together with Lemma~\ref{lemma: computing ht}.  This completes the proof. 
\end{proof}

The above lemma establishes, in particular, a bijection between $SK_2(\Gamma_h)$, the set of sink sets of size $2$, and $\mathcal{R}_1(I_h)$. By Proposition~\ref{lemma: abelian ht 1} we know $\mathcal{R}_1(I_h) \cong I_h$ is the set of all singleton subsets of $I_h$, so this implies that 
\[
\lvert SK_2(\Gamma_h) \rvert = \lvert I_h \rvert. 
\]
More concretely, the bijection~\eqref{eq: map T to RT} associates to a sink set $T = \{j, i\} \subseteq [n]$ with $i>j$ the subset $\{t_i - t_j \} \subseteq I_h$ of height $1$. 

\begin{example} 
Continuing Example~\ref{example: smaller graph} with the acyclic orientation drawn therein, the sink set is $\{2,5\}$ and 
the associated (singleton) subset of $I_h$ of height $1$ is $\{ t_5-t_2 \} \subseteq I_h$.
\end{example}

Our next proposition makes the connection between the maximum sink-set size and the coefficients $d_\lambda$ determining the representation $H^*(\Hess(\mathsf{S},h))$.

\begin{proposition}\label{proposition:dlambda zero} 
Let $h: [n]\to [n]$ be a Hessenberg function.  Then
\[
	m(\Gamma_h) = max\{ \hsm i \hsm \vert \hsm d_{\lambda}\neq 0 \textup{ for some } \lambda \vdash n \textup{ with $i$ parts}\}
\]
 where the $d_\lambda$ are the non-negative coefficients appearing in~\eqref{eq:decomp into Specht}.
\end{proposition} 

We first need the following lemma. 

\begin{lemma}\label{lem: P-tab and ind sets} 
If $T=\{ i_1, i_2,..., i_k \}$ is a subset of $[n]$ whose elements fill a single row in a $P_h$-tableau, then $T$ is an independent set of vertices in $\Gamma_h$.
\end{lemma}

\begin{proof} 
Suppose the elements of $T$ are listed in increasing order (in the order they appear in the row of the $P_h$-tableau).  By condition (2) in Definition \ref{def: P-tab}, we get $i_j>h(i_{j-1})$ for all $j$ such that $2\leq j \leq k$.  Lemma \ref{lemma: equivalency sink and independence} now implies that $T$ is an independent set of vertices.
\end{proof}

\begin{proof}[Proof of Proposition~\ref{proposition:dlambda zero}] 
Let 
\[
ind(\Gamma_h) := \max\{ |T| \hsm \vert\hsm T\subseteq V(\Gamma_h) \textup{ and } T \textup{ is independent} \}.
\]
By Lemma~\ref{lemma: equivalency sink and independence} it suffices to show that
\begin{eqnarray}\label{eq: ind sets and dlambda=0}
ind(\Gamma_h) = \max\{ \hsm i \hsm \vert \hsm d_{\lambda}\neq 0 \textup{ for some } \lambda \vdash n \textup{ with $i$ parts}\}.
\end{eqnarray}
Suppose $\lambda\vdash n$ is a partition of $n$ with $k$ parts such that $d_{\lambda}\neq 0$.  By Theorem \ref{thm: irreducible coefficients} there exists at least one $P_h$-tableau of shape $\lambda^{\vee}$.  Since $\lambda$ has $k$ parts, $\lambda^{\vee}$ has $k$ boxes in the first row.  By Lemma \ref{lem: P-tab and ind sets} the entries in the first row of this $P_h$-tableau form an independent set of vertices in $\Gamma_h$.  Therefore the LHS of~\eqref{eq: ind sets and dlambda=0} is greater than or equal to the RHS.

To prove the opposite inequality, let $T = \{\ell_1, \ell_2, \ldots, \ell_m\}$, where $\ell_1 < \ell_2 < \cdots < \ell_m$ be an independent subset of vertices in $\Gamma_h$ of maximal size.  By Lemma~\ref{lemma: equivalency sink and independence} we know $\ell_{i+1} > h(\ell_{i})$ for all $i\in [m-1]$. Consider the partition $\lambda^\vee = (m, 1, \ldots, 1)$ of $n$ of ``hook shape'' with first row containing $m$ boxes and all other rows containing only one box. Also consider the filling of the Young diagram of shape $\lambda^\vee$ given by filling the top row with $\ell_1, \ldots, \ell_m$ in increasing order, and filling the remaining boxes by $[n] \setminus \{\ell_1, \ldots, \ell_m\}$ in increasing order from top to bottom. We claim that this is a $P_h$-tableau of shape $\lambda^\vee$. By construction, conditions (1) and (2) of Definition \ref{def: P-tab} are already met, so we have only to check condition (3). Note that, for a pair $i$ and $j$ with $i$ appearing immediately below $j$, the condition (3) -- namely, that $j \leq h(i)$ -- holds automatically if $j < i$ (since $h(i) \geq i$ by definition of Hessenberg functions). Since $\lambda^\vee$ is of hook shape, the only places where condition (3) must be checked is along the leftmost column of $\lambda^\vee$, and since by construction the filling contains entries which increase from top to bottom starting at the second row, the argument above implies that the only remaining place where condition (3) must be checked is for the entry $\ell_1$ in the top-left box of $\lambda^\vee$ and the entry $\ell' := \mathrm{min}([n] \setminus \{\ell_1,\ldots,\ell_m\})$ in the unique box in the second row, for which we must show that $\ell_1 \leq h(\ell')$. Suppose for a contradiction that $\ell_1 > h(\ell')$ (and hence $\ell' < \ell_1$). This implies there is no edge connecting $\ell'$ with $\ell_1$ for any $i$, $1 \leq i \leq m$. Thus $T' = \{\ell', \ell_1, \ldots, \ell_m\}$ is a sink set of $\Gamma_h$ by Lemma \ref{lemma: equivalency sink and independence}.  Since $|T'|=m+1$, this contradicts the maximality of $m=|T|$. Thus $\ell_1 \leq h(\ell')$ and hence the above filling is indeed a $P_h$-tableau.  By construction of the $\lambda^\vee$, its dual partition $\lambda$ has $m \geq k+1$ parts proving that the RHS of Equation~\eqref{eq: ind sets and dlambda=0} is greater than or equal to the LHS.
\end{proof}

The following is now straightforward. In the case that $I_h$ is abelian, the corresponding restriction on the partitions that can appear in the RHS of~\eqref{eq: decomp into Mlambda} is quite striking. 

\begin{corollary}\label{corollary: max sink set gives bound on lambda}
Let $h: [n] \to [n]$ be a Hessenberg function and let $c_{\lambda}$ and $c_{\lambda,i}$ be the coefficients appearing in~\eqref{eq: decomp into Mlambda}. Then $c_{\lambda}=c_{\lambda,i}=0$ for all $\lambda \vdash n$ with more than $m(\Gamma_h) = ht(I_h)+1$ parts and for all $i \geq 0$.  In particular, if $I_h$ is abelian, then $c_{\lambda}=c_{\lambda,i}=0$ for all $\lambda \vdash n$ with more than $2$ parts and for all $i \geq 0$. 
\end{corollary} 

\begin{proof} 
It follows from Proposition~\ref{proposition:dlambda zero} that, under the hypotheses, $d_{\lambda}=0$ for all $\lambda$ with more than $m(\Gamma_h)$ parts. Now apply Lemma~\ref{lemma: dlambda zero implies all other zero} to $H^*(\Hess(\mathsf{S},h))$. For the abelian case, if $I_h$ is non-empty then this follows from Propositions~\ref{lemma: abelian ht 1} and~\ref{lem: root subsets}. If $I_h$ is empty, then $h=(n,n,\ldots,n)$ and $\Hess(\mathsf{S},h) = \Flags(\C^n)$. The corresponding graph $\Gamma_h$ has the property that every vertex is connected to every other vertex, implying that $m(\Gamma_h)=1$ and hence $c_{\lambda}=c_{\lambda,i}=0$ for all $\lambda$ with $2$ or more parts and all $i\geq 0$. Hence the conclusion holds in this case as well. 
\end{proof}

We have already indicated that our strategy for proving Theorem~\ref{theorem:main} is by induction, using the association of $\Gamma_h$ with $\Gamma_{h_T} = \Gamma_h -T$ for a sink set $T$ as in Lemma~\ref{lem: induction graph}.  Let $\mathsf{S}_T$ denote any regular semisimple element in $\mathfrak{gl}(n-|T|, \C)$.
 It will be useful for us to know that vanishing conditions on the coefficients of the dot action representation on $H^*(\Hess(\mathsf{S},h))$ imply vanishing conditions for the analogous coefficients of $H^*(\Hess(\mathsf{S}_T,h_T))$. To state the lemma precisely we introduce some terminology.  For each partition $\mu$ of $n-|T|$  we define $c_{\mu, i}^T$ (respectively $d_{\mu,i}^T$) to be the coefficient of $M^{\mu}$ (respectively $\mathcal{S}^{\mu}$) for the decomposition of the $\Symm_{n-|T|}$-representation $H^{2i}(\Hess(\mathsf{S}_T, h_T))$ in $\mathcal{R}ep(\Symm_{n-|T|})$.

\begin{lemma}\label{lem: maximal sink set induction} 
Let $h: [n] \to [n]$ be a Hessenberg function, and let $T \in SK_k(\Gamma_h)$ be a sink set of $\Gamma_h$. Then 
\begin{enumerate} 
\item $m(\Gamma_{h_T}) \leq m(\Gamma_h)$, and 
\item $c_\mu^T = 0$ and $c_{\mu,i}^T = 0$ for all $\mu \vdash (n-\lvert T \rvert)$ with more than $m(\Gamma_h)$ parts and all $i \geq 0$. 
\end{enumerate} 
In particular, if $I_h$ is abelian and $T \in SK_2(\Gamma_h)$, then $c^T_{\lambda,i}=0$ for all $\lambda$ with more than $2$ parts and all $i \geq 0$. 
\end{lemma}

\begin{proof} 
We begin with the first claim. Since $\Gamma_{h_T}$ is by definition an induced subgraph of $\Gamma_h$, if there exists an independent set of vertices in $\Gamma_{h_T}$, then the corresponding subset is also independent in $\Gamma_h$. 
It follows from Lemma~\ref{lemma: equivalency sink and independence} that $m(\Gamma_{h_T}) = \max \{ \lvert T' \rvert \hsm \vert \hsm T'\subseteq V(\Gamma_{h_T}) \textup{ is independent in } \Gamma_{h_T} \}$ so we conclude
$m(\Gamma_{h_T}) \leq m(\Gamma_h)$ as desired. The second statement follows from the first by Corollary~\ref{corollary: max sink set gives bound on lambda}. 
\end{proof}


\section{An inductive formula for the coefficients of the dot action}\label{section: inductive formula}

In this section, we state our main theorem, which gives an inductive formula which, in the case when $I_h$ is abelian, expresses Tymoczko's ``dot action'' representation on $H^{2i}(\Hess(\mathsf{S},h))$ as a combination of trivial representations together with a sum of tabloid representations with coefficients associated to \emph{smaller} Hessenberg varieties in $\Flags(\C^{n-2})$. To illustrate this result, we give an extended example when $n=6$. We also state three technical results -- one (simple) lemma and two propositions -- and give a proof of Theorem~\ref{theorem:induction} based on these three results.  Each of the Propositions below are themselves inductive formulas, and are of interest in their own right. The proofs of the two propositions are postponed to Section~\ref{section: the proofs}.

\begin{theorem}\label{theorem:induction} 
Let $n$ be a positive integer and $n \geq 3$. Let $h: [n] \to [n]$ be a Hessenberg function such that $I_h$ is abelian and $i \geq 0$ be a non-negative integer. In the representation ring $\mathcal{R}ep(\Symm_n)$ we have the equality 
\begin{equation}\label{eq:main inductive step}
H^{2i}(\Hess(\mathsf{S}, h)) =  c_{(n),i} M^{(n)} +  \sum_{T \in SK_2(\Gamma_h)} \left( \sum_{\substack{\mu \,\vdash (n-2)\\ \mu=(\mu_1,\mu_2)}} c_{\mu, i-\deg(T)}^T M^{(\mu_1+1,\mu_2+1)} \right). 
\end{equation}
\end{theorem} 

We first illustrate the theorem via an extended example. 

\begin{example} 
Let $n=6$ and $h=(3,4,5,6,6,6)$ as in Example~\ref{ex: pictures}. Then $I_h$ is abelian, and $|I_h| = 6$.  Thus, there are six maximum dimensional sink sets in $\SK_2(\Gamma_h)$. The graphs below show the acyclic orientation $\omega\in \mathcal{A}_2(\Gamma_h)$ such that $\asc(\omega)=\deg(T)$ for each $T\in \SK_2(\Gamma_h)$.  In each case, the sink set $T$ and incident edges are highlighted in red, we display the corresponding acyclic orientation of $\Gamma_h-T\cong \Gamma_{h_T}$ in the center column, and the Hessenberg function $h_T$ in the right column.
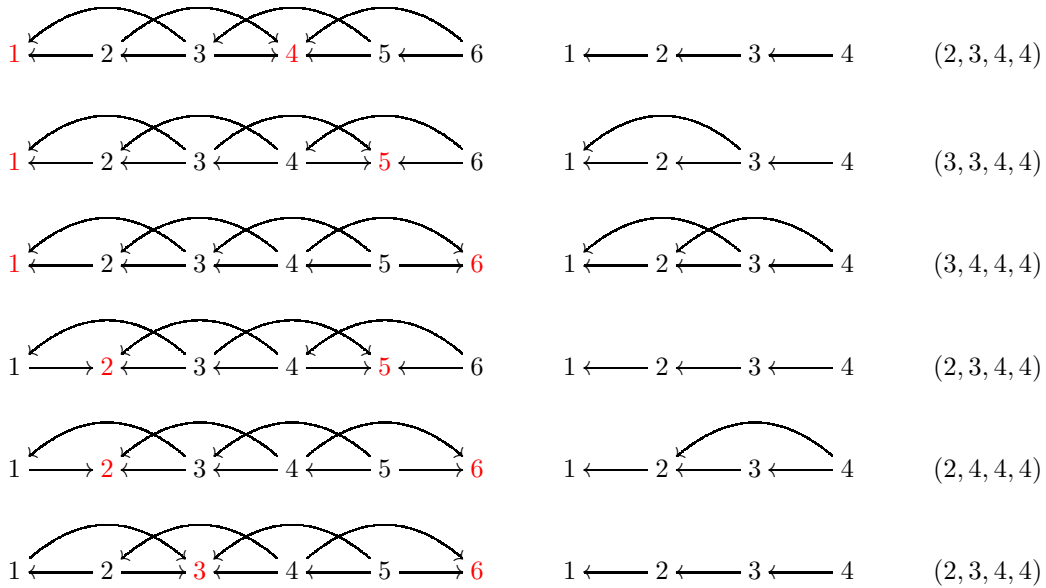
\vspace*{.15in}
\[\xymatrix{ {\color{red}1}    & {2} \ar@[red][l] \ar@[red]@/^1.5pc/[rr]  & 3 \ar@[red][r] \ar[l] \ar@[red]@/_1.5pc/[ll]  &  {\color{red}4}    & {5} \ar@/_1.5pc/[ll] \ar@[red][l]  & {6} \ar[l] \ar@[red]@/_1.5pc/[ll] &
1 & 2 \ar[l] & 3 \ar[l] & 4 \ar[l] & (2,3,4,4)
}\]
\vspace*{.05in}
\[\xymatrix{ {\color{red}1}    & {2} \ar@[red][l]  & 3 \ar[l] \ar@[red]@/_1.5pc/[ll]  \ar@[red]@/^1.5pc/[rr] &  {4} \ar@/_1.5pc/[ll]  \ar@[red][r] \ar[l]   & {\color{red}5}  & {6} \ar@[red][l] \ar@/_1.5pc/[ll] & 
1 & 2 \ar[l] & 3 \ar[l] \ar@/_1.5pc/[ll] & 4 \ar[l]  & (3,3,4,4)
}\]
\vspace*{.05in}
\[\xymatrix{ {\color{red}1}    & {2} \ar@[red][l]  & 3 \ar[l] \ar@[red]@/_1.5pc/[ll]  &  {4} \ar[l] \ar@/_1.5pc/[ll]   \ar@[red]@/^1.5pc/[rr]   & { 5} \ar[l] \ar@/_1.5pc/[ll] \ar@[red][r]  & {\color{red}6} & 
1 & 2 \ar[l] & 3 \ar@/_1.5pc/[ll] \ar[l]  & 4  \ar@/_1.5pc/[ll] \ar[l]  & (3,4,4,4)
}\]
\vspace*{.05in}
\[\xymatrix{ {1}    \ar@[red][r]   & {\color{red}2}  & 3 \ar@/_1.5pc/[ll] \ar@[red][l]  \ar@[red]@/^1.5pc/[rr] &  {4} \ar@[red]@/_1.5pc/[ll]  \ar@[red][r] \ar[l]   & {\color{red}5}  & {6} \ar@[red][l] \ar@/_1.5pc/[ll] & 
1 & 2 \ar[l] & 3 \ar[l] & 4 \ar[l]  & (2,3,4,4)
}\]
\vspace*{.05in}
\[\xymatrix{ {1}  \ar@[red][r]  & {\color{red}2}   & 3 \ar@[red][l] \ar@/_1.5pc/[ll]  &  {4} \ar[l] \ar@[red]@/_1.5pc/[ll]   \ar@[red]@/^1.5pc/[rr]   & { 5} \ar[l] \ar@/_1.5pc/[ll] \ar@[red][r]  & {\color{red}6} & 
1 & 2 \ar[l] & 3 \ar[l]  & 4  \ar@/_1.5pc/[ll] \ar[l]  & (2,4,4,4)
}\]
\vspace*{.05in}
\[\xymatrix{ {1}   \ar@[red]@/^1.5pc/[rr]  & {2} \ar@[red][r]  \ar[l]  & {\color{red}3}    &  {4} \ar@[red][l] \ar@/_1.5pc/[ll]   \ar@[red]@/^1.5pc/[rr]   & { 5} \ar[l] \ar@[red]@/_1.5pc/[ll] \ar@[red][r]  & {\color{red}6} & 
1 & 2 \ar[l] & 3 \ar[l]  & 4  \ar[l] &  (2,3,4,4)
}\]
Since the graphs are symmetric, $\Gamma\setminus\{1,5\} \cong \Gamma\setminus\{2,6\}$ and $\Hess(\mathsf{S}',(3,3,4,4))\cong \Hess(\mathsf{S}', (2,4,4,4))$ where $\mathsf{S}'\in \mathfrak{gl}(n-2,\C)$ is a regular semisimple element.  The representation $H^*(\Hess(\mathsf{S}', h_T))$ for each $T\in \SK_2(\Gamma_h)$ is as shown in the table below.  The reader can confirm this using the graded version of Theorem~\ref{thm: irreducible coefficients}, namely \cite[Theorem 6.3]{ShareshianWachs2016}, together with \eqref{eq:Mlambda in terms of Slambda}.
\begin{eqnarray*}
\begin{array}{c|l|l|l}
\textup{Hessenberg function $h_T$:} & (2,3,4,4) & (3, 3, 4, 4) & (3,4,4,4) \\ \hline
H^0(\Hess(\mathsf{S}', h_T)) & M^{(4)} &  M^{(4)} &  M^{(4)}\\ 
H^2(\Hess(\mathsf{S}', h_T)) & M^{(4)} + M^{(3,1)} + M^{(2,2)} &  2M^{(4)}+ M^{(3,1)} &  3M^{(4)}\\ 
H^4(\Hess(\mathsf{S}', h_T)) & M^{(4)}+ M^{(3,1)} + M^{(2,2)}  &  2M^{(4)}+ 2 M^{(3,1)} &  4M^{(4)} + M^{(3,1)} \\ 
H^6(\Hess(\mathsf{S}', h_T))& M^{(4)} &  2M^{(4)}+ M^{(3,1)} &  4M^{(4)} + M^{(3,1)}\\ 
H^8(\Hess(\mathsf{S}', h_T))& \empty &  M^{(4)} &  3M^{(4)}\\ 
H^{10}(\Hess(\mathsf{S}', h_T)) & \empty  &  \empty  &  M^{(4)}
\end{array}
\end{eqnarray*}
Next we see that
\[
\deg(\{1,4\}) = \deg(\{ 1,5 \}) = \deg(\{1,6\}) = 2,\; \deg(\{ 2,5 \})=\deg(\{ 2,6 \}) = 3, \textup{ and } \deg(\{3,6\}) = 4.
\]
from the graphs above.  We now have all the information we need to compute $H^*(\Hess(\mathsf{S}, h))$ in all degrees as the shifted sum of $M^{(\mu_1+1, \mu_2+1)}$'s where $M^{(\mu_1,\mu_2)}$ appears in the representations above.  The next two tables show how to shift these representations using $\deg(T)$ in order to obtain $H^*(\Hess(\mathsf{S},h))$.

\begin{eqnarray*}
\begin{array}{c | c | c | c  }
\textup{Sink set $T\in \SK_2(\Gamma_h)$:} & \{ 1, 4\} & \{ 1, 5\} & \{ 1,6 \} \\ \hline
H^2(\Hess(\mathsf{S},h)) & & & \\
H^4(\Hess(\mathsf{S},h)) & M^{(5,1)} &  M^{(5,1)} &  M^{(5,1)} \\
H^6(\Hess(\mathsf{S},h)) & M^{(5,1)} + M^{(4,2)} + M^{(3,3)} &  2M^{(5,1)}+ M^{(4,2)} &  3M^{(5,1)} \\
H^8(\Hess(\mathsf{S},h)) & M^{(5,1)}+ M^{(4,2)} + M^{(3,3)}  &  2M^{(5,1)}+ 2 M^{(4,2)} &  4M^{(5,1)} + M^{(4,2)} \\
H^{10}(\Hess(\mathsf{S},h)) & M^{(5,1)} &  2M^{(5,1)}+ M^{(4,2)} &  4M^{(5,1)} + M^{(4,2)} \\
H^{12}(\Hess(\mathsf{S},h)) &  \empty &  M^{(5,1)} &  3M^{(5,1)} \\
H^{14}(\Hess(\mathsf{S},h))  & \empty  &  \empty  &  M^{(5,1)} \\
H^{16}(\Hess(\mathsf{S},h)) & & & \\ \hline \hline
\textup{Sink set $T\in \SK_2(\Gamma_h)$:} & \{ 2, 5 \} & \{ 3,6 \} & \{ 2,6 \}\\ \hline
H^4(\Hess(\mathsf{S},h)) & & & \\
H^6(\Hess(\mathsf{S},h)) & M^{(5,1)} &   &  M^{(5,1)} \\
H^8(\Hess(\mathsf{S},h)) & M^{(5,1)} + M^{(4,2)} + M^{(3,3)} &   M^{(5,1)} &  2M^{(5,1)}+M^{(4,2)} \\
H^{10}(\Hess(\mathsf{S},h)) & M^{(5,1)}+ M^{(4,2)} + M^{(3,3)}  &  M^{(5,1)} + M^{(4,2)} + M^{(3,3)}  &  2M^{(5,1)} + 2M^{(4,2)} \\
H^{12}(\Hess(\mathsf{S},h)) & M^{(5,1)} &  M^{(5,1)} + M^{(4,2)} + M^{(3,3)}  &  2M^{(5,1)} + M^{(4,2)} \\
H^{14}(\Hess(\mathsf{S},h)) &  \empty &  M^{(5,1)} &  M^{(5,1)} \\
H^{16}(\Hess(\mathsf{S},h))  & \empty  &  \empty  &  \empty \\
\end{array}
\end{eqnarray*}
For example, we get,
\[
H^{8}(\Hess(\mathsf{S},h)) = c_{(6),4} M^{(6)}+11 M^{(5,1)} + 6 M^{(4,2)} + 2M^{(3,3)}. 
\]

\end{example}

 As mentioned above, we prove Theorem~\ref{theorem:induction} using the following three results, recorded as a Lemma and two Propositions, each of which are themselves inductive formulas.  Indeed, Lemma~\ref{lem: two-part induction} expresses the number $N_{\lambda', \lambda}$ associated to two partitions of $n$ in terms of the same value associated to two partitions of the smaller integer $n-2$. Proposition~\ref{prop: reg step} gives a formula for the Poincar\'e polynomial of $\Hess(\mathsf{N},h) \subseteq \Flags(\C^n)$ in terms of Poincar\'e polynomials of regular nilpotent Hessenberg varieties in $\Flags(\C^{n-2})$, and Proposition~\ref{prop:induction step}
 is of a similar flavor.  
 Throughout the remainder of this section and the next, for a positive integer $n \geq 3$, we let $\mathsf{N}'$ and $\mathsf{S}'$ denote the regular nilpotent and regular semisimple elements, respectively, in $\g \ell(n-2,\C)$.

\begin{lemma}\label{lem: two-part induction} 
Let $n$ be a positive integer and $n \geq 3$. Let $\mu=(\mu_1, \mu_2)$ and $\mu' = (\mu'_1, \mu'_2)$ be any partitions of $n-2$ with at most $2$ parts. Then 
\begin{equation}\label{eq: N equality} 
\dim \left(M^{(\mu_1+1, \mu_2+1)}\right)^{\Symm_{(\mu_1'+1, \mu_2'+1)}} = 
\dim \left(M^{\mu}\right)^{\Symm_{\mu'}} +1.
\end{equation}
\end{lemma} 

\begin{proof} 
Recall from~\eqref{eq:def Nlambdanu} and the related discussion that $N_{\mu, \mu'} = \dim (M^{\mu})^{\Symm_{\mu'}}$
 and the matrix $N=(N_{\mu, \mu'})$ is symmetric. To prove the lemma it clearly suffices to prove the formula 
 \begin{equation}\label{eq: formula for N}
 N_{(a,b),(c,d)} = b+1
 \end{equation}
for any $a,b,c,d \geq 0$ integers with $a+b=c+d=k$ for a fixed positive integer $k$ and $a \geq c$, since this would imply that the LHS and RHS of~\eqref{eq: N equality} are equal, thus proving~\eqref{eq: N equality}. To prove~\eqref{eq: formula for N}, we recall that in general $N_{\mu,\mu'}$ is the number of matrices $A=(a_{ij})$ with $a_{ij} \geq 0$ integers such that $\row(A)=\mu$ and $\col(A)=\mu'$ (see \cite[Corollary 7.12.3]{Stanley-EnumCombVol2}), where $\row(A)$ is the vector obtained from a matrix by taking row-wise sums, and $\col(A)$ is the vector obtained by taking column-wise sums, 
\[
	\row(A):=(r_1, r_2,...) \textup{ where } r_i=\sum_j a_{ij} \textup{ and }
	\col(A):=(c_1, c_2,... ) \textup{ where } c_j=\sum_i a_{ij}.
\]
In our case, since both $(a,b)$ and $(c,d)$ have only $2$ parts, this is equal to the number of matrices 
\[
\begin{pmatrix} \alpha & \beta \\ \gamma & \delta \end{pmatrix} \textup{ such that }\alpha+\beta=a, \gamma + \delta = b, \alpha + \gamma=c, \textup{ and } \beta+\delta = d. 
\]
It is both straightforward to see and well-known that this is the number of ways to fill a Young diagram of shape $(a,b)$ with $c$ many $1$'s and $d$ many $2$'s, such that the rows are weakly increasing. Since there are only $2$ rows in this Young diagram, the filling is completely determined by the number of boxes in the 2nd row which contain a $1$. Since $a \geq c$, it follows that $d\geq b$, and this number of boxes is between $0$ and $b$.  Thus there are precisely $b+1$ many such fillings, proving~\eqref{eq: formula for N} as desired. 
\end{proof} 

\begin{remark}  It is also a well known fact that $\dim(M^{\lambda})^{\mu} = |\Symm_{\mu}\backslash\Symm_n/\Symm_{\lambda}|$.  When both $\lambda$ and $\mu$ have two parts, there is a set of coset representatives for $\Symm_{\mu}\backslash\Symm_n/\Symm_{\lambda}$ known as \emph{bigrassmanian permutations}. These elements play an important role in the combinatorial properties of the symmetric group and in Schubert calculus.
\end{remark}

For any $\mathsf{X} \in \mathfrak{gl}(n,\C)$ and Hessenberg function $h:[n] \to [n]$, we denote by $P(\Hess(\mathsf{X},h),t)$ the Poincar\'e polynomial (with variable $t$) associated to the Hessenberg variety $\Hess(\mathsf{X},h)$. For the varieties considered in this paper, all Poincar\'e polynomials are concentrated in even degrees. 

\begin{proposition}\label{prop: reg step}  Let $n$ be a positive integer and $n\geq 3$. Let $h: [n]\to [n]$ be a Hessenberg function such that $I_h$ is abelian.
  Let $\mathsf{N}$ be a regular nilpotent element of $\mathfrak{gl}(n,\C)$ and $\mathsf{N}'$ be a regular nilpotent element of $\mathfrak{gl}(n-2, \C)$.  Then 
\[
P(\Hess(\mathsf{N},h), t) = \sum_{i=0}^{|\Phi_h^-|} c_{(n),i} \,t^{2i} +\sum_{T\in \SK_2(\Gamma_h)} t^{2\deg(T)} P(\Hess(\mathsf{N}', h_T), t).
\]
In particular, the $2i$-th Betti number of $\Hess(\mathsf{N},h)$ satisfies 
\[
\dim H^{2i}(\Hess(\mathsf{N},h)) = c_{(n),i}+\sum_{T\in \SK_2(\Gamma_h)} \dim H^{2i-2\deg(T)}(\Hess(\mathsf{N}',h_T)).
\]
\end{proposition}

\begin{proposition}\label{prop:induction step} Let $n$ be a positive integer, $n\geq 3$. Let $h: [n]\to [n]$ be a Hessenberg function such that $I_h$ is abelian.  Let $\mathsf{X}_{\nu}$ be the regular element of $\mathfrak{gl}(n,\C)$ associated to $\nu = (\mu_1+1, \mu_2+1) \vdash n$ and $\mathsf{X}_{\mu}$ be a regular element of $\mathfrak{gl}(n-2, \C)$ associated to $\mu = (\mu_1, \mu_2) \vdash (n-2)$.  Then
\begin{equation}\label{eq:induction equation} 
P(\Hess(\mathsf{X}_{\nu}, h), t) = P(\Hess(\mathsf{N}, h), t) + \sum_{T\in \SK_2(\Gamma_h)} t^{2\deg(T)} P(\Hess(\mathsf{X}_{\mu}, h_T), t).
\end{equation} 
In particular, the $2i$-th Betti number of $\Hess(\mathsf{X}_{\nu}, h)$ satisfies 
\[ 
\dim H^{2i}(\Hess(\mathsf{X}_{\nu},h)) = 
\dim H^{2i}(\Hess(\mathsf{N}, h)) + \sum_{T\in \SK_2(\Gamma_h)} \dim H^{2i-2\deg(T)} (\Hess(\mathsf{X}_{\mu}, h_T)).
\]
\end{proposition}

 Below, we give a proof of Theorem \ref{theorem:induction} using the  three results above. 
The basic idea of the proof is as follows. A priori, the assertion of Theorem~\ref{theorem:induction} is an equality in the representation ring $\mathcal{R}ep(\Symm_n)$. We first reduce this problem to a collection of equalities of integers by taking $\Symm_\nu$-invariants for varying $\nu \vdash n$ and using Proposition~\ref{prop: Brosnan Chow}. Next, we repeatedly use Brosnan and Chow's Theorem~\ref{thm: BrosnanChow main thm} to relate these $\Symm_\nu$-invariant subspaces to the Betti numbers of other regular Hessenberg varieties. In this manner, the problem is reduced to an induction on the Poincar{\'e} polynomials of regular Hessenberg varieties.

\begin{proof}[Proof of Theorem \ref{theorem:induction}]  

Since $h$ is an abelian Hessenberg function,  we know from Corollary~\ref{corollary: max sink set gives bound on lambda} that $c_{\lambda,i}=0$ for all $\lambda \vdash n$ with $3$ or more parts. In other words, by the abelian assumption, we know 
\[
H^{2i}(\Hess(\mathsf{S},h)) = \sum_{\substack{\lambda \vdash n \\ \lambda \textup{ has at most $2$ parts} }} c_{\lambda,i} M^\lambda.
\]
Therefore the LHS of~\eqref{theorem:induction} can be written as a linear combination of $M^{\lambda}$'s for $\lambda$ with at most $2$ parts. By inspection, the same is true of the RHS of~\eqref{theorem:induction}. An application of Lemma~\ref{lemma: restrict to submatrix} implies that in order to prove~\eqref{theorem:induction} it suffices to prove the equality
\begin{eqnarray*}\label{eq: nu-fixed pts}\begin{split}
	&\dim (H^{2i}(\Hess(\mathsf{S},h)))^{\Symm_\nu}  = c_{(n),i} \dim (M^{(n)})^{\Symm_\nu} +\\ 
	& \quad\quad\quad\quad\quad\quad\quad\quad\quad\sum_{T\in \SK_2(\Gamma_h)} \left( \sum_{\substack{\mu\vdash (n-2)\\ \mu=(\mu_1,\mu_2)}} c_{\mu, i-\deg(T)}^T \dim (M^{(\mu_1+1, \mu_2+1)})^{\Symm_{\nu}}\right)
\end{split}\end{eqnarray*}
for all $\nu \vdash n$ with at most $2$ parts.

Note that since $M^{(n)}$ is the trivial $1$-dimensional $\Symm_n$-representation, we have $(M^{(n)})^{\Symm_\nu} = M^{(n)}$ for all $\nu \vdash n$, and in particular, $\dim (M^{(n)})^{\Symm_\nu} = 1$ for all $\nu \vdash n$. We also know from Theorem~\ref{thm: BrosnanChow main thm} that 
\[
\dim (H^{2i}(\Hess(\mathsf{S},h)))^{\Symm_\nu} = \dim H^{2i}(\Hess(\mathsf{X}_{\nu},h))
\]
where $\mathsf{X}_\nu$ denotes a regular element of $\mathfrak{gl}(n,\C)$ with Jordan block sizes given by $\nu \vdash n$. It follows that it suffices to prove 
\begin{eqnarray}\label{eq: nu-fixed pts2}
	&&\dim H^{2i} (\Hess(\mathsf{X}_{\nu} , h)) = c_{(n),i} + \sum_{T\in \SK_2(\Gamma_h)} \left(\sum_{\substack{\mu \vdash (n-2)\\ \mu=(\mu_1,\mu_2)}} c_{\mu, i-\deg(T)}^T \dim (M^{(\mu_1+1, \mu_2+1)})^{\Symm_{\nu}}\right)
\end{eqnarray}
for all $\nu \vdash n$ with at most two parts.

To prove~\eqref{eq: nu-fixed pts2} we take cases. 
First, we consider~\eqref{eq: nu-fixed pts2} for the unique case in which $\nu \vdash n$ has only one part, namely $\nu=(n)$. In this case $\Symm_{\nu}=\Symm_n$ and $X_{\nu}$ may be taken to be the regular nilpotent element $\mathsf{N}$ of $\mathfrak{gl}(n,\C)$. Recall also that $\dim(M^{(\mu_1+1,\mu_2+1)})^{\Symm_n} =1$ since the multiplicity of the trivial representation in any $M^{(\mu_1+1,\mu_2+1)}$ is $1$.  Thus, in this case the RHS of~\eqref{eq: nu-fixed pts2} is
\begin{eqnarray}\label{eq: (n)-case}
	c_{(n),i} + \sum_{T\in \SK_2(\Gamma_h)}\; \sum_{\substack{\mu \vdash (n-2)\\ \mu=(\mu_1,\mu_2)}} c_{\mu, i-\deg(T)}^T. 
\end{eqnarray}
In order to simplify~\eqref{eq: nu-fixed pts2} further we recall that the coefficients $c^T_{\mu, i-\deg(T)}$ appearing there are associated to $H^{2i-2\deg(T)}(\Hess(\mathsf{S}', h_T))$ by the equality 
\[
H^{2i-2\deg(T)}(\Hess(\mathsf{S}', h_T)) = \sum_{\mu \vdash (n-2)} c^T_{\mu, i-\deg(T)} M^{\mu}.
\]
Now the assumption that $I_h$ is abelian implies $c^T_{\mu, i-\deg(T)} = 0$ for all $\mu \,\vdash (n-2)$ with more than $2$ parts by Lemma~\ref{lem: maximal sink set induction}. Thus we have $H^{2i-2\deg(T)}(\Hess(\mathsf{S}', h_T)) = \sum_{\mu=(\mu_1,\mu_2) \vdash (n-2)} c^T_{\mu, i-\deg(T)} M^\mu$
and taking $\Symm_{n-2}$-invariants we obtain 
\begin{eqnarray*}
\dim H^{2i-2\deg(T)}(\Hess(\mathsf{N}', h_T)) = \dim (H^{2i-2\deg(T)}(\Hess(\mathsf{S}', h_T)))^{\Symm_{n-2}}
= \sum_{\substack{\mu \vdash (n-2)\\ \mu=(\mu_1,\mu_2)}} c^T_{\mu, i-\deg(T)}
\end{eqnarray*}
where the first equality follows from Theorem~\ref{thm: BrosnanChow main thm}. 
Therefore~\eqref{eq: (n)-case} can be rewritten as  
\begin{equation*}\label{eq: nu-fixed pts trivial piece}
c_{(n),i} + \sum_{T \in SK_2(\Gamma_h)} \dim H^{2i-2\deg(T)}(\Hess(\mathsf{N}', h_T))
\end{equation*}
and now~\eqref{eq: nu-fixed pts2} follows for the case $\nu=(n)$ and $\mathsf{X}_\nu=\mathsf{N}$ by  Proposition~\ref{prop: reg step}.

Next, we consider the case in which $\nu = (\mu'_1+1, \mu_2'+1)\vdash n$ for some $\mu'=(\mu_1', \mu_2') \vdash (n-2)$, i.e. the case in which $\nu$ has exactly two parts.  Using an argument similar to the above, the RHS of~\eqref{eq: nu-fixed pts} for $\nu=(\mu_1'+1,\mu_2'+1) \vdash n$ can be expressed as 
\begin{eqnarray*}
        && c_{(n),i} + \sum_{T\in \SK_2(\Gamma_h)}\; \sum_{\substack{\mu\vdash (n-2)\\ \mu=(\mu_1,\mu_2)}} c_{\mu, i-\deg(T)}^T 
        \dim(M^{(\mu_1+1, \mu_2+1)})^{\Symm_{(\mu_1'+1,\mu_2'+1)}} \\ 
       && = c_{(n),i} + \sum_{T\in \SK_2(\Gamma_h)}\; \sum_{\substack{\mu\vdash (n-2)\\ \mu=(\mu_1,\mu_2)}} c_{\mu, i-\deg(T)}^T \left( \dim (M^{\mu})^{\Symm_{\mu'}} + 1 \right) \;\; \textup{ by Lemma~\ref{lem: two-part induction}  } \\ 
	&&= c_{(n),i} + \sum_{T\in \SK_2(\Gamma_h)}\; \sum_{\substack{\mu\vdash (n-2)\\ \mu=(\mu_1,\mu_2)}} c_{\mu, i-\deg(T)}^T + \sum_{T\in \SK_2(\Gamma_h)}\; \sum_{\substack{\mu\vdash (n-2)\\ \mu=(\mu_1,\mu_2)}} c_{\mu, i-\deg(T)}^T \dim (M^{\mu})^{\Symm_{\mu'}}\\
	&& = \dim(H^{2i}(\Hess(\mathsf{N}, h))) +  \sum_{T\in \SK_2(\Gamma_h)}\; \sum_{\substack{\mu\vdash (n-2)\\ \mu=(\mu_1,\mu_2)}} c_{\mu, i-\deg(T)}^T \dim (M^{\mu})^{\Symm_{\mu'}} \\
\end{eqnarray*}
where in the last expression, both $\mu=(\mu_1, \mu_2)$ and $\mu'=(\mu_1', \mu_2')$ are partitions of $n-2$, and the last equality follows from the case $\nu=(n)$ proven above.
Since $\nu = (\mu_1'+1, \mu_2'+2)$, it follows from Proposition~\ref{prop:induction step} that to prove~\eqref{eq: nu-fixed pts2} it is enough to prove
\begin{eqnarray}\label{eq: two-parts}
	\dim H^{2i-2\deg(T)} (\Hess(\mathsf{X}_{\mu'}, h_T)) = \sum_{\substack{\mu\vdash (n-2)\\ \mu=(\mu_1,\mu_2)}} c_{\mu, i-\deg(T)}^T \dim (M^{\mu})^{\Symm_{\mu'}}
\end{eqnarray}
for each $T\in \SK_2(\Gamma_h)$. To see this, recall that the coefficients $c_{\mu, i-\deg(T)}^T$ are defined by the equality 
\[
H^{2i - 2\deg(T)}(\Hess(\mathsf{S}', h_T)) = \sum_{\mu\vdash (n-2)} c^T_{\mu, i-\deg(T)} M^\mu
\]
in $\mathcal{R}ep(\Symm_n)$. Moreover, as in the argument above, since $I_h$ is abelian we know from Lemma~\ref{lem: maximal sink set induction} that\break $c^T_{\mu, i-\deg(T)} = 0$ for all $\mu \vdash (n-2)$ with more than $2$ parts. This observation, together with taking $\Symm_{\mu'}$-invariants for $\mu'=(\mu_1',\mu_2') \vdash (n-2)$, yields 
\begin{equation}\label{eq: mu-fixed points n-2}
\dim (H^{2i-2\deg(T)}(\Hess(\mathsf{S}', h_T)))^{\Symm_{\mu'}} = 
\sum_{\substack{\mu\vdash (n-2)\\ \mu=(\mu_1,\mu_2)}} c_{\mu, i-\deg(T)}^T \dim(M^\mu)^{\Symm_{\mu'}}.
\end{equation}
Now another application of Theorem~\ref{thm: BrosnanChow main thm} on the LHS of~\eqref{eq: mu-fixed points n-2} yields the equality in~\eqref{eq: two-parts} as desired. Hence~\eqref{eq: nu-fixed pts2} holds for all $\nu$ with at most $2$ parts, concluding the proof.
\end{proof}


\section{Proofs of Propositions~\ref{prop: reg step} and~\ref{prop:induction step} and the abelian graded Stanley-Stembridge conjecture}\label{section: the proofs}

In this section, we prove the two main inductive propositions from the previous section. The arguments involve the combinatorics of $\Symm_n$ and root systems. Given all of the preparation in the previous sections, the arguments are lengthy but not particularly difficult. \textbf{Throughout this section we work in the setting of Propositions~\ref{prop: reg step} and~\ref{prop:induction step} and Theorem~\ref{theorem:main}. Thus we always assume that $n \geq 3$, that $h:[n]\to [n]$ is a Hessenberg function such that $I_h$ is abelian,  and that any partition has at most two parts.}

\subsection{The proof of Proposition~\ref{prop: reg step}}

We begin with a proof of Proposition~\ref{prop: reg step}. This is much simpler than the proof of Proposition~\ref{prop:induction step}, which occupies the bulk of this section, due to the fact that the cohomology of the regular nilpotent Hessenberg variety $\Hess(\mathsf{N},h)$ is related to the subspace of $H^*(\Hess(\mathsf{S},h))$ which is invariant under the entire group $\Symm_n$, as opposed to a Young subgroup $\Symm_\nu$ for some $\nu \vdash n$. The fact that $\dim(M^\lambda)^{\Symm_n}=1$ for any partition $\lambda$ then allows us to use Theorem \ref{thm: Stanley} to translate our problem into the language of acyclic orientations. Our ``sink-set decomposition''~\eqref{eq: sink set decomposition}, and the inductive description of acyclic orientations given in Proposition~\ref{proposition: max sink set induction}, then yields the result. We now make this sketch precise.

\begin{proof}[Proof of Proposition~\ref{prop: reg step}]

We begin by observing that 
\[
P(\Hess(\mathsf{N},h), t) = \sum_{i=0}^{\lvert \Phi_h^- \rvert} \dim H^{2i}(\Hess(\mathsf{N},h)) t^{2i} = 
\sum_{i=0}^{\lvert \Phi_h^- \rvert} \dim (H^{2i}(\Hess(\mathsf{S},h)))^{\Symm_n}t^{2i} 
\]
where the first equality is the definition of the Poincar\'e polynomial, together with the fact that 
$\dim_{\C}(\Hess(\mathsf{N},h)) = \lvert \Phi_h^- \rvert$ \cite[Corollary 2.7]{Precup2015}, and the second equality is by Brosnan and Chow's
Theorem~\ref{thm: BrosnanChow main thm}.
Since $H^{2i}(\Hess(\mathsf{S},h)) = \sum_{\nu\, \vdash n} c_{\nu,i} M^\nu$ by definition of the coefficients $c_{\nu,i}$, by taking $\Symm_n$-invariants we obtain 
\begin{equation*}
\begin{split} 
P(\Hess(\mathsf{N},h),t) 
 & = \sum_{i=0}^{\lvert \Phi_h^- \rvert} \left( \sum_{\nu \vdash n} c_{\nu,i} \dim(M^\nu)^{\Symm_n} \right) t^{2i}  \\
 & = \sum_{i=0}^{\lvert \Phi_h^- \rvert} \left( \sum_{\nu \vdash n} c_{\nu,i} \right) t^{2i} \;\;
 \textup{ since $\dim (M^{\nu})^{\Symm_n} = 1$ } \\
 & = \sum_{i=0}^{\lvert \Phi_h^- \rvert} \left(  \sum_{\substack{\substack{\nu\vdash n \textup{ and } \nu \\ \textup{ has at most $2$ parts}}}} c_{\nu,i} \right) t^{2i} \;\; \textup{ by Lemma~\ref{lem: maximal sink set induction} since $I_h$ is abelian} \\ 
   & = \sum_{i=0}^{\lvert \Phi_h^- \rvert} c_{(n),i} t^{2i}  + \sum_{i=0}^{\lvert \Phi_h^- \rvert} \left( \sum_{\substack{\substack{\nu\vdash n \textup{ and } \nu \\ \textup{has $2$ parts}}}} c_{\nu,i} \right) t^{2i}. 
 \end{split} 
 \end{equation*}
 A similar argument yields
 \[
 P(\Hess(\mathsf{N}', h_T),t) = \sum_{i=0}^{\lvert \Phi_{h_T}^- \rvert} \left( \sum_{\mu \vdash (n-2)} c^T_{\mu, i} \right) t^{2i}
 \]
 for any $T \in SK_2(\Gamma_h)$. The above equalities imply that in order to prove the proposition it suffices to prove
\begin{eqnarray}\label{eq: triv case0} 
	\sum_{i=0}^{|\Phi_h^-|} \; \left( \sum_{\substack{\nu\vdash n \textup{ and } \nu \\ \textup{has $2$ parts}}} c_{\nu,i} \right) \,t^{2i} =  \sum_{T\in \SK_2(\Gamma_h)} t^{2\deg(T)} 
	\sum_{i=0}^{\lvert \Phi_{h_T}^- \rvert} \left( \sum_{\mu \vdash (n-2)} c^T_{\mu, i} \right) t^{2i}.
\end{eqnarray}
Applying Theorem~\ref{thm: Stanley} and our previous combinatorial analysis of acyclic orientations we have	
\begin{eqnarray}\label{eq: triv case1}\begin{split}
	\quad\quad\quad \sum_{i=0}^{|\Phi_h^-|} \;  \left( \sum_{\substack{\nu\vdash n \textup{ and } \nu \\ \textup{has $2$ parts}}} c_{\nu,i}\,\right) \, t^{2i} &= \sum_{i=0}^{|\Phi_h^-|}\,  | \{ \omega \in \mathcal{A}_2(\Gamma_h) \hsm \vert \hsm  \asc(\omega)=i \}|\,t^{2i} \;\; \textup{ by Theorem~\ref{thm: Stanley} }  \\	 
	&= \sum_{T\in \SK_2(\Gamma_h)} \; \sum_{i=0}^{|\Phi_h^-|} \; |\{\omega\in \mathcal{A}_2(\Gamma_h) \hsm \vert \hsm \asc(\omega)=i \textup{ and } \sk(\omega)=T\} |\,t^{2i} \\ 
	&\quad\quad\quad\quad\quad\quad\quad\quad\quad\textup{by our sink-set decomposition from~\eqref{eq: sink set decomposition}} \\
	&= \sum_{T\in \SK_2(\Gamma_h)} \; \sum_{i=\deg(T)}^{|\Phi_h^-|} \; |\{\omega_T \in \mathcal{A}(\Gamma_{h_T}) \hsm \vert \hsm \asc(\omega_T)=i-\deg(T)\} |\,t^{2i}\\
	&\quad\quad\quad\quad\quad\quad\quad\quad\quad\textup{ by Proposition~\ref{proposition: max sink set induction}, } \\ 
\end{split}\end{eqnarray}
where the sum over the index $i$ ranges between $\deg(T)$ and $\lvert \Phi_h^- \rvert$ because it follows from Proposition~\ref{proposition: max sink set induction} that if $\sk(\omega)=T$ then $\asc(\omega)\geq \deg(T)$. 
For each $T\in \SK_2(\Gamma_h)$ we shift the index $i$ of the sum appearing on the RHS of the above equation to get
\begin{eqnarray}\label{eq: triv case2} \begin{split}
	&\sum_{i=\deg(T)}^{|\Phi_h^-|} \; |\{\omega_T \in \mathcal{A}(\Gamma_{h_T}) \hsm \vert \hsm \asc(\omega_T) =i-\deg(T)\} |\,t^{2i}\\
	& \quad\quad\quad\quad\quad = t^{2\deg(T)}\, \sum_{i=0}^{|\Phi_h^-|-\deg(T)} |\{ \omega_T\in \mathcal{A}(\Gamma_{h_T}) \hsm \vert \hsm \asc(\omega_T) = i \}|\,t^{2i} \\
	& \quad\quad\quad\quad\quad = t^{2\deg(T)}\, \sum_{i=0}^{|\Phi_{h_T}^-|} |\{ \omega_T\in \mathcal{A}(\Gamma_{h_T}) \hsm \vert \hsm \asc(\omega_T) = i \}|\,t^{2i}
\end{split}\end{eqnarray}
where the last equality follows from that fact that $|\Phi_h^-|-\deg(T) \geq |\Phi_{h_T}^-|$ by Lemma \ref{lem: deg(T) property1} and \break$|\{ \omega\in \mathcal{A}(\Gamma_{h_T}) \hsm \vert \hsm \asc(\omega) = i \}|=0$ for all $i>|\Phi_{h_T}^-|$ since $\asc(\omega_T)\leq |E(\Gamma_{h_T})|=|\Phi_{h_T}^-|$ for all $\omega_T\in \mathcal{A}(\Gamma_{h_T})$. 
Putting together Corollary~\ref{corollary: sum clambda for all i} with the above equation~\eqref{eq: triv case2} we obtain
\begin{eqnarray}\label{eq: triv case3}
\quad\quad\sum_{i=\deg(T)}^{|\Phi_h^-|} \; |\{\omega_T \in \mathcal{A}(\Gamma_{h_T}) \hsm \vert \hsm \asc(\omega_T) =i-\deg(T)\} |\,t^{2i} = t^{2\deg(T)} \sum_{i=0}^{|\Phi^-_{h_T}|}\; \left( \sum_{\mu \vdash (n-2)}  c_{\mu,i}^T\,\right) \, t^{2i}
\end{eqnarray}
for each $T\in \SK_2(\Gamma_h)$.  Finally, Equations~\eqref{eq: triv case1} and~\eqref{eq: triv case3} together imply~\eqref{eq: triv case0} as desired.
\end{proof}


\subsection{Proof of Proposition~\ref{prop:induction step}} \label{sec: main proofs}

In this section we prove Proposition~\ref{prop:induction step}. This argument is the technical heart of this paper and is rather involved, so a sketch of the overall picture may be helpful. Our starting point is the explicit and purely combinatorial formula for the Betti numbers $b_{2i}$ of the regular Hessenberg variety $\Hess(\mathsf{X}_{\nu}, h)$ given by the second author in \cite{Precup2016} which expresses $b_{2i}$ as the number of permutations $w\in \Symm_n$ satisfying certain conditions related to $\nu \vdash n$ and $h$. Our assumptions that $I_h$ is abelian and that all partitions have at most $2$ parts simplifies the combinatorics of the Poincar\'e polynomial. From there, the remainder of the argument is a careful analysis of the sets of permutations in question, which boils down to the combinatorics of $\Symm_n$ and the root system of type A. There are two points worth mentioning. First, it turns out to be important that the formula for the Poincar\'e polynomial in \cite{Precup2016} is valid for any two-part composition $n=\nu_1+\nu_2$ where $\nu=(\nu_1,\nu_2)$ is not necessarily a partition, i.e., we may have $\nu_1 < \nu_2$ instead of the more customary $\nu_1 \geq \nu_2$. Accordingly, in this section, the standing hypotheses on $\nu=(\nu_1, \nu_2)$ are as follows: 
\[
\nu_1, \nu_2 \in \Z, \quad \nu_1+\nu_2=n, \quad \nu_1 \geq 0, \nu_2 \geq 0.
\]
Allowing this level of generality allows us to prove an important special case in our arguments below. Secondly, in order to reduce the argument to the special case mentioned in the previous sentence, we use the \textbf{shortest coset representative} (also used by the second author in \cite{Precup2016}) of a permutation $w\in \Symm_n$ as an element of $W\backslash \Symm_n$ for a certain Young subgroup $W\subseteq \Symm_n$.

To begin, we state the formula for the Betti numbers of regular Hessenberg varieties given in \cite{Precup2016}.  We prepare some terminology. For each $w\in \Symm_n$ we define the \textbf{inversion set of $w$} as 
\[
	N^-(w):= \{ \gamma\in \Phi^- \hsm \vert \hsm w(\gamma)\in \Phi^+  \},
\]
i.e., for each $w\in \Symm_n$, the set $N^-(w)$ consists of the negative roots which become positive under the action of $w$. In Lie type A this can be expressed quite concretely. Indeed, let $\gamma=t_i-t_j$ for some $i>j$. The action of $\Symm_n$ on roots is given by 
\[
w(t_i - t_j) = t_{w(i)} - t_{w(j)}.
\]
Thus $\gamma\in N^-(w)$ if and only if $w(i)<w(j)$.  We may therefore naturally identify the set $N^-(w)$ with the set of ordered pairs
\[
\inv(w) := \{ (i,j) \hsm \vert \hsm i>j \textup{ and } w(i)<w(j) \}.
\]
We with use this identification frequently throughout this section.

We now state (a special case of) the key formula for the Poincar\'e polynomial of regular Hessenberg varieties \cite[Lemma 2.6]{Precup2016}, which is the starting point of our discussion. Let $\nu=(\nu_1, \nu_2) \in \Z^2$ such that $\nu_1+\nu_2=n$ and $\nu_1, \nu_2 \geq 0$. There is a corresponding subset $J_\nu$ of positive simple roots given by
\[
J_\nu := \Delta \setminus \{\alpha_{\nu_1}\} \subset \Delta
\]
if $0 < \nu_1 < n$ and $J_\nu = \Delta$ otherwise. 
Recall from Section~\ref{sec: Hessenberg basics} that $\mathsf{X}_{\nu} = \mathsf{X}_{(\nu_1, \nu_2)}$ denotes a matrix in standard Jordan canonical form which is regular of Jordan type $\nu$. Recall also that the value of the Poincar\'e polynomial $P(\Hess(\mathsf{X}_{\nu}, h), t)$ when $t=1$ is equal to $\dim H^*(\Hess(\mathsf{X}_{\nu},h))$. 

\begin{lemma} \label{lem: Betti numbers of regular Hess} (\cite[Lemma 2.6]{Precup2016}) 
Let $n$ be a positive integer, $h: [n] \to [n]$ a Hessenberg function and $\nu = (\nu_1, \nu_2)$ as above. Then 
\begin{equation}\label{eq: Poincare polynomial of reg Hess}
	P(\Hess(\mathsf{X}_{\nu},h),t) = \sum_{\substack{w\in \Symm_n \\ w^{-1}(J_\nu) \subseteq \Phi_h }} t^{2|N^-(w)\cap \Phi_h^-|}.
\end{equation}
In particular, evaluating at $t=1$, we obtain 
\[
P(\Hess(\mathsf{X}_{\nu},h),1) = \dim H^*(\Hess(\mathsf{X}_{\nu},h)) = \lvert \{ w \in \Symm_n \hsm \vert \hsm w^{-1}(J_{\nu}) \subseteq \Phi_h \} \rvert.
\]
\end{lemma}

The exponents appearing in the RHS of~\eqref{eq: Poincare polynomial of reg Hess} can be interpreted concretely in terms of the Hessenberg function. Indeed, under the identification of $N^-(w)$ with $\inv(w)$, the set $N^-(w)\cap \Phi_h^-$ corresponds to 
\[
	\inv_h(w):=\{ (i,j) \hsm \vert \hsm i>j ,\; w(i)<w(j), \textup{ and } i\leq h(j)  \}.
\]

Next, we analyze the indexing set of the summation on the RHS of~\eqref{eq: Poincare polynomial of reg Hess}, i.e., we study the set of $w \in \Symm_n$ such that $w^{-1}(J_\nu) \subseteq \Phi_h$. 
Let $\nu=(\nu_1, \nu_2)$. Then 
$J_\nu = \Delta \setminus \{ \alpha_\nu \}$.  Thus for any $w\in \Symm_n$ with $w^{-1}(J_{\nu}) \subseteq \Phi_h$, we have that either $w^{-1}(\Delta) \subseteq \Phi_h$, or, $w^{-1}(J_\nu) \subseteq \Phi_h$ and $w^{-1}(\alpha_{\nu}) \in I_h$.  Motivated by this observation we consider the set
\[
\mathcal{D}_\nu := \{w \in \Symm_n \hsm \vert \hsm w^{-1}(J_\nu) \subseteq \Phi_h \textup{ and } w^{-1}(\alpha_{\nu}) \in I_h \}.
\]
With this notation in place we obtain from Lemma~\ref{lem: Betti numbers of regular Hess} that
\begin{eqnarray}\label{eq: poincare polynomial decomp} \begin{split}
	P(\Hess(\mathsf{X}_{\nu}, h),t) &= \sum_{\substack{w\in \Symm_n\\ w^{-1}(\Delta)\subseteq \Phi_h}} t^{2|N^-(w)\cap \Phi_h^-|}+ \sum_{w\in \mathcal{D}_{\nu}} t^{2|N^-(w)\cap \Phi_h^-|}\\
	&= P(\Hess(\mathsf{N}, h), t) + \sum_{w\in \mathcal{D}_{\nu}} t^{2|N^-(w)\cap \Phi_h^-|}
\end{split}\end{eqnarray}
where the second equality follows from an application of Lemma~\ref{lem: Betti numbers of regular Hess} to the case $\nu=(n)$, the trivial composition with $\nu_1=0$ or $\nu_2=0$, and for which we may take $X_{(n)}=\mathsf{N}$ and $J_{(n)} = \Delta$. The discussion above indicates that the key step in the proof of Proposition~\ref{prop:induction step} is the following.

\begin{proposition}\label{prop: graded case}  
Under the notation and assumptions of Proposition \ref{prop:induction step}, in particular with $\nu=(\mu_1+1, \mu_2+1) \vdash n$, we have 
\begin{eqnarray}\label{eq:induction equation 2}
\sum_{w\in \mathcal{D}_{\nu}} t^{2|N^-(w)\cap \Phi_h^-|} = \sum_{T\in \SK_2(\Gamma_h)} t^{2\deg(T)} P(\Hess(\mathsf{X}_{\mu},h_T),t)
\end{eqnarray}
where $\mu=(\mu_1,\mu_2) \vdash (n-2)$.  
\end{proposition}

In order to prove Proposition~\ref{prop: graded case}, first recall from Proposition~\ref{lem: root subsets} that there is a bijection $\SK_2(\Gamma_h) \to \mathcal{R}_1(I_h)$ which assigns each sink set $T$ of cardinality $2$ to a unique singleton set $R_T \in \mathcal{R}_1(I_h)$.  Under this correspondence, if $R_T=\{\beta\}$, we write $\beta=\beta_T$. Moreover, by Proposition~\ref{lemma: abelian ht 1}, for any Hessenberg function $h$, the ideal $I_h$ bijectively corresponds to $\mathcal{R}_1(I_h)$ by $\beta \leftrightarrow \{\beta\}$. With this in mind, for each $\beta\in I_h$, we set the notation 
\[
	\mathcal{D}_{\nu}(\beta) := \{ w\in \Symm_n \hsm \vert \hsm w^{-1}(J_{\nu})\subseteq \Phi_h \textup{ and } w^{-1}(\alpha_{\nu}) = \beta \}.
\]
 From the above discussion it is straightforward to see that 
\begin{eqnarray}\label{eq: Dnu decomposition}
\mathcal{D}_\nu = \bigsqcup_{\beta \in I_h} \mathcal{D}_\nu(\beta) = \bigsqcup_{T\in\SK_2(\Gamma_h) } \mathcal{D}_{\nu}(\beta_T).
\end{eqnarray}
We can now rewrite~\eqref{eq:induction equation 2} as
\begin{equation}\label{eq: Poincare poly in terms of Dnu}
	 \sum_{T\in \SK_2(\Gamma_h)}\sum_{w\in \mathcal{D}_{\nu}(\beta_T)} t^{2|N^-(w)\cap \Phi_h^-|} = \sum_{T\in \SK_2(\Gamma_h)}t^{2\deg(T)} P(\Hess(\mathsf{X}_{\mu}, h_T), t).
\end{equation}
The proof of this equality consists of two parts.  We first show an ungraded version of this equality (i.e., we prove that the equation above holds when we evaluate at $t=1$) and then address the graded case.  We need some additional notation for our proof.

Suppose $\beta= t_a - t_b \in I_h$. Recall that this is equivalent to the conditions $a>b$ and $a>h(b)$. If $w \in \mathcal{D}_{\nu}(\beta)$ then by definition of $\mathcal{D}_{\nu}(\beta)$ we must have 
\[
w^{-1}(t_{\nu_1}-t_{\nu_1+1}) = t_{w^{-1}(\nu_1)} - t_{w^{-1}(\nu_1+1)} = t_a - t_b,
\]
or equivalently
\begin{equation}\label{eq:wa and wb}
w(a)=\nu_1 \textup{ and } w(b)=\nu_1+1, 
\end{equation}
so the $b$-th entry in the one-line notation of $w$ is $\nu_1+1$ and the $a$-th entry is $\nu_1$. 

In the arguments that follow it will be useful to choose a specific element of $\mathcal{D}_\nu(\beta)$ for each $\beta \in I_h$. We define this element as follows. 

\begin{definition} \label{def: w-nu-beta}
Suppose $\beta=t_a-t_b \in I_h$. We define a permutation in $\Symm_n$, denoted $w_{\nu, \beta}$, by: 
\begin{enumerate} 
\item $w_{\nu,\beta}(a)=\nu_1$ and $w_{\nu,\beta}(b)=\nu_1+1$ (i.e. $w_{\nu,\beta}$ satisfies the condition~\eqref{eq:wa and wb}) and 
\item the remaining entries in the one-line notation of $w_{\nu,\beta}$ list the integers $[n] \setminus \{\nu_1, \nu_1+1\}$ in increasing order from left to right. 
\end{enumerate} 
\end{definition} 

\begin{example} \label{ex: w-nu-beta example}
Let $n=6$ and $\beta= t_5-t_2$, so $a=5$ and $b=2$. Let $\nu=(4,2)$. Then $w_{\nu,\beta}(2)=\nu_1+1=5$ and $w_{\nu,\beta}(5)=\nu_1=4$, and the remaining entries are filled, in increasing order, by $[6] \setminus \{4,5\} = \{1,2,3,6\}$. The one-line notation of $w_{\nu,\beta}$ is 
$[1 \hsm \mathbf{5} \hsm 2 \hsm 3 \hsm \mathbf{4} \hsm 6 ]$
where condition~\eqref{eq:wa and wb} determines the entries in bold. 
\end{example}

We need the following. 

\begin{lemma}  \label{lem: wmin satisfies the Hess. condition}
Let $w_{\nu,\beta}$ be as above and suppose $I_h$ is abelian. Then
\begin{enumerate} 
\item If $(i,j)\in \inv(w_{\nu,\beta})$ then $\{ i,j \} \cap \{ a,b \}\neq \emptyset$, and
\item $w_{\nu,\beta} \in \mathcal{D}_{\nu}(\beta)$. 
\end{enumerate} 
\end{lemma}

\begin{proof} 
To prove (1), we will show the contrapositive, i.e. if $\{ i,j \}\cap \{a,b\} = \emptyset $ then $(i,j)\notin \inv(w_{\nu,\beta})$.  Suppose $(i,j)$ such that $i>j$ and $\{i,j\}\cap \{a,b\} = \emptyset$.  Since $\{i,j\}\cap \{a,b\} = \emptyset$ we have $\{ w_{\nu,\beta}(i), w_{\nu,\beta}(j) \} \cap \{\nu_1, \nu_1+1\} =\emptyset$, and it follows that $w_{\nu, \beta}(i)>w_{\nu, \beta}(j)$ by condition (2) in Definition~\ref{def: w-nu-beta}.  Therefore $(i,j) \notin \inv(w_{\nu, \beta})$.

Now we prove (2).  By definition, $w_{\nu, \beta}^{-1}(\alpha_{\nu})=\beta$ so we need only show that $w_{\nu,\beta}^{-1}(J_{\nu}) \subseteq \Phi_h$. We take cases. First consider the case in which $\alpha+\alpha_{\nu}\in \Phi$, i.e. $\alpha$ and $\alpha_\nu$ correspond to adjacent vertices in the Dynkin diagram. Seeking a contradiction, suppose $w^{-1}_{\nu, \beta}(\alpha)\in I_h$.  Since $\alpha+\alpha_\nu \in \Phi$ we also have 
$w_{\nu,\beta}^{-1}(\alpha + \alpha_\nu) = w_{\nu,\beta}^{-1}(\alpha) + w_{\nu,\beta}^{-1}(\alpha_\nu) \in I_h$ since $I_h$ is an ideal.  On the other hand, this is a contradiction since $I_h$ is abelian.
Next, consider the case in which $\alpha+\alpha_{\nu}\notin \Phi$, i.e. $\alpha$ and $\alpha_\nu$ are not adjacent in the Dynkin diagram.  This means that $\alpha=t_{i}-t_{i+1}$ where $\{ i,i+1 \}\cap \{ \nu_1, \nu_1+1 \}= \emptyset$. Condition (2) in Definition~\ref{def: w-nu-beta} implies  
$w_{\nu, \beta}^{-1}(i) <w_{\nu,\beta}^{-1}(i+1)$ since $i,i+1\in [n]\setminus\{ \nu_1,\nu_1+1 \}$.  Therefore $w_{\nu, \beta}^{-1}(\alpha)\in \Phi^+ \subseteq \Phi_h$ as desired.
\end{proof}

With the element $w_{\nu, \beta}$ chosen as above, we can now construct some explicit maps which will be useful for our induction argument. Recall the bijection $\phi_T: [n]\setminus T \to [n-|T|]$ defined in Section~\ref{sec: Sink sets and induced subgraphs} by $\phi_T(j)=j-j'$ where $j'$ denotes the number of vertices $i\in T$ such that $i\leq j$.  We now consider this bijection for the case in which $T=\{a,b\}$ more carefully.  This bijection allows us to view $\mathfrak{gl}(n-2,\C)$ as a Lie subalgebra of $\mathfrak{gl}(n,\C)$.  We make this identification more precise in equations~\eqref{eq: tau to x-tau}, \eqref{eq: gl(n-2) identification}, and \eqref{eq: gl(n-2) roots identification} below. 

Viewing $\Symm_n$ as the automorphism group of the set $[n]$ of letters $\{1,2,\ldots, n\}$, there is a stabilizer subgroup 
\begin{equation}\label{eq: def Stab(a,b)}
	\Stab(a,b) := \{w \in \Symm_n \hsm\vert\hsm w(a)=a, w(b)=b\} \cong \mathrm{Aut}([n] \setminus \{a,b\})
\end{equation}
of $\Symm_n$ which is naturally isomorphic to $\Symm_{n-2}$ via the map
\begin{eqnarray}\label{eq: tau to x-tau}
\tau \mapsto x_{\tau} :=\left[ \phi_T(\tau(1))\hsm \phi_T(\tau(2)) \hsm \cdots \hsm \widehat{\phi_T(b)}\hsm \cdots \hsm \widehat{\phi_T(a)} \hsm \cdots \hsm \phi_T(\tau(n)) \right] 
\end{eqnarray}
where the marked entries are deleted.  In what follows we will frequently identify $\Stab(a,b)$ with $\Symm_{n-2}$. We have the following. 

\begin{lemma} \label{claim: Stab map}
Let $w \in \Symm_n$ such that $w(a)=\nu_1$ and $w(b)=\nu_1+1$. Then there exists a unique $\tau \in \Stab(a,b)$ such that $w = w_{\nu, \beta}\tau$. In particular, there is a well-defined map 
\[
\Psi_{\nu, \beta}: \{w \in \Symm_n \hsm\vert \hsm w(a)=\nu_1, w(b)=\nu_1+1 \} \to \Symm_{n-2} 
\]
defined by $\Psi_{\nu, \beta}(w) = x_{\tau}$ where $x_{\tau}$ is the unique element in $\Symm_{n-2}$ corresponding to $\tau \in \Stab(a,b)$ via~\eqref{eq: tau to x-tau}.
\end{lemma} 

\begin{proof}[Sketch of proof]
The hypotheses on $w$ completely determine the $a$-th and $b$-th entries in its one-line notation. The other entries must be a permutation of the set $[n] \setminus \{a,b\}$, and the hypotheses on $w$ place no condition on this permutation. Recall that for $w_{\nu,\beta}$ and any permutation $\tau$ in $\Symm_n$, right-composition with $\tau$ ``acts on positions'', i.e. if $w_{\nu,\beta}$ sends $i$ to $w_{\nu,\beta}(i)$, then $w_{\nu,\beta} \tau$ must send $i$ to $w_{\nu,\beta}(\tau(i))$. Thus if $\tau$ stabilizes $a$ and $b$, then $w=w_{\nu,\beta}\tau$ still satisfies $w(a)=\nu_1$ and $w(b)=\nu_1+1$. Moreover, it is straightforward to see that such a $\tau$ is unique, making $\Psi_{\nu,\beta}$ well-defined. 
\end{proof}

\begin{example}  \label{ex: w to tau correspondence}
Let $n=5$, $\beta=t_5-t_2$, and $\nu = (3,2)$.  Then $a=5$, $b=2$, and $\nu_1=3$.  Let $h=(3,4,5,5,5)$, which is abelian, and $I_h$ contains $\beta$. 
The table below gives an explicit description of the map $\Psi_{\nu, \beta}$ defined in Lemma~\ref{claim: Stab map}.
\begin{eqnarray*}
\begin{array} {c|c|c}
w\in \Symm_5 \textup{ such that } w(5)=3 \textup{ and } w(2) = 4 & \tau \in \Symm_5 & x_{\tau} \in \Symm_3 \\ \hline
\left[ 1\hsm \mathbf{4} \hsm 2 \hsm 5 \hsm \mathbf{3} \right] & \left[ 1\hsm \mathbf{2} \hsm 3 \hsm 4 \hsm \mathbf{5} \right] & \left[1 \hsm 2 \hsm 3\right] \\
\left[ 2\hsm \mathbf{4} \hsm 1 \hsm 5 \hsm \mathbf{3} \right] & \left[ 3\hsm \mathbf{2} \hsm 1 \hsm 4 \hsm \mathbf{5} \right] & \left[2 \hsm 1 \hsm 3\right] \\
\left[ 1\hsm \mathbf{4} \hsm 5 \hsm 2 \hsm \mathbf{3} \right] & \left[ 1\hsm \mathbf{2} \hsm 4 \hsm 3 \hsm \mathbf{5} \right] & \left[1 \hsm 3 \hsm 2\right] \\
\left[ 2\hsm \mathbf{4} \hsm 5 \hsm 1 \hsm \mathbf{3} \right] & \left[ 3\hsm \mathbf{2} \hsm 4 \hsm 1 \hsm \mathbf{5} \right] & \left[2 \hsm 3 \hsm 1\right] \\
\left[ 5\hsm \mathbf{4} \hsm 1 \hsm 2 \hsm \mathbf{3} \right] & \left[ 4\hsm \mathbf{2} \hsm 1 \hsm 3 \hsm \mathbf{5} \right] & \left[3 \hsm 1 \hsm 2\right] \\
\left[ 5\hsm \mathbf{4} \hsm 2 \hsm 1 \hsm \mathbf{3} \right] & \left[ 4\hsm \mathbf{2} \hsm 3 \hsm 1 \hsm \mathbf{5} \right] & \left[3 \hsm 2 \hsm 1\right]
\end{array}
\end{eqnarray*} 
\end{example}

There is a natural Lie subalgebra of $\mathfrak{gl}(n,\C)$ obtained by ``setting the variables in rows/columns $a$ and $b$ equal to zero''.  More precisely, we have a natural Lie algebra isomorphism, 
\begin{eqnarray}\label{eq: gl(n-2) identification}
\{ X \in \mathfrak{gl}(n,\C) \hsm\vert\hsm X_{ij} = 0 \textup{ if } \{i,j\} \cap \{a,b\}\neq \emptyset \} \cong \mathfrak{gl}(n-2,\C)
\end{eqnarray}
defined explicitly on the basis $\{E_{ij}\hsm \vert\hsm \{i,j\}\cap \{a,b\}=\emptyset \}$ of the LHS by $E_{ij}\mapsto E_{\phi_T(i)\phi_T(j)}$ and extended linearly.  Recall that in Section~\ref{sec: Sink sets and induced subgraphs} we proved that each sink set $T\in \SK_2(\Gamma_h)$ corresponds to a Hessenberg function $h_T:[n-2]\to [n-2]$ whose incomparability graph is obtained by deleting the vertices in $T$ and any incident edges from $\Gamma_h$. In fact, this is the Hessenberg function which corresponds to the Hessenberg space $H\cap \mathfrak{gl}(n-2,\C)$ under the identification in \eqref{eq: gl(n-2) identification}, and $\phi_T$ can also be used to give an explicit map between the corresponding root systems.  Using the notation of this section, $\beta= \beta_T = t_a-t_b$ so $T=\{a,b\}$.  We let 
\[
	\Phi[T]:=\{ t_i-t_j  \hsm \vert\hsm 1\leq i,j \leq n; \; \{ i,j\} \cap \{a,b\} = \emptyset \} \subseteq \Phi
\]
and $\Phi_T$ denote the root system of $\mathfrak{gl}(n-2, \C)$.  Now there is an explicit isomorphism of root systems
\begin{eqnarray}\label{eq: gl(n-2) roots identification}
\Phi[T] \cong \Phi_T \textup{ defined by } t_i-t_j \mapsto t_{\phi_T(i)} - t_{\phi_T(j)}.
\end{eqnarray}
where $\Phi[T]$ is viewed as a subroot system of $\Phi$ (since $\Phi[T]$ is closed under addition in $\Phi$).  Moreover, if $\Phi^-[T] := \Phi^- \cap \Phi[T]$, $\Phi_h[T] :=\Phi_h\cap \Phi[T]$, and $\Phi_h^-[T]: = \Phi_h^-\cap \Phi[T]$ then these subsets of $\Phi[T]$ correspond to $\Phi_T^-$, $\Phi_{h_T}$, and $\Phi_{h_T}^-$ respectively via~\eqref{eq: gl(n-2) roots identification}.

\begin{remark} \label{rem: inversion set compatibility}  The root system isomorphism given in~\eqref{eq: gl(n-2) roots identification} is compatible with the corresponding identification $\Stab(a,b)\cong \Symm_{n-2}$ given in~\eqref{eq: tau to x-tau}.  If $\tau\in \Stab(a,b)$ and $t_i - t_j \in \Phi[T]$ then $\tau(t_i-t_j)\in \Phi[T]$ and
\[
	t_k-t_{\ell} = \tau(t_i-t_j) \Leftrightarrow t_{\phi_T(k)}-t_{\phi_T(\ell)} = x_{\tau}(t_{\phi_T(i)} - t_{\phi_T(j)}).
\]
In particular,~\eqref{eq: gl(n-2) roots identification} maps $N^-(\tau)$ to $N^-(x_{\tau})$.
\end{remark}

\begin{example}  \label{ex: Phi to Phi[T] correspondence}
Continuing Example~\ref{ex: w to tau correspondence}, we get
\[
	\Phi^-[T] = \{ t_3-t_1, t_4-t_1, t_4-t_3 \}
\]
since $T=\{ 2,5 \}$.  Pictorially,
to obtain $\Phi[T]$ (respectively $\Phi^-[T]$) from $\Phi$ (respectively $\Phi^-$) we
simply remove those roots
in the $a$-th and $b$-th rows and columns.  The picture below 
illustrates the case 
$T=\{2,5\}$.
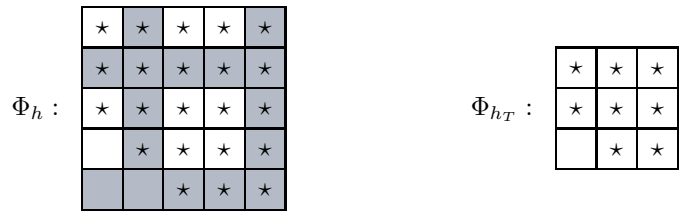
 \[
\ytableausetup{centertableaux}
\Phi_h: \;\; \begin{ytableau} \star & *(ltgrey) \star & \star & \star & *(ltgrey)\star \\ *(ltgrey)\star & *(ltgrey)\star & *(ltgrey)\star & *(ltgrey)\star &  *(ltgrey) \star \\  \star & *(ltgrey)\star & \star & \star & *(ltgrey)\star \\ \empty & *(ltgrey)\star & \star & \star & *(ltgrey)\star \\  *(ltgrey)\empty & *(ltgrey)\empty & *(ltgrey)\star & *(ltgrey)\star & *(ltgrey)\star   \end{ytableau}
\quad\quad\quad\quad\quad\quad\quad
\Phi_{h_T}: \;\; \begin{ytableau} \star & \star & \star \\ \star & \star & \star \\  \empty & \star & \star \end{ytableau}
\]
\end{example}

\begin{lemma} \label{lem: wmin conjugation}
Let $T \in SK_2(\Gamma_h)$ and let $\beta=\beta_T$ be the corresponding element of $I_h$. Let $w_{\nu,\beta}$ be the permutation of Definition~\ref{def: w-nu-beta} associated to $\beta=\beta_T$.  
\begin{enumerate} 
\item If $\alpha_i = t_{i} - t_{i+1}$ is adjacent to $\alpha_{\nu}$ in the Dynkin diagram (equivalently, $i=\nu_1-1$ or $i=\nu_1+1$), then $w_{\nu, \beta}^{-1}(\alpha) \not \in \Phi[T]$. 
\item If $\alpha_i$ is not adjacent to $\alpha_{\nu}$ in the Dynkin diagram, then $w_{\nu, \beta}^{-1}(\alpha_i) \in \Phi_h[T]$. 
\end{enumerate} 
Furthermore, the set $w_{\nu,\beta}^{-1}(J_{\nu}) \cap \Phi[T]$ becomes the subset $J_{\mu}$ for $\mu = (\mu_1, \mu_2)\vdash (n-2)$ via the identification in~\eqref{eq: gl(n-2) roots identification}.
\end{lemma} 

\begin{proof} To prove (1), recall that $w_{\nu, \beta}^{-1}(\nu_1)=a$ and $w_{\nu, \beta}^{-1}(\nu_1+1)=b$ by definition.  If $\alpha_i=t_i-t_{i+1}$ is adjacent to $\alpha_{\nu}$ in the Dynkin diagram then
\[
	w_{\nu,\beta}^{-1}(\alpha_i) = t_{w_{\nu,\beta}^{-1}(i)} - t_{w_{\nu, \beta}^{-1}(i+1)} = \left\{  \begin{array}{ll} t_{w_{\nu,\beta}^{-1}(\nu_1-1)} - t_{a}  & \textup{ if } i=\nu_1-1\\ t_{b} - t_{w_{\nu, \beta}^{-1}(\nu_1+2)} & \textup{ if } i=\nu_1+1 \end{array} \right..
\]
In either case, $w_{\nu, \beta}^{-1}(\alpha) \notin \Phi[T]$.

Now we prove (2). If $\alpha_i=t_i-t_{i+1}$ is not adjacent to $\alpha_{\nu}$ in the Dynkin diagram, then it is certainly the case that $w_{\nu, \beta}^{-1}(\alpha_i) \in \Phi[T]$ since $\{i, i+1\} \cap \{ \nu_1, \nu_1+1 \}=\emptyset$ implies $\{ w_{\nu, \beta}^{-1}(i), w_{\nu, \beta}^{-1}(i+1) \}\cap \{a,b\}=\emptyset$.  In addition $w_{\nu, \beta}^{-1}(\alpha_i)\in \Phi_h$ since $w_{\nu,\beta} \in \mathcal{D}_{\nu}(\beta)$ by Lemma~\ref{lem: wmin satisfies the Hess. condition}.

From (1) and (2) it follows that $\gamma \in w_{\nu,\beta}^{-1}(J_{\nu}) \cap \Phi[T]$ only if $\gamma=w_{\nu,\beta}^{-1}(\alpha)$ for some $\alpha=t_i-t_{i+1}\in J_{\nu}$ such that $\alpha$ is not adjacent to $\alpha_{\nu}$ in the Dynkin diagram. In this case, $\{w_{\nu, \beta}^{-1}(i), w_{\nu, \beta}^{-1}(i+1)\} \cap \{a,b\} = \emptyset$.  Furthermore, since the entries in positions $n\setminus\{a,b\}$ of the one-line notation for $w_{\nu,\beta}$ increase from left to right it follows that $\phi_T(w_{\nu, \beta}^{-1}(i+1))=\phi_T(w_{\nu, \beta}^{-1}(i))+1$.  In other words, $\alpha_{\phi_T(i)}= t_{\phi_T(i)} - t_{\phi_T(i+1)}$ is a positive simple root in $\Phi_T$.  Therefore the set $w_{\nu, \beta}^{-1}(J)$ for $J\subseteq \Delta$ defined by
\begin{itemize}
\item $J=\{ \alpha_1, ..., \alpha_{\nu_1-2}, \alpha_{\nu_1+2},..., \alpha_{n-1} \}$ if $2<\nu_1<n-2$, or
\item $J=\{ \alpha_3,..., \alpha_{n-1} \}$ if $\nu_1=1$, or
\item $J=\{ \alpha_4, ..., \alpha_{n-1} \}$ if $\nu_1=2$, or
\item $J=\{ \alpha_1,..., \alpha_{n-4} \}$ if $\nu_1= n-2$, or
\item $J=\{ \alpha_1,..., \alpha_{n-3} \}$ if $\nu_1= n-1$
\end{itemize} 
corresponds to a subset of positive simple roots in $\Phi_T$.  Finally, since~\eqref{eq: gl(n-2) roots identification} is an isomorphism of root systems, it preserves the addition of roots.  Thus if two simple roots in $J$ are adjacent in the Dynkin diagram for $\mathfrak{gl}(n,\C)$, then their images under $w_{\nu, \beta}^{-1}$ and~\eqref{eq: gl(n-2) roots identification} are also adjacent in the Dynkin diagram for $\mathfrak{gl}(n-2,\C)$. It follows that $w_{\nu, \beta}^{-1}(J) = w_{\nu,\beta}^{-1}(J_{\nu}) \cap \Phi[T]$ is the subset $J_{\mu}$ via the identification given in~\eqref{eq: gl(n-2) roots identification}.
\end{proof} 

\begin{example} \label{ex: J-nu to J-mu example}  Using the same set-up from Examples~\ref{ex: w to tau correspondence} and \ref{ex: Phi to Phi[T] correspondence}, we have $J_{\nu} =\{ \alpha_1, \alpha_2, \alpha_4 \}$ since $\nu=(3,2)$.  We track what happens to each of these simple roots under the action of $w_{\nu,\beta}^{-1}$ (where $w_{\nu,\beta} = \left[ 1\hsm \mathbf{4} \hsm 2 \hsm 5 \hsm \mathbf{3} \right] $) and subsequent identification with simple roots in $\mathfrak{gl}(n-2,\C)$ below.
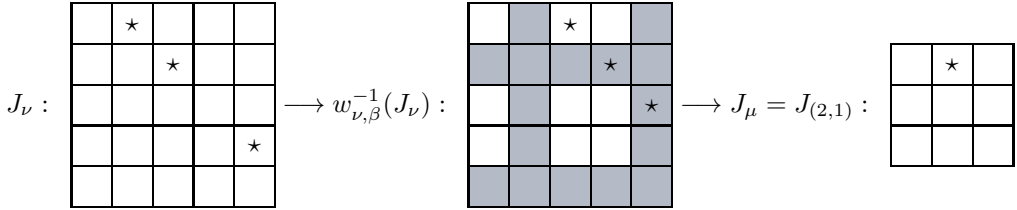
\[
\ytableausetup{centertableaux}
J_{\nu}: \;\; \begin{ytableau} \empty & \star & \empty & \empty & \empty \\ \empty & \empty & \star & \empty & \empty \\  \empty & \empty & \empty & \empty & \empty \\ \empty & \empty & \empty & \empty & \star \\  \empty & \empty & \empty & \empty & \empty    \end{ytableau}
\longrightarrow
w_{\nu,\beta}^{-1}(J_{\nu}) : \;\; \begin{ytableau} \empty & *(ltgrey)\empty & \star & \empty & *(ltgrey)\empty \\  *(ltgrey)\empty & *(ltgrey)\empty & *(ltgrey)\empty & *(ltgrey)\star & *(ltgrey)\empty \\  \empty & *(ltgrey)\empty & \empty & \empty & *(ltgrey) \star \\ \empty & *(ltgrey)\empty & \empty & \empty & *(ltgrey)\empty \\  *(ltgrey)\empty & *(ltgrey)\empty & *(ltgrey)\empty & *(ltgrey)\empty & *(ltgrey)\empty    \end{ytableau}
\longrightarrow
J_{\mu} = J_{(2,1)} : \;\; \begin{ytableau} \empty & \star & \empty  \\ \empty & \empty & \empty  \\  \empty & \empty & \empty      \end{ytableau}
\]
In addition, we consider the restriction of $\Psi_{\nu,\beta}$ to the set $\mathcal{D}_{\nu}(\beta)$.  The table from Example~\ref{ex: w to tau correspondence} becomes
\begin{eqnarray*}
\begin{array} {c|c|c}
w\in \mathcal{D}_{\nu}(\beta) & \tau \in \Stab(a,b) & x_{\tau} \in \Symm_3 \\ \hline
\left[ 1\hsm \mathbf{4} \hsm 2 \hsm 5 \hsm \mathbf{3} \right] & \left[ 1\hsm \mathbf{2} \hsm 3 \hsm 4 \hsm \mathbf{5} \right] & \left[1 \hsm 2 \hsm 3\right] \\
\left[ 2\hsm \mathbf{4} \hsm 1 \hsm 5 \hsm \mathbf{3} \right] & \left[ 3\hsm \mathbf{2} \hsm 1 \hsm 4 \hsm \mathbf{5} \right] & \left[2 \hsm 1 \hsm 3\right] \\
\left[ 1\hsm \mathbf{4} \hsm 5 \hsm 2 \hsm \mathbf{3} \right] & \left[ 1\hsm \mathbf{2} \hsm 4 \hsm 3 \hsm \mathbf{5} \right] & \left[1 \hsm 3 \hsm 2\right] \\
\left[ 5\hsm \mathbf{4} \hsm 1 \hsm 2 \hsm \mathbf{3} \right] & \left[ 4\hsm \mathbf{2} \hsm 1 \hsm 3 \hsm \mathbf{5} \right] & \left[3 \hsm 1 \hsm 2\right] \\
\left[ 5\hsm \mathbf{4} \hsm 2 \hsm 1 \hsm \mathbf{3} \right] & \left[ 4\hsm \mathbf{2} \hsm 3 \hsm 1 \hsm \mathbf{5} \right] & \left[3 \hsm 2 \hsm 1\right]
\end{array}
\end{eqnarray*}
since $w=\left( 2\hsm \mathbf{4} \hsm 5 \hsm 1 \hsm \mathbf{3} \right)$ does not satisfy the condition that $w^{-1}(J_{\nu}) \subseteq \Phi_h$.  In addition, we note that the image of $\mathcal{D}_{\nu}(\beta)$ under $\Psi_{\nu, \beta}$ consists of those $x \in \Symm_3$ such that $x^{-1}(J_{\mu}) \subseteq \Phi_{h_T}$.  Our next lemma proves that this is true in general.
\end{example}

The next lemma identifies each $\mathcal{D}_{\nu}(\beta)$ with the set of permutations satisfying the ``Hessenberg conditions'' for $\Hess(\mathsf{X}_{\mu}, h_T)$, the smaller Hessenberg variety associated to $h_T$ and $\mu \vdash (n-2)$.

\begin{lemma}\label{lem:equivalence of Hess conditions} The following are equivalent: 
\begin{enumerate} 
\item $\tau^{-1}(w_{\nu, \beta}^{-1}(J_\nu)) \subseteq \Phi_h$ and 
\item $\tau^{-1}(w_{\nu, \beta}^{-1}(J_\nu) \cap \Phi[T]) \subseteq \Phi_h[T]$. 
\end{enumerate} 
In particular, the map $\Psi_{\nu, \beta}$ defined in Lemma~\ref{claim: Stab map} restricts to a bijection
\[
	{\Psi}_{\nu, \beta}: \mathcal{D}_{\nu}(\beta) \to \{x \in \Symm_{n-2} \hsm \vert\hsm x^{-1}(J_{\mu}) \subseteq \Phi_{h_T} \}
\]
where $\beta=\beta_T$. 
\end{lemma}

\begin{proof}
The fact that (1) implies (2) is clear since $\tau$ preserves $\Phi[T]$ and $\Phi_h[T]= \Phi_h\cap \Phi[T]$. 
To prove that (2) implies (1), assume $\tau^{-1}(w_{\nu, \beta}^{-1}(J_\nu) \cap \Phi[T]) \subseteq \Phi_h[T]$ and let $\gamma\in\tau^{-1}(w_{\nu, \beta}^{-1}(J_\nu))$ so $\tau(\gamma) = w_{\nu, \beta}^{-1}(\alpha)$ for some $\alpha\in J_{\nu}$. Either $\alpha$ is adjacent to $\alpha_{\nu}$ in the Dynkin diagram or it is not.  We consider each case.

If $\alpha$ is not adjacent to $\alpha_{\nu}$ in the Dynkin diagram, then statement (2) of Lemma~\ref{lem: wmin conjugation} implies that $w_{\nu, \beta}^{-1}(\alpha) \in w_{\nu, \beta}^{-1}(J_\nu) \cap \Phi[T]$ so $\gamma= \tau^{-1}(w_{\nu, \beta}^{-1}(\alpha)) \in\tau^{-1}(w_{\nu, \beta}^{-1}(J_\nu) \cap \Phi[T]) \subseteq \Phi_h[T] \subseteq \Phi_h$.  Now assume $\alpha$ is adjacent to $\alpha_{\nu}$ in the Dynkin diagram, and for the sake of  contradiction suppose that $\gamma = \tau^{-1}(w_{\nu, \beta}^{-1}(\alpha)) \in I_h$.  Now $\alpha+\alpha_{\nu} \in \Phi$ and $(w_{\nu, \beta} \tau)^{-1}(\alpha), (w_{\nu, \beta} \tau)^{-1}(\alpha_{\nu})$ are both elements of $I_h$.  Their sum must also be an element of $I_h$, contradicting the assumption that $I_h$ is abelian.  

The last assertion of the Lemma now follows directly from the last assertion of Lemma~\ref{lem: wmin conjugation}, the identifications in~\eqref{eq: tau to x-tau}, \eqref{eq: gl(n-2) roots identification}, and Remark~\ref{rem: inversion set compatibility}.
\end{proof}

We are now ready to prove an ungraded version of Proposition~\ref{prop: graded case}, i.e., for $t=1$.

\begin{proposition} 
Under the notation and assumptions of Proposition~\ref{prop:induction step}, let $T \in SK_2(\Gamma_h)$ and let $\beta=\beta_T$ as above. Then 
\begin{equation}\label{eq: DnubetaT and dim}
\lvert \mathcal{D}_{\nu}(\beta)  \rvert = \dim H^*(\Hess(\mathsf{X}_{\mu}, h_T))
\end{equation}
where $\mu=(\mu_1, \mu_2) \vdash (n-2)$. In particular, the ungraded version of Proposition~\ref{prop: graded case} holds, i.e., 
\[
\lvert \mathcal{D}_{\nu} \rvert = \sum_{T \in SK_2(\Gamma_h)} \dim H^*(\Hess(\mathsf{X}_{\mu}, h_T)).
\]
\end{proposition}

\begin{proof} 
In the setting of the proposition, $\mu=(\mu_1,\mu_2)$ is a partition of $n-2$. Recall from Lemma~\ref{lem: Betti numbers of regular Hess} that 
\begin{eqnarray} \label{eq: Hess dimn}
\dim H^*(\Hess(\mathsf{X}_{\mu}, h_T)) = \lvert \{ x \in \Symm_{n-2} \hsm \vert \hsm x^{-1}(J_\mu) \subseteq \Phi_{h_T} \} \rvert. 
\end{eqnarray}
Thus to prove~\eqref{eq: DnubetaT and dim} it suffices to show that 
\begin{equation}\label{eq: DnubetaT and tau} 
\lvert \mathcal{D}_{\nu}(\beta_T) \rvert = \lvert \{ x \in \Symm_{n-2} \hsm \vert \hsm x^{-1}(J_\mu) \subseteq \Phi_{h_T} \} \rvert
\end{equation}
but Lemma~\ref{lem:equivalence of Hess conditions} establishes a bijective correspondence between the two sets in question, so~\eqref{eq: DnubetaT and tau} holds. Now the decomposition 
\[
\mathcal{D}_\nu = \bigsqcup_{T \in SK_2(\Gamma_h)} \mathcal{D}_{\nu}(\beta_T)
\]
from~\eqref{eq: Dnu decomposition} immediately yields 
\[
\lvert \mathcal{D}_{\nu} \rvert = \sum_{T \in SK_2(\Gamma_h)} \lvert \mathcal{D}_{\nu}(\beta_T) \rvert = \sum_{T \in SK_2(\Gamma_h)} \dim H^*(\Hess(\mathsf{X}_{\mu}, h_T))
\]
where the second equality uses~\eqref{eq: Hess dimn} and~\eqref{eq: DnubetaT and tau} above. This concludes the proof. 
\end{proof}

We now turn our attention to the graded case, namely Proposition \ref{prop: graded case}.  The difficulty is that the inversion sets $N^-(w)\cap \Phi_h^- \cong \inv_h(w)$ for $w\in \mathcal{D}_{\nu}(\beta)$ are not related to the inversions sets $N^-(\tau)\cap \Phi_h^-[T] \cong \inv_{h_T}(x_{\tau})$ by a simple formula.  To remedy this, we need to shift each $w\in \mathcal{D}_{\nu}(\beta)$ by an appropriate translation, namely $w\mapsto \sigma_{\nu}w$ where $\sigma_{\nu}$ is the permutation defined below.

\begin{definition} \label{def: sigma-nu}
Define $\sigma_{\nu} \in \Symm_n$ by the following conditions: 
\begin{enumerate}
\item $\sigma_{\nu}(\nu_1)=1$ and $\sigma_{\nu}(\nu_1+1)=2$, and
\item the remaining entries in the one-line notation of $\sigma_{\nu}$ list the integers $[n]\setminus \{1,2\}$ in increasing order from left to right, i.e.
\[
	\sigma_{\nu}(i) = \left\{  \begin{array}{ll} i+2 & \textup{if } i<\nu_1\\
							i & \textup{if } i>\nu_1+1 \end{array}\right. . 
\]
\end{enumerate}
\end{definition}

Note that $\sigma_{\nu}$ is uniquely determined by the value of $\nu_1$.

\begin{example}\label{ex: sigma-nu translation}
Using the same set-up as in Examples~\ref{ex: w to tau correspondence} and \ref{ex: Phi to Phi[T] correspondence}, recall that $n=5$, $\nu=(3,2)$, and $\beta=t_5-t_2$.  Since $\nu_1=3$ and $\nu_1+1=4$ we get
\[
\sigma_{\nu} = \left[ 3 \hsm 4 \hsm \mathbf{1} \hsm \mathbf{2} \hsm 5  \right]
\]	
where condition (1) in Definition~\ref{def: sigma-nu} determines the entries in bold and condition (2) determines the rest.  Consider the translation $w\mapsto \sigma_{\nu}w$ for each $w\in \mathcal{D}_{\nu}(\beta)$, displayed in the table below.
\begin{eqnarray*}
\begin{array} {c|c}
w\in \mathcal{D}_{\nu}(\beta) & \sigma_{\nu}w \in \Symm_5 \\ \hline
\left[ 1\hsm \mathbf{4} \hsm 2 \hsm 5 \hsm \mathbf{3} \right] & \left[3\hsm \mathbf{2} \hsm 4 \hsm 5 \hsm \mathbf{1}\right] \\
\left[ 2\hsm \mathbf{4} \hsm 1 \hsm 5 \hsm \mathbf{3} \right] & \left[4\hsm \mathbf{2} \hsm 3 \hsm 5 \hsm \mathbf{1}\right] \\
\left[ 1\hsm \mathbf{4} \hsm 5 \hsm 2 \hsm \mathbf{3} \right] &\left[3\hsm \mathbf{2} \hsm 5 \hsm 4 \hsm \mathbf{1}\right] \\
\left[ 5\hsm \mathbf{4} \hsm 1 \hsm 2 \hsm \mathbf{3} \right] & \left[5\hsm \mathbf{2} \hsm 3 \hsm 4 \hsm \mathbf{1}\right]\\
\left[ 5\hsm \mathbf{4} \hsm 2 \hsm 1 \hsm \mathbf{3} \right] &\left[5\hsm \mathbf{2} \hsm 4 \hsm 3 \hsm \mathbf{1}\right]
\end{array}
\end{eqnarray*}
Note that translation by $\sigma_{\nu}$ sends $4\mapsto 2$ and $3\mapsto 1$ in the one-line notation for $w$, but the rest of the entries of $w$ remain in the same relative order in the one-line notation for $\sigma_{\nu}w$ as they were in the one-line notation of $w$.
\end{example}

The next lemma shows that translating the set $\mathcal{D}_{\nu}(\beta)$ by $\sigma_{\nu}$ does not change the inversions of $w$ that are also elements of $\Phi^-[T]$.

\begin{lemma} \label{lem: sigma-nu translation}
Suppose $\sigma_{\nu}$ is defined as above.  For all $w\in \mathcal{D}_{\nu}(\beta)$ we have 
\[
N^-(\sigma_\nu w) \cap \Phi^-[T] = N^-(w) \cap \Phi^-[T].
\]
\end{lemma}
\begin{proof}
Recall that $\Phi^-[T] = \{ t_i - t_j\in \Phi^- \hsm \vert \hsm \{i,j\}\cap \{a,b\} = \emptyset  \}$. By definition, $w(a)=\nu_1$ and $w(b) = \nu_1+1$ so $\sigma_{\nu}w(a) = 1$ and $\sigma_{\nu}w(b)=2$.  Thus the $a$-th and $b$-th entry of $\sigma_{\nu}w$ in one-line notation is determined.  By Condition (2) in Definition~\ref{def: sigma-nu}, $\sigma_{\nu}$ preserves the relative order of the values in the one-line notation of $w$ which are not in positions $a$ or $b$.  It follows that for $(i,j)$ with $\{i,j\} \cap \{a,b\} = \emptyset$ and $i>j$, we have $(i,j) \in \inv(w)$ if and only if $(i,j) \in \inv(\sigma_\nu w)$.
\end{proof}

The next lemma relates the grading computation for $\sigma_\nu w$ to the grading computation for $\tau$, up to a translation by $\deg(T)$. This explains why it is useful to introduce the translation by $\sigma_{\nu}$. One of the key points in the proof is that the LHS of~\eqref{eq:shift by degT} can be related to the edges of $\Gamma_h$ which contribute to the computation of $\deg(T)$. 

\begin{lemma} \label{lem: degree shifts}  
Let $w\in \mathcal{D}_{\nu}(\beta)$ for some $\beta= \beta_T\in I_h$ corresponding to $T\in \SK_2(\Gamma_h)$.  Let $w=w_{\nu, \beta}\tau$ be the decomposition of $w$ given in Lemma~\ref{claim: Stab map} for a unique $\tau\in \Stab(a,b)$. Then
\begin{equation}\label{eq:shift by degT} 
	|N^-(\sigma_{\nu}w)\cap \Phi_h^-| = \deg(T) + |N^-(\tau)\cap \Phi_h^-[T]|.
\end{equation} 
\end{lemma}

\begin{proof}  Since
$\Phi^- = (\Phi^- \setminus \Phi^-[T]) \sqcup \Phi^-[T]$ we also have $\Phi_h^- = (\Phi_h^- \cap (\Phi^- \setminus \Phi^-[T])) \sqcup (\Phi_h^- \cap \Phi^-[T])$.  Since $\Phi_h^- \cap \Phi^-[T]$ is the set $\Phi_h^-[T]$ by definition, we conclude
\[
	|N^-(\sigma_{\nu}w)\cap \Phi_h^-| = |N^-(\sigma_{\nu}w) \cap \Phi_h^-\cap(\Phi^- \setminus \Phi^-[T])| + |N^-(\sigma_{\nu}w)\cap \Phi_h^-[T]|. 
\] 
Hence to prove~\eqref{eq:shift by degT}
it suffices to prove that 
\begin{eqnarray}\label{eq: eq1}
|N^-(\sigma_{\nu}w) \cap \Phi_h^-\cap(\Phi^- \setminus \Phi^-[T])| = \deg(T) 
\end{eqnarray}
and 
\begin{eqnarray}\label{eq: eq2}
N^-(\sigma_{\nu}w)\cap \Phi_h^-[T] = N^-(\tau)\cap \Phi_h^-[T].
\end{eqnarray}

We first prove~\eqref{eq: eq1}.  By definition, $\Phi^- \setminus \Phi^-[T] = \{ t_i - t_j \in \Phi^- \hsm \vert \hsm  \{i,j\}\cap \{a,b\}\neq \emptyset \}$.  Since $w \in \mathcal{D}_{\nu}(\beta)$, we know $w(a)=\nu_1$ and $w(b) = \nu_1+1$ by~\eqref{eq:wa and wb}, and by construction of $\sigma_\nu$ this implies $\sigma_{\nu} w (a)= 1$ and $\sigma_{\nu} w(b) = 2$. It follows that $1$ is in the $a$-th position of the one-line notation for $\sigma_{\nu} w$ and $2$ is in the $b$-th position.  Using the identification $N^-(\sigma_\nu w) \cong \mathrm{inv}(\sigma_\nu w)$, we obtain  
\begin{eqnarray*}
	N^-(\sigma_{\nu}w)  \cap (\Phi^- \setminus \Phi^-[T]) = \{ (b,j) \hsm \vert \hsm 1\leq j < b  \}  \cup \{ (a,j) \hsm \vert \hsm 1\leq j < a  \}
\end{eqnarray*}
and therefore 
\begin{eqnarray*}
	N^-(\sigma_{\nu}w)\cap \Phi_h^- \cap (\Phi^- \setminus \Phi^-[T]) &=& \{ (b,j) \hsm \vert \hsm 1\leq j < b \textup{ and } b \leq h(j)  \}  \cup \\
	&&\quad\quad\quad\quad \{ (a,j) \hsm \vert \hsm 1\leq j < a \textup{ and } a\leq h(j) \}.
\end{eqnarray*}
Since $T=\{a,b\}$, the elements in the sets above correspond to edges of $\Gamma_h$ that are incident to the vertices in $T$ and which must be oriented to the right in order for $a$ and $b$ to be sinks.  Thus, \eqref{eq: eq1} now follows immediately from Lemma~\ref{lem: deg(T) property1}.

Next, in order to prove~\eqref{eq: eq2} we note that $N^-(\sigma_\nu w) \cap \Phi^-[T] = N^-(w) \cap \Phi^-[T]$ by Lemma~\ref{lem: sigma-nu translation}. Intersecting both sides with $\Phi_h^-$ we obtain 
\begin{equation}\label{eq:intersect with PhihT}
N^-(\sigma_\nu w) \cap \Phi_h^-[T] = N^-(w) \cap \Phi_h^-[T].
\end{equation}
 Next we claim 
\begin{equation}\label{eq:w and tau}
N^-(w) \cap \Phi_h^-[T] = N^-(\tau) \cap \Phi_h^-[T].
\end{equation}
As in the argument above, to see this it suffices to prove $N^-(w) \cap \Phi^-[T] = N^-(\tau) \cap \Phi^-[T]$, since~\eqref{eq:w and tau} follows by intersecting both sides with $\Phi_h^-$. 

Suppose $t_i-t_j \in N^-(w) \cap \Phi^-[T]$. We wish to show $t_i-t_j \in N^-(\tau) \cap \Phi^-[T]$. By assumption we know $i>j$ and $\{i,j\} \cap \{a,b\} = \emptyset$ and $w_{\nu,\beta} \tau(i)=w(i) < w(j) =w_{\nu,\beta} \tau(j)$. Suppose in order to obtain a contradiction that $\tau(i) > \tau(j)$. Then $(\tau(i), \tau(j)) \in \inv(w_{\nu,\beta})$ by the above. Moreover, since $\tau \in \Stab(a,b)$ and $\{i,j\} \cap \{a,b\} = \emptyset$, we have $\{\tau(i), \tau(j)\} \cap \{a,b\} = \emptyset$ also. This contradicts part (1) of Lemma~\ref{lem: wmin satisfies the Hess. condition}. Thus $\tau(i) < \tau(j)$, or equivalently $t_i-t_j \in N^-(\tau) \cap \Phi^-[T]$ as desired. Conversely, suppose $t_i-t_j \in N^-(\tau) \cap \Phi^-[T]$. 
 Then $\tau(i)<\tau(j)$ and $\{\tau(i), \tau(j)\} \cap \{(a,b)\}= \emptyset$ and $w_{\nu, \beta}\tau(i)<w_{\nu,\beta}\tau(j)$ by condition (2) of Definition~\ref{def: w-nu-beta}. Hence $t_i-t_j\in N^-(w)\cap \Phi[T]$.  From the above it follows that $N^-(w)\cap \Phi[T] = N^-(\tau)\cap \Phi[T] $ and we obtain ~\eqref{eq:w and tau}. Now~\eqref{eq: eq2} follows from~\eqref{eq:w and tau} and~\eqref{eq:intersect with PhihT}.
\end{proof}

\begin{example}  
We confirm the results of Lemmas~\ref{lem: sigma-nu translation} and~\ref{lem: degree shifts} for $n=5$, $\nu=(3,2)$, and $\beta= t_5-t_2$.  This is the case considered in Example~\ref{ex: sigma-nu translation}, where $h=(3,4,5,5,5)$.  The table below displays each $w\in \mathcal{D}_{\nu}(\beta)$ and computes $N^-(w)\cap \Phi_h^-[T]$ and $N^-(w)\cap \Phi_h^- \cap (\Phi^- \setminus \Phi^-[T])$.  Here we use the identification $\inv(w)\cong N^-(w)$.
\begin{eqnarray*}
\begin{array} {c|c|c}
w\in \mathcal{D}_{\nu}(\beta) & N^-(w)\cap \Phi^-[T] & N^-(w)\cap \Phi_h^- \cap (\Phi^- \setminus \Phi^-[T]) \\ \hline
\left[ 1\hsm \mathbf{4} \hsm 2 \hsm 5 \hsm \mathbf{3} \right] & \emptyset & \{(3,2), (5,4)\} \\
\left[ 2\hsm \mathbf{4} \hsm 1 \hsm 5 \hsm \mathbf{3} \right] & \{(3,1)\} & \{ (3,2), (5,4) \}  \\
\left[ 1\hsm \mathbf{4} \hsm 5 \hsm 2 \hsm \mathbf{3} \right] & \{ (4,3) \} & \{ (4,2), (5,3) \} \\
\left[ 5\hsm \mathbf{4} \hsm 1 \hsm 2 \hsm \mathbf{3} \right] & \{ (3,1),(4,1) \} & \{ (2,1), (3,2), (4,2) \} \\
\left[ 5\hsm \mathbf{4} \hsm 2 \hsm 1 \hsm \mathbf{3} \right] & \{ (3,1),(4,1),(4,3)\} & \{ (2,1), (3,2),(4,2)\}
\end{array}
\end{eqnarray*}
Now we do the same computation for $\sigma_{\nu}w$.
\begin{eqnarray*}
\begin{array} {c|c|c|c}
w\in \mathcal{D}_{\nu}(\beta) & \sigma_{\nu}w & N^-(\sigma_{\nu}w)\cap \Phi^-[T] &N^-(\sigma_{\nu}w)\cap \Phi_h^- \cap (\Phi^- \setminus \Phi^-[T])   \\ \hline
\left[ 1\hsm \mathbf{4} \hsm 2 \hsm 5 \hsm \mathbf{3} \right] & \left[3\hsm \mathbf{2} \hsm 4 \hsm 5 \hsm \mathbf{1}\right]& \emptyset & \{ (2,1), (5,3), (5,4) \} \\
\left[ 2\hsm \mathbf{4} \hsm 1 \hsm 5 \hsm \mathbf{3} \right] & \left[4\hsm \mathbf{2} \hsm 3 \hsm 5 \hsm \mathbf{1}\right] & \{ (3,1) \} & \{ (2,1), (5,3), (5,4) \} \\
\left[ 1\hsm \mathbf{4} \hsm 5 \hsm 2 \hsm \mathbf{3} \right] & \left[3\hsm \mathbf{2} \hsm 5 \hsm 4 \hsm \mathbf{1}\right] & \{ (4,3) \} &  \{ (2,1), (5,3), (5,4) \} \\
\left[ 5\hsm \mathbf{4} \hsm 1 \hsm 2 \hsm \mathbf{3} \right] & \left[5\hsm \mathbf{2} \hsm 3 \hsm 4 \hsm \mathbf{1}\right] & \{ (3,1),(4,1) \} &  \{ (2,1), (5,3), (5,4) \}\\
\left[ 5\hsm \mathbf{4} \hsm 2 \hsm 1 \hsm \mathbf{3} \right] & \left[3\hsm \mathbf{2} \hsm 4 \hsm 3 \hsm \mathbf{1}\right] & \{ (3,1), (4,1), (4,3) \} &  \{ (2,1), (5,3), (5,4) \}
\end{array}
\end{eqnarray*}
The information in the tables above confirms the results of Lemma~\ref{lem: sigma-nu translation}. The graph below shows the orientation $\omega$ of $\Gamma_h$ with the property that $\sk(\omega)=\{2,5\}$ and $\asc(\omega)=\deg(T)$.
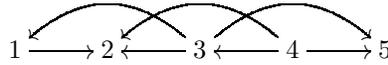
\vspace*{.15in}
\[
\xymatrix{1 \ar[r]   & {2}  & 3  \ar[l] \ar@/_1.5pc/[ll]  \ar@/^1.5pc/[rr] &  4  \ar[l] \ar[r] \ar@/_1.5pc/[ll]   & {5} }
\]
From the graph, we see $\deg(T)=3$, so the table above also confirms $|N^-(\sigma_{\nu}w)\cap \Phi_h^- \cap (\Phi^- \setminus \Phi^-[T])| = \deg(T)$ for all $w\in \mathcal{D}_{\nu}(\beta)$.  Finally, we consider each $\tau\in \Stab(2,5)$ such that $w=w_{\nu,\beta}\tau$  and compute both $N^-(\tau)\cap \Phi_h^-[T]$ and $N^-(\sigma_{\nu} w)\cap \Phi_{h}^-[T]$.
\begin{eqnarray*}
\begin{array} {c|c|c|c}
w\in \mathcal{D}_{\nu}(\beta) & \tau \in \Stab(2,5) & N^-(\tau)\cap \Phi_h^-[T] & N^-(\sigma_{\nu}w)\cap \Phi_h^-[T] \\ \hline
\left[ 1\hsm \mathbf{4} \hsm 2 \hsm 5 \hsm \mathbf{3} \right] &\left[ 1\hsm \mathbf{2} \hsm 3 \hsm 4 \hsm \mathbf{5} \right] & \emptyset & \emptyset \\
\left[ 2\hsm \mathbf{4} \hsm 1 \hsm 5 \hsm \mathbf{3} \right] & \left[ 3\hsm \mathbf{2} \hsm 1 \hsm 4 \hsm \mathbf{5} \right] & \{ (3,1) \} &  \{ (3,1) \} \\
\left[ 1\hsm \mathbf{4} \hsm 5 \hsm 2 \hsm \mathbf{3} \right] & \left[ 1\hsm \mathbf{2} \hsm 4 \hsm 3 \hsm \mathbf{5} \right] & \{ (4,3) \}  & \{ (4,3) \} \\
\left[ 5\hsm \mathbf{4} \hsm 1 \hsm 2 \hsm \mathbf{3} \right] &  \left[ 4\hsm \mathbf{2} \hsm 1 \hsm 3 \hsm \mathbf{5} \right] & \{ (3,1) \}& \{ (3,1) \} \\
\left[ 5\hsm \mathbf{4} \hsm 2 \hsm 1 \hsm \mathbf{3} \right] &  \left[ 4\hsm \mathbf{2} \hsm 3 \hsm 1 \hsm \mathbf{5} \right]  & \{ (3,1), (4,3) \} & \{ (3,1), (4,3) \}
\end{array}
\end{eqnarray*}
This table confirms that $N^-(w)\cap \Phi_h^-[T] = N^-(\sigma_{\nu} w)\cap \Phi_h^-[T] = N^-(\tau)\cap \Phi_h^-[T]$ as shown in the proof of Lemma~\ref{lem: degree shifts}.
\end{example}

\begin{corollary}\label{cor: sigma shifts}
Let $\nu=(\nu_1, \nu_2) = (\mu_1+1, \mu_2+1)\vdash n$ and $\sigma_{\nu}$ be defined as above.  Then
\[
	\sum_{\sigma_{\nu} w \,:\, w\in \mathcal{D}_{\nu}} t^{2|N^-(\sigma_{\nu} w)\cap \Phi_h^-|} = \sum_{T\in \SK_2(\Gamma_h)} t^{2\deg(T)}\, P(\Hess(\mathsf{X}_{\mu},h_T), t).
\]
\end{corollary}

\begin{proof}  Under the identifications in~\eqref{eq: tau to x-tau} and \eqref{eq: gl(n-2) roots identification}, $\tau$ becomes $x_{\tau}$ and $\Phi_h^-[T] $ becomes $\Phi_{h_T}^-$.  By Remark~\ref{rem: inversion set compatibility} we have that $|N^-(\tau)\cap \Phi_h^-[T]| = |N^-(x_{\tau}) \cap \Phi_{h_T}^-|$.  
Therefore
\begin{eqnarray*}
	\sum_{\sigma_{\nu} w \,:\, w\in \mathcal{D}_{\nu}} t^{2|N^-(\sigma_{\nu} w)\cap \Phi_h^-|} &=& \sum_{T\in \SK_2(\Gamma_h)} \; \sum_{\sigma_{\nu}w \,:\,w\in \mathcal{D}_{\nu}(\beta_T)} t^{2|N^-(\sigma_{\nu} w)\cap \Phi_h^-|}\;\;  \textup{ since } \mathcal{D}_{\nu} = \bigsqcup_{T \in SK_2(\Gamma_h)} \mathcal{D}_{\nu}(\beta_T) \\
	& = &  \sum_{T \in SK_2(\Gamma_h)} \;\sum_{\substack{\sigma_\nu w_{\nu,\beta} \tau\, : \,\tau \in \Stab(a,b), \\  x_{\tau}^{-1}(J_\mu) \subseteq \Phi_{h_T}}} t^{2 \lvert N^-(\sigma_\nu w) \cap \Phi_h^- \rvert} \;\;  \textup{ by Lemma~\ref{lem:equivalence of Hess conditions} } \\ 
&=& \sum_{T\in \SK_2(\Gamma_h)} t^{2\deg(T)} \sum_{\substack{x_{\tau} \in \Symm_{n-2} \\ x_{\tau}^{-1}(J_{\mu}) \subseteq \Phi_{h_T}}} t^{2|N^-(x_{\tau}) \cap \Phi_{h_T}^- |} \;\; \textup{ by Lemma~\ref{lem: degree shifts} } \\
&=&  \sum_{T\in \SK_2(\Gamma_h)} t^{2\deg(T)} P(\Hess(\mathsf{X}_{\mu}, h_T), t) \;\; \textup{ by Lemma~\ref{lem: Betti numbers of regular Hess} } \\
\end{eqnarray*}
as desired. 
\end{proof}

We now consider the special case $\nu=(n-1, 1)$, which turns out to be critical.  Note that $\mathsf{X}_{(1,n-1)}$ and $\mathsf{X}_{\nu}$ are conjugate.  The basic trick in our proof below is to remember that the isomorphism class of Hessenberg varieties is preserved under conjugation, so in particular the Hessenberg varieties $\Hess(\mathsf{X}_{\nu},h)$ and $\Hess(\mathsf{X}_{(1,n-1)}, h)$ are isomorphic and hence have the same Poincar\'e polynomial.

\begin{proposition} \label{cor: subreg case}
Suppose $\nu=(n-1,1) \vdash n$. Then 
\begin{equation}\label{eq: inverted partition case}
\sum_{w\in \mathcal{D}_{\nu}} t^{2|N^-(w)\cap \Phi_h^-|} = \sum_{\sigma_{\nu} w \,:\, w\in \mathcal{D}_{\nu}} t^{2|N^-(\sigma_{\nu} w)\cap \Phi_h^-|}.	
\end{equation} 
\end{proposition}

\begin{proof} 
Consider the two partitions $\nu = (n-1,1)$ and $\nu'=(1,n-1)$ of $n$. As already noted above, $\Hess(\mathsf{X}_{\nu},h) \cong \Hess(\mathsf{X}_{\nu'},h)$ and hence the two Hessenberg varieties have the same Poincar\'e polynomial. Notice that in both cases, the partition $\mu$ of $n-2$ corresponding to the compositions defined by $\nu=(\nu_1, \nu_2) = (\mu_1+1, \mu_2+1)$ and $\nu' = (\nu_1',\nu_2') = (\mu_1+1, \mu_2+1)$ is the trivial partition of $n-2$ and in particular the $\mathsf{X}_\mu$ which appears in the RHS of both Proposition~\ref{prop: graded case} and Corollary~\ref{cor: sigma shifts} for both cases $\nu=(n-1,1)$ and $\nu'=(1,n-1)$, may be taken to be $\mathsf{N}'$, the regular nilpotent element of $\mathfrak{gl}(n-2,\C)$. Also observe that $\sigma_{\nu'}$ is the identity permutation by definition, since $\nu_1'=1$ in this case. 

From the above considerations we obtain: 
\begin{eqnarray*}\begin{split}
	P(\Hess(\mathsf{X}_{\nu},h),t)&=P(\Hess(\mathsf{X}_{\nu'},h),t)\\
	& =P(\Hess(\mathsf{N},h),t) + \sum_{w \,:\, w\in \mathcal{D}_{\nu'}} t^{2|N^-(w)\cap \Phi_h^-|}   \\ 
	& =P(\Hess(\mathsf{N},h),t) + \sum_{\sigma_{\nu'} w \,:\, w\in \mathcal{D}_{\nu'}} t^{2|N^-(\sigma_{\nu'} w)\cap \Phi_h^-|} \;\; \textup{ since $\sigma_{\nu'} = e$ } \\ 
	&= P(\Hess(\mathsf{N},h),t) + \sum_{T\in \SK_2(\Gamma_h)} t^{2\deg(T)}\, P(\Hess(\mathsf{N}',h_T), t) \\
	&= P(\Hess(\mathsf{N},h),t) + \sum_{\sigma_{\nu} w \,:\, w\in \mathcal{D}_{\nu}} t^{2|N^-(\sigma_{\nu} w)\cap \Phi_h^-|} \\ 
\end{split}\end{eqnarray*}
where the second equality is by Lemma~\ref{lem: Betti numbers of regular Hess} and~\eqref{eq: poincare polynomial decomp} applied to $\nu'=(1,n-1)$ and the fourth (respectively fifth) equality is 
by applying Corollary~\ref{cor: sigma shifts} to $\nu'=(1,n-1)$ (respectively $\nu=(n-1,1)$). 
On the other hand we also know 
\begin{equation*}\label{eq: 5.8 applied to n-1,1}
P(\Hess(\mathsf{X}_{\nu},h),t) = P(\Hess(\mathsf{N},h),t) + \sum_{w \, : \, w \in \mathcal{D}_{\nu}} t^{2|N^-(w)\cap \Phi_h^-|}
\end{equation*}
by Lemma~\ref{lem: Betti numbers of regular Hess} and~\eqref{eq: poincare polynomial decomp} applied to $\nu=(n-1,1)$. Since the RHS of both of the above equalities must be equal, the equality in~\eqref{eq: inverted partition case} follows. 
\end{proof}

The Proposition above proves a special case of our desired formula--namely, it implies that Proposition~\ref{prop: graded case} holds for $\mathsf{X}_{(n-1,1)}$ when combined with Corollary~\ref{cor: sigma shifts}.  We now reduce to this special case using calculations involving the shortest coset representative.

Let $\nu=(\nu_1, \nu_2) \vdash n$. Consider the Young subgroup $W_{\nu} := \Symm_{(\nu_1+1,1,...,1)}$ of $\Symm_n$ defined as permutations of $\{1,2,\ldots,\nu_1+1\}=[\nu_1+1]$. Note that $W_{\nu}$ is the Weyl group of $\mathfrak{gl}(\nu_1+1,\C)$ viewed as a subalgebra of $\mathfrak{gl}(n,\C)$ by identifying it with the upper left-hand $(\nu_1+1)\times (\nu_1+1)$ corner of the matrices in $\mathfrak{gl}(n,\C)$. We will need the following for our proof of Proposition \ref{prop: graded case} below (see e.g. \cite[Section 5]{Kos61}).

\begin{lemma} \label{lem: inversion set induction}
Any $w\in \Symm_n$ can be factored uniquely as $w=yz$ for some $y\in W_{\nu}$ and 
\[
z\in {^\nu W}:= \{ z\in \Symm_n \hsm\vert\hsm z^{-1}(\alpha_i)\in \Phi^+ \textup{ for all } i=1,..., \nu_1 \}.
\]
Moreover, $N^-(w) = N^-(z)\sqcup z^{-1}N^-(y)$.
\end{lemma}

The set ${^\nu W}$ is known as the \textbf{set of shortest cosets representatives} for $W_{\nu}\backslash \Symm_n$.  The factors $y$ and $z$ in the decomposition given in Lemma~\ref{lem: inversion set induction} have a straightforward interpretation in terms of the one-line notation of $w$. More specifically, we can describe the one-line notation of $z$ as follows. Suppose 
\[
w = [ w(1) \hsm w(2) \cdots \hsm w(n) ] 
\]
is the one-line notation of $w$. In order to obtain the one-line notation for $z$, we look at the entries in $w$ which lie in $\{1,2,\ldots,\nu_1+1\}$, and re-write them in increasing order from left to right. All other entries remain unchanged. The result is the one-line notation for $z$. 

\begin{example} \label{ex: shortest coset one-line}
For example, if $n=7$ and $\nu_1+1=4$ and $w = [6 \hsm \mathbf{4} \hsm \mathbf{1} \hsm 7 \hsm \mathbf{2} \hsm 5 \hsm \mathbf{3}]$ where the numbers in boldface correspond to the entries in $\{1,2,3,4\}$, by re-ordering just these entries we obtain $z = [6 \hsm \mathbf{1} \hsm \mathbf{2} \hsm 7 \hsm \mathbf{3} \hsm 5 \hsm \mathbf{4}]$. Now $y$ is simply the element of $\Symm_4 \subseteq \Symm_7$ which permutes $\{1,2,3,4\}$ to the ordering that was found in the original $w$, so in this example $y = [4 \hsm 1 \hsm 2 \hsm 3 \hsm 5 \hsm 6 \hsm 7]$ 
which can also be viewed (since $y$ stabilizes $\{5,6,7\}$) as $y=[4 \hsm 1 \hsm 2 \hsm 3] \in \Symm_4$. 
\end{example}

\begin{example}\label{example: shortest coset reps} 
Let $n=5$ and $\nu=(3,2)$ so $\nu_1=3$ and $\nu_1+1=4$ as in Example~\ref{ex: w to tau correspondence}. Then $W_\nu = \Symm_{\{1,2,3,4\}} \cong \Symm_4$. The set of shortest coset representatives for $\Symm_4 \setminus \Symm_5$ consist of the permutations $w$ for which the values of $\{1,2,3,4\}$ appear in increasing order in the one-line notation of $w$. Thus there are $5$ elements in $^{\nu} W$ in this case: 
\[
^{\nu} W = \{ [1 \hsm 2\hsm 3\hsm 4 \hsm 5], [1 \hsm 2 \hsm 3 \hsm 5 \hsm 4], [1 \hsm 2 \hsm 5 \hsm 3 \hsm 4], [1\hsm 5 \hsm 2 \hsm 3 \hsm 4], [5 \hsm 1\hsm 2\hsm 3\hsm 4] \}.
\]

\end{example}

\begin{example} 
To illustrate the decomposition $N^-(w) = N^-(z) \sqcup z^{-1}N^-(y)$ we consider the example 
$w = (6, 4, 1, 7, 2, 5, 3), z = (6, 1, 2, 7, 3, 5, 4)$ as in Example~\ref{ex: shortest coset one-line}. Then
\[
N^-(z) = \{ (2,1), (3, 1), (5,1), (6, 1), (7,1), (5,4), (6,4), (7,4), (7,6) \}
\]
and $N^-(y) = \{ (2,1), (3,1), (4,1) \}$ so
\[
z^{-1} N^-(y) = \{ (3,2), (5,2), (7,2) \}. 
\]
The reader can check that $N^-(w) = N^-(z) \sqcup z^{-1}N^-(y)$.
\end{example} 

The following result is a special case of \cite[Proposition 5.2]{Precup2013}.

\begin{lemma}\label{lemma: h_v induction}
Suppose $h: [n] \to [n]$ is a Hessenberg function and $H \subseteq \mathfrak{gl} (n,\C)$ is the associated Hessenberg space.  For every $z\in {^\nu W}$
\[
	H_z:= zH z^{-1}\cap \mathfrak{gl}(\nu_1+1,\C)
\] 
is a Hessenberg space of $\mathfrak{gl}(\nu_1+1,\C)$.
\end{lemma}

As for Lemma~\ref{lem: inversion set induction} above, there is a straightforward way to interpret the Hessenberg space $H_z$ in terms of the one-line notation for $z \in {^\nu W}$. Recall from Definition~\ref{definition:Hessenberg subspace} that $H$ is spanned by $\{E_{ij} \hsm \vert \hsm i \leq h(j) \}$, which implies $z H z^{-1} \cap \mathfrak{gl} (\nu_1+1,\C)$ is spanned by $\{E_{ij} \hsm \vert \hsm i, j \in [\nu_1+1] \textup{ and } z^{-1}(i) \leq h(z^{-1}(j)) \}$. 
Now let $h_z: [\nu_1+1]\to [\nu_1+1]$ denote the Hessenberg function corresponding to $H_z$. From the definition we obtain $\Phi_{h_z} = z\Phi_h \cap \Phi_{\nu}$ where $\Phi_{\nu}:=\{ t_i-t_j \hsm \vert \hsm 1\leq i, j \leq \nu_1+1 \} \subseteq \Phi$ is the root system of $\mathfrak{gl}(\nu_1+1,\C)$ considered as a subroot system of $\Phi$.  Similarly, $I_{h_z}= zI_h \cap \Phi_{\nu}$ is the ideal of $\Phi_\nu^-$ corresponding to $h_z$.  

We are finally ready to prove Proposition \ref{prop: graded case}.  The main idea is to decompose the set $\mathcal{D}_{\nu}$ by subdividing the elements according to their shortest coset representative.  This allows us to reduce to the case in which $y\in W_{\nu}$ and $J_{(\nu_1, \nu_2)}\cap \Phi_{\nu} = J_{(\nu_1,1)}$, the special case from Proposition~\ref{cor: subreg case}.

\begin{proof}[Proof of Proposition \ref{prop: graded case}]

We first claim that it suffices to prove that for all $\nu=(\nu_1, \nu_2) =(\mu_1+1, \mu_2+1) \vdash n$, we have
\begin{eqnarray} \label{eq: key equality}
	\sum_{w\in \mathcal{D}_{\nu}} t^{2 |N^-(w)\cap \Phi_h^-|} = \sum_{\sigma_{\nu} w \,:\, w\in \mathcal{D}_{\nu}} t^{2 |N^-(\sigma_{\nu} w)\cap \Phi_h^-|}.		
\end{eqnarray}
Indeed, given~\eqref{eq: key equality} it follows from Corollary \ref{cor: sigma shifts} that
\begin{eqnarray}\label{eq: prop claim}
	\sum_{w\in \mathcal{D}_{\nu}} t^{ 2 |N^-(w)\cap \Phi_h^-|} = \sum_{T\in \SK_2(\Gamma_h)} t^{2\deg(T)}\, P(\Hess(\mathsf{X}_{\mu},h_T), t)
\end{eqnarray}
which is the desired claim of Proposition~\ref{prop: graded case}. 
We now proceed to prove~\eqref{eq: key equality}.  

Given $w\in \mathcal{D}_{\nu}$, let $w=yz$ with $y\in W_{\nu}$ and $z\in {^\nu W}$ be the decomposition from Lemma \ref{lem: inversion set induction} and let $\sigma_\nu$ be as above. 
Note that, by definition, $\sigma_\nu \in W_\nu$. It follows that $\sigma_{\nu}w = (\sigma_\nu y)(z)$, where $\sigma_\nu y \in W_\nu$ and $z \in {^{\nu} W}$, is the decomposition of $\sigma_\nu w$ from Lemma~\ref{lem: inversion set induction} above. Thus
\[
	N^-(w)=N^-(z)\sqcup z^{-1}N^-(y) \textup{ and } N^-(\sigma_{\nu}w) = N^-(z) \sqcup z^{-1}N^-(\sigma_{\nu}y). 
\]	
In particular, 
\begin{eqnarray*}
	|N^-(w)\cap \Phi_h^-| &=& |N^-(z)\cap \Phi_h^-| + |z^{-1}N^-(y)\cap \Phi_h^-|.
\end{eqnarray*}
Next note that the action of $z$ gives a bijective correspondence between $z^{-1} N^-(y) \cap \Phi_h^-$ and 
\[
	N^-(y) \cap z \Phi_h^- = N^-(y) \cap z\Phi_h^- \cap \Phi_\nu = N^-(y) \cap \Phi^-_{h_z}
\]
where the first equality holds since $N^-(y) \subseteq \Phi_\nu$ and the second equality is by definition of $h_z$. Thus 
\begin{equation}\label{eq:Nw and Nz and Ny}
\lvert N^-(w) \cap \Phi_h^- \rvert = \lvert N^-(z) \cap \Phi_h^- \rvert + \lvert N^-(y) \cap \Phi_{h_z}^- \rvert.
\end{equation}
Similarly
$|N^-(\sigma_{\nu}w) \cap \Phi_h^-| =  |N^-(z)\cap \Phi_h^-| + |N^-(\sigma_{\nu}y)\cap \Phi_{h_z}^-|$. 

For each $z\in {^\nu W}$, define 
\[
	\mathcal{D}_{\nu, z} := \{ w\in \mathcal{D}_{\nu} \hsm \vert \hsm w=yz \textup{ for some } y\in W_{\nu}\}.  
\]
Note that $\mathcal{D}_{\nu, z}$ may be empty for some $z$. However, Lemma~\ref{lem: inversion set induction} guarantees that  
$\mathcal{D}_\nu = \bigsqcup_{z\in {^\nu W}} \mathcal{D}_{\nu,z}$.

Fix $z$ such that $\mathcal{D}_{\nu,z} \neq \emptyset$ and let $y \in \Symm_\nu$ such that $w = yz \in \mathcal{D}_{\nu,z}$. We have that $w^{-1}(J_{\nu}) \subseteq \Phi_h$ and $w^{-1}(\alpha_{\nu}) \in I_h$ if and only if  $y^{-1}(J_{\nu}) \subseteq z \Phi_h$ and $y^{-1}(\alpha_{\nu}) \in zI_h$. Next we claim that these two conditions hold if and only if $y^{-1}(J_\nu \cap \Phi_\nu) \subseteq \Phi_{h_z}$ and $y^{-1}(\alpha_{\nu}) \in I_{h_z}$. The implication
in the forward direction is straightforward since we can intersect both conditions with $\Phi_\nu$ and $y \in \Symm_\nu$ preserves $\Phi_\nu$. 
Thus it suffices to show the reverse implication. 
By the definition of $\Phi_{h_z}$ and $I_{h_z}$, it in fact suffices to show that 
$y^{-1}(J_\nu) \subseteq z \Phi_h$. Note that $J_\nu \setminus (J_\nu \cap \Phi_\nu) = \{\alpha_{\nu_1+1}, \ldots, \alpha_{n-1}\}$. Since we already know $y^{-1}(J_\nu \cap \Phi_\nu) \subseteq z\Phi_h$, it is enough to show $y^{-1}(\alpha_s) \in z \Phi_h$ for $\nu_1+1 \leq s \leq n-1$. We take cases. For $s$ with $\nu_1+2 \leq s \leq n-1$, the fact that $y \in \Symm_{\nu}$ implies $y^{-1}(\alpha_s)=\alpha_s$, and now the assumption that $w=yz$ lies in $\mathcal{D}_{\nu}$ implies the desired result.  For $s  = \nu_1+1$, suppose for a contradiction that $y^{-1}(\alpha_{\nu_1+1}) \not \in z \Phi_h$. Then by definition of the ideal $I_h$ we have $y^{-1}(\alpha_{\nu_1+1}) \in z I_h$. On the other hand, by assumption $y^{-1}(\alpha_{\nu_1})$ also lies in $z I_h$. Since $\alpha_{\nu_1}$ and $\alpha_{\nu_1+1}$ are adjacent in the Dynkin diagram (equivalently, $\alpha_{\nu_1} + \alpha_{\nu_1+1} \in \Phi$) and since $I_h$ is an ideal, we conclude that $y^{-1}(\alpha_{\nu_1}) + y^{-1}(\alpha_{\nu_1+1}) \in z I_h$, i.e. $w^{-1}(\alpha_{\nu_1}) + w^{-1}(\alpha_{\nu_1+1}) \in I_h$ where both $w^{-1}(\alpha_{\nu_1})$ and $w^{-1}(\alpha_{\nu_1+1})$ are elements of $I_h$, but this contradicts the fact that $I_h$ is an abelian ideal. Hence we must have $y^{-1}(\alpha_{\nu_1+1}) \in z\Phi_h$, as desired. 

In addition, $J_{\nu}\cap \Phi_{\nu} = J_{(\nu_1,1)}$ when viewed as a subset of simple roots in $\mathfrak{gl}(\nu_1+1,\C)$.  The above considerations allow us to conclude 
\begin{equation}\label{eq: Dvz versus Dv1} 
\{ y\in W_{\nu} \hsm\vert\hsm yz\in \mathcal{D}_{\nu,z} \} = \{ y\in W_{\nu} \hsm\vert\hsm y^{-1}(J_{(\nu_1,1)}) \subseteq \Phi_{h_z} \textup{ and } y^{-1}(\alpha_{\nu}) \in I_{h_z} \}
\end{equation}
where the RHS is equal to the set $\mathcal{D}_{(\nu_1,1)}\subseteq W_{\nu}$ corresponding to the Hessenberg function $h_z: [\nu_1+1] \to [\nu_1+1]$.  
Finally, since the $z$ action on the set $\Phi$ is an automorphism it is straightforward from the definitions to see that $I_h$ abelian implies that $I_{h_z}$ is abelian. Moreover, from its definition it follows that $\sigma_{(\nu_1,1)}$ can be identified with $\sigma_{\nu}$ via the inclusion $W_\nu \cong \Symm_{\nu_1+1} \hookrightarrow \Symm_n$, since the first part of each of the partitions $\nu=(\nu_1, \nu_2)$ and $(\nu_1,1)$ is the same.

The arguments above imply that 
\begin{eqnarray*}
	\sum_{w\in \mathcal{D}_{\nu}} t^{ 2 |N^-(w)\cap \Phi_h^-|} &=& \sum_{z\in {^\nu W}} \; \sum_{yz\in \mathcal{D}_{\nu, z}} t^{2|N^-(z)\cap \Phi_h^-|+ 2| N^-(y)\cap \Phi_{h_z}^-| }\quad  \textup{ by~\eqref{eq:Nw and Nz and Ny} and $\mathcal{D}_{\nu} = \bigsqcup_{z \in {^{\nu}W}} \mathcal{D}_{\nu,z}$  } \\
	&=& \sum_{z\in {^\nu W}}\, t^{2|N^-(z)\cap \Phi_h^-|} \sum_{y\in \mathcal{D}_{(\nu_1,1)}} t^{2| N^-(y) \cap \Phi_{h_z}^-|}  \quad \textup{ by~\eqref{eq: Dvz versus Dv1} } \\
	&=& \sum_{z\in {^\nu W}}\, t^{2|N^-(z)\cap \Phi_h^-|} \sum_{\sigma_{(\nu_1,1)}y \,:\, y \in \mathcal{D}_{(\nu_1,1)}} t^{2 |N^-(\sigma_{(\nu_1,1)}y) \cap \Phi_{h_z}^-|} \quad \textup{ by Proposition~\ref{cor: subreg case} } \\
	&=& \sum_{z\in {^\nu W}}\; \sum_{\sigma_{\nu}yz \,:\, yz \in \mathcal{D}_{\nu,z}} t^{2|N^-(z)\cap \Phi_h^-|+2 |N^-(\sigma_{\nu}y) \cap \Phi_{h_z}^-|}\quad \textup{ since $\sigma_{(\nu_1,1)} = \sigma_{\nu}$}\\
	&=& \sum_{z\in {^\nu W}}\, \sum_{\sigma_{\nu}w \,:\, w \in \mathcal{D}_{\nu, z}} t^{2 |N^-(\sigma_{\nu}w) \cap \Phi_{h}^-|} \quad \textup{ by~\eqref{eq:Nw and Nz and Ny} applied to $\sigma_\nu w = (\sigma_\nu y)(z)$} \\
	&=&\sum_{\sigma_{\nu} w \,:\, w\in \mathcal{D}_{\nu}} t^{ 2 |N^-(\sigma_{\nu} w)\cap \Phi_h^-|}
\end{eqnarray*}
and proves~\eqref{eq: key equality}, as desired. 
\end{proof}

The proof of our main technical proposition is now a simple matter. 

\begin{proof}[Proof of Proposition~\ref{prop:induction step}] 
As already noted in~\eqref{eq: poincare polynomial decomp}, Lemma~\ref{lem: Betti numbers of regular Hess} implies 
\[
P(\Hess(\mathsf{X}_\nu, h), t) = P(\Hess(\mathsf{N}, h), t) + \sum_{w \in \mathcal{D}_\nu} t^{2 \lvert N^-(w) \cap \Phi_h^- \rvert}.
\]
Now Proposition~\ref{prop: graded case} says 
\[
\sum_{w \in \mathcal{D}_\nu} t^{2 \lvert N^-(w) \cap \Phi_h^- \rvert} = \sum_{T \in SK_2(\Gamma_h)} t^{2 \deg(T)} P(\Hess(\mathsf{X}_\mu, h_T),t)
\]
so the desired statement follows immediately. 
\end{proof}


\subsection{Proof of the graded Stanley-Stembridge conjecture for the abelian case}\label{subsec: graded SS for abelian}

We can now prove the graded Stanley-Stembridge conjecture for the abelian case by induction. We have the following. 

\begin{corollary}\label{corollary: graded SS for abelian}
Let $n$ be a positive integer and $h:[n] \to [n]$ a Hessenberg function such that $I_h$ is abelian. Then the integers $c_{\lambda,i}$ appearing in~\eqref{eq: decomp into Mlambda} are non-negative.
\end{corollary}

\begin{proof}
We argue by induction. Our base cases are $n=1$ and $n=2$. The case $n=1$ is trivial in the sense that the regular semisimple Hesenberg variety under consideration is just a single point, and the symmetric group is the trivial group. Hence the claim holds in this case. 

The next case $n=2$ is the first case in which the corresponding flag variety $\Flags(\C^n) = \Flags(\C^2) \cong \P^1$ is non-trivial. In this case there are only two Hessenberg functions to consider: $h=(1,2)$ and $h=(2,2)$. Both cases correspond to abelian ideals. If $h=(1,2)$, the corresponding variety $\Hess(\mathsf{S},(1,2))$ consists of two points $\{N,S\}$ (the ``north pole'' and ``south pole'' of the $\P^1$) and hence its cohomology is non-zero only in degree $0$.  In this case, the corresponding Hessenberg space $H$ is the Borel subalgebra and Teff's results \cite{Teff2013a} prove that $H^0(\Hess(\mathsf{S},(1,2))) \cong M^{(1,1)}$.  (The reader may confirm this by computing the corresponding representation directly using, for example, the explicit description of Tymoczko's action in \cite{Tym08} via GKM theory.)  If $h=(2,2)$, the variety $\Hess(\mathsf{S},(2,2))$ is equal to the entire flag variety $\P^1$ and has non-zero cohomology only in degrees $0$ and $2$. Another direct computation shows that $H^0(\Hess(\mathsf{S},(2,2))) \cong M^{(2)}$ and $H^2(\Hess(\mathsf{S},(2,2))) \cong M^{(2)}$. In both cases we conclude that the $\Symm_2$-representation, in each degree, is a non-negative sum of tabloid representations. Thus the claim of the theorem holds in these base cases.

Now suppose $n \geq 3$ and let $\mathsf{S}'$ be any regular semisimple element of $\mathfrak{gl}(k,\C)$ for $k$ for $1 \leq k \leq n-1$.  Suppose also by induction that for any Hessenberg function $h': [k] \to [k]$ with $I_{h'}$ abelian, we have that 
\[
H^{2i}(\Hess(\mathsf{S}', h')) = \sum_{\lambda\vdash k} c_{\lambda,i}' M^{\lambda}
\]
where the $c_{\lambda,i}'$ are all non-negative. We wish to show that the corresponding statement is true for $k=n$.  To see this, fix $i\geq 0$. Suppose $h: [n] \to [n]$ is a Hessenberg function such that $I_h$ is abelian. We wish to show that 
\[
H^{2i}(\Hess(\mathsf{S},h)) = \sum_{\lambda \vdash n} c_{\lambda,i} M^{\lambda}
\]
where each $c_{\lambda, i} \in \Z$ and $c_{\lambda,i} \geq 0$. 

On the other hand, by Theorem~\ref{theorem:induction} we know   
\[
H^{2i}(\Hess(\mathsf{S}, h)) = c_{(n),i} M^{(n)} + \left( \sum_{T \in SK_2(\Gamma_h)} \left( \sum_{\substack{\mu \vdash (n-2)\\ \mu=(\mu_1,\mu_2)}} c_{\mu, i-\deg(T)}^T M^{(\mu_1+1,\mu_2+1)} \right) \right). 
\]
On the RHS of the above equality, the coefficient $c_{(n),i}$ is non-negative by Corollary~\ref{corollary: triv coeff is nonneg}.  Moreover, in the summation expression on the RHS, each $\mu$ is a partition of $n-2$ and $c_{\mu, i-\deg(T)}^T$ is the coefficient of $M^{\mu}$ in the decomposition of $H^{2i-2\deg(T)}(\Hess(\mathsf{S}',h_T))$.  Since $h$ is an abelian Hessenberg function, $h_T$ is also abelian by Lemma~\ref{lem: maximal sink set induction}, and therefore each coefficient $c_{\mu, i-\deg(T)}^T$ is non-negative by the induction hypothesis. Thus, the above equality expresses $H^{2i}(\Hess(\mathsf{S},h))$ as a non-negative linear combination of tabloid  representations. This completes the inductive step and hence the proof of the theorem.
\end{proof}


\section{A conjecture for the general case}\label{section: conjecture} 

Although we prove our main result, Theorem~\ref{theorem:main}, for abelian Hessenberg varieties only, much of the framework and analysis in Sections~\ref{sec: sink sets and induction} and~\ref{sec: sink sets, ideals and representations} is general.  In particular the analysis of maximal sink sets of the graph $\Gamma_h$ in Section~\ref{sec: sink sets and induction} shows that every acyclic orientation of $\Gamma_h$ corresponding to such a sink set $T$ is obtained inductively from an acyclic orientation of the smaller graph $\Gamma_{h_T}$ on $n-|T|$ vertices.  Using Theorem~\ref{thm: Stanley}, this indicates a correspondence between the representations $H^*(\Hess(\mathsf{S},h))$ and $H^*(\Hess(\mathsf{S}', h_T))$, where $\mathsf{S}'$ denotes a regular semisimple element in $\mathfrak{gl}(n-|T|, h_T)$.  Suppose $\lambda \vdash  n$ has $m(\Gamma_h)$ parts.  We now formulate a conjectured formula for the coefficient of $M^{\lambda}$ occurring in $H^*(\Hess(\mathsf{S}, h))$ as a function of the coefficients of $M^{\mu}$ occurring in $H^*(\Hess(\mathsf{S}', h_T))$ where $\mu \vdash (n-|T|)$.  

\begin{conjecture}\label{conj: general case} 
Let $h: [n] \to [n]$ be a Hessenberg function and $\lambda\vdash n$ be a partition with exactly $m = m(\Gamma_h)$ parts.  Let $\mu = (\mu_1, \mu_2,..., \mu_m)$ be a partition of $n-|T|$ such that $\lambda = (\mu_1+1, \mu_2+1 , \cdots, \mu_m +1)$.  Then for all $i\geq 0$,
\[
	c_{\lambda, i} = \sum_{T\in \SK_m(\Gamma_h)} c_{\mu, i-\deg(T)}^T.
\]
\end{conjecture} 

This conjecture extends the results of Theorem~\ref{theorem:induction} to arbitrary regular semisimple Hessenberg varieties.  However, unless the Hessenberg variety is abelian (i.e., unless $m(\Gamma_h)\leq 2$), this formula does not determine the entire representation $H^*(\Hess(\mathsf{S},h))$.  The next example demonstrates this.

\begin{example} 
Suppose $n=7$ and $h=(3,4,5,6,7,7,7)$. In this case, $m(\Gamma_h) = 3$ so $h$ is not an abelian. We will show that Conjecture~\ref{conj: general case} correctly predicts the coefficients of $M^{\lambda}$ for $\lambda \vdash 7$ with exactly three parts.  Note that, in this case, there is only one maximal sink set, namely $T=\{1, 4, 7\}$.  The graph below shows the acyclic orientation $\omega \in \mathcal{A}_2(\Gamma_h)$ such that $\asc(\omega) = \deg(\{1,4,7\})$.  The vertices in $\{1,4,7\}$ and incident edges are highlighted in red, and we display the corresponding acyclic orientation of $\Gamma_h - T \cong \Gamma_{h_T}$ on the right.
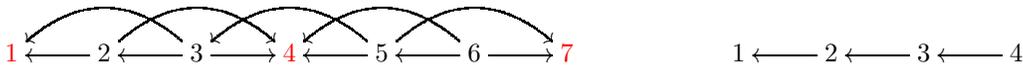
\vspace*{.15in}
\[\xymatrix{ {\color{red}1}    & {2} \ar@[red][l] \ar@[red]@/^1.5pc/[rr]  & 3 \ar@[red][r] \ar[l] \ar@[red]@/_1.5pc/[ll]  &  {\color{red}4}    & {5} \ar@[red]@/^1.5pc/[rr]  \ar@/_1.5pc/[ll] \ar@[red][l]  & {6} \ar@[red][r] \ar[l] \ar@[red]@/_1.5pc/[ll] & {\color{red}7} & &
1 & 2 \ar[l] & 3 \ar[l] & 4 \ar[l]
}\]
The graphs above show that $\deg(\{1,4,7\}) = 4$ and $h_T=(2,3,4,4)$.  The left-hand table below computes the representation $H^*(\Hess(\mathsf{S}_T,h_T))$ as a sum of tabloid representations in each degree.  The right-hand table below shows the tabloid representations corresponding to partitions with $3$ parts that occur as summands of $H^*(\Hess(\mathsf{S},h))$ in $\mathcal{R}ep(\Symm_n)$.
\[
\begin{array}{|c|c|}\hline
H^0(\Hess(\mathsf{S}',h_T)) & M^{(4)}\\ 
H^2(\Hess(\mathsf{S}',h_T))  & M^{(4)} + M^{(3,1)} + M^{(2,2)} \\
H^4(\Hess(\mathsf{S}',h_T)) &M^{(4)} + M^{(3,1)} + M^{(2,2)} \\
H^6(\Hess(\mathsf{S}',h_T))  & M^{(4)}\\
\hline
\end{array}
\quad
\begin{array}{|c|c|}\hline
H^8(\Hess(\mathsf{S},h)) & M^{(5,1,1)}\\
H^{10}(\Hess(\mathsf{S},h)) &M^{(5,1,1)} + M^{(4,2,1)} + M^{(3,3,1)}  \\
H^{12}(\Hess(\mathsf{S},h)) & M^{(5,1,1)} + M^{(4,2,1)} + M^{(3,3,1)}\\
H^{14}(\Hess(\mathsf{S},h)) & M^{(5,1,1)} \\\hline
\end{array}
\]
These tables confirm Conjecture~\ref{conj: general case}.  Although the conjectured formula correctly determines the tabloid representations $M^{\lambda}$ appearing for $\lambda$ with three parts, it does \emph{not} determine the entire representation.  For example, 
\[
H^{10}(\Hess(\mathsf{S},h)) = 32M^{(7)} + 27M^{(6,1)}+19 M^{(5,2)} + 15M^{(4,3)}+ M^{(5,1,1)} + M^{(4,2,1)} + M^{(3,3,1)}
\]
and we do not know of an inductive formula for the coefficients of $M^{(6,1)}$, $M^{(5,2)}$, or $M^{(4,3)}$ at this time.
\end{example} 

The formula given in Conjecture~\ref{conj: general case} does not determine the coefficients for $M^{\lambda}$ unless $\lambda$ has a maximal number of parts.  In this sense, this conjecture represents \emph{the tip of an iceberg}. To obtain a formula which fully generalizes Theorem~\ref{theorem:induction} we need some idea of how to inductively obtain the coefficients $c_{\lambda}$ for $\lambda$ with any number of parts. We intend to pursue this in future work.

%

\begin{thebibliography}{10}

\bibitem{AHHM}
H.~Abe and M.~Harada and T.~Horiguchi and M.~Masuda. 
\newblock The cohomology rings of regular nilpotent {H}essenberg varieties in {L}ie type {A}. 
\newblock {\em Internat. Math. Res. Notices,} https://doi.org/10.1093/imrn/rnx275, 2017. 

\bibitem{AHM2017}
H.~Abe and T.~Horiguchi and M.~Masuda. 
\newblock The cohomology rings of regular semisimple {H}essenberg varieties for {$h=(h(1),n,...,n)$}, 
2017, https://arxiv.org/abs/1704.00934. 

\bibitem{BrosnanChow2015}
P.~Brosnan and T.~Y. Chow.
\newblock Unit interval orders and the dot action on the cohomology of regular
  semisimple {H}essenberg varieties, 2015, https://arxiv.org/abs/1511.00773.

\bibitem{DeMProSha92}
F.~De~Mari, C.~Procesi, and M.~A. Shayman.
\newblock Hessenberg varieties.
\newblock {\em Trans. Amer. Math. Soc.}, 332(2):529--534, 1992.

\bibitem{DeMari1987}
F.~De~Mari Casareto~dal Verme.
\newblock {\em On the topology of the {H}essenberg varieties of a matrix}.
\newblock ProQuest LLC, Ann Arbor, MI, 1987.
\newblock Thesis (Ph.D.)--Washington University in St. Louis.

\bibitem{Ful97}
W.~Fulton.
\newblock {\em Young tableaux}, volume~35 of {\em London Mathematical Society
  Student Texts}.
\newblock Cambridge University Press, Cambridge, 1997.
\newblock With applications to representation theory and geometry.

\bibitem{Gasharov1996}
V.~Gasharov.
\newblock Incomparability graphs of {$(3+1)$}-free posets are {$s$}-positive.
\newblock In {\em Proceedings of the 6th {C}onference on {F}ormal {P}ower
  {S}eries and {A}lgebraic {C}ombinatorics ({N}ew {B}runswick, {NJ}, 1994)},
  volume 157, pages 193--197, 1996.
  
\bibitem{GebhardSagan2001}
D.~Gebhard and B.~Sagan.
\newblock A chromatic symmetric function in noncommuting variables.
\newblock {\em J. Algebraic Combin.}, 13(3):227--255, 2001.

\bibitem{Guay-Paquet2013}
M.~Guay-Paquet.
\newblock A modular relation for the chromatic symmetric functions of
  (3+1)-free posets, 2013, https://arxiv.org/abs/1306.2400.

\bibitem{Guay-Paquet2016}
M.~Guay-Paquet.
\newblock A second proof of the {S}hareshian-{W}achs conjecture, by way of a
  new {H}opf algebra, 2016, https://arxiv.org/abs/1601.05498.
  
\bibitem{Haglund2008}
J.~Haglund. 
\newblock The {$q,t$}-{C}atalan numbers and the space of diagonal harmonics: with an appendix on the combinatorics of {M}acdonald polynomials. 
\newblock American Mathematical Society, Providence, RI, 2008. 
\item University Lecture Series, Vol. 41. 

\bibitem{Humphreys1972}
J.~E. Humphreys.
\newblock {\em Introduction to {L}ie algebras and representation theory}.
\newblock Springer-Verlag, New York-Berlin, 1972.
\newblock Graduate Texts in Mathematics, Vol. 9.

\bibitem{Kos61}
B.~Kostant.
\newblock Lie algebra cohomology and the generalized {B}orel-{W}eil theorem.
\newblock {\em Ann. of Math. (2)}, 74:329--387, 1961.

\bibitem{Kostant1998}
B.~Kostant.
\newblock The set of abelian ideals of a {B}orel subalgebra, {C}artan
  decompositions, and discrete series representations.
\newblock {\em Internat. Math. Res. Notices}, (5):225--252, 1998.

\bibitem{MbirikaTymoczko} 
A.~Mbirika and J.~Tymoczko. 
\newblock Generalizing {T}anisaki's ideal via ideals of truncated symmetric functions.
\newblock {\em J. of Alg. Combinatorics} 37:167--199, 2012. 

\bibitem{Precup2013}
M.~Precup.
\newblock Affine pavings of {H}essenberg varieties for semisimple groups.
\newblock {\em Selecta Math. (N.S.)}, 19(4):903--922, 2013.

\bibitem{Precup2015}
M.~Precup.
\newblock The connectedness of {H}essenberg varieties.
\newblock {\em J. Algebra}, 437:34--43, 2015.

\bibitem{Precup2016}
M.~Precup.
\newblock The {B}etti numbers of regular {H}essenberg varieties are
  palindromic, 2016, https://arxiv.org/abs/1603.07662.
\newblock Accepted for publication in \emph{Trans. Groups.}

\bibitem{Rahman2017}
M.~S. Rahman.
\newblock {\em Basic graph theory}.
\newblock Undergraduate Topics in Computer Science. Springer, Cham, 2017.

\bibitem{ShareshianWachs2012}
J.~Shareshian and M.~L. Wachs.
\newblock Chromatic quasisymmetric functions and {H}essenberg varieties.
\newblock In {\em Configuration spaces}, volume~14 of {\em CRM Series}, pages
  433--460. Ed. Norm., Pisa, 2012.

\bibitem{ShareshianWachs2016}
J.~Shareshian and M.~L. Wachs.
\newblock Chromatic quasisymmetric functions.
\newblock {\em Adv. Math.}, 295:497--551, 2016.

\bibitem{SomTym06}
E.~Sommers and J.~Tymoczko.
\newblock Exponents for {$B$}-stable ideals.
\newblock {\em Trans. Amer. Math. Soc.}, 358(8):3493--3509 (electronic), 2006.

\bibitem{Stanley1995}
R.~P. Stanley.
\newblock A symmetric function generalization of the chromatic polynomial of a
  graph.
\newblock {\em Adv. Math.}, 111(1):166--194, 1995.

\bibitem{Stanley-EnumCombVol2}
R.~P. Stanley.
\newblock {\em Enumerative combinatorics. {V}ol. 2}, volume~62 of {\em
  Cambridge Studies in Advanced Mathematics}.
\newblock Cambridge University Press, Cambridge, 1999.
\newblock With a foreword by Gian-Carlo Rota and appendix 1 by Sergey Fomin.

\bibitem{Stanley-personal}
R.~P. Stanley. 
\newblock Personal communication. See also \url{http://www-math.mit.edu/~rstan/pubs/} 

\bibitem{StanleyStembridge1993}
R.~P. Stanley and J.~R. Stembridge.
\newblock On immanants of {J}acobi-{T}rudi matrices and permutations with
  restricted position.
\newblock {\em J. Combin. Theory Ser. A}, 62(2):261--279, 1993.

\bibitem{Teff2013a}
N.~J. Teff.
\newblock {\em The {H}essenberg representation}.
\newblock ProQuest LLC, Ann Arbor, MI, 2013.
\newblock Thesis (Ph.D.)--The University of Iowa.

\bibitem{Tym06}
J.~S. Tymoczko.
\newblock Linear conditions imposed on flag varieties.
\newblock {\em Amer. J. Math.}, 128(6):1587--1604, 2006.

\bibitem{Tym08}
J.~S. Tymoczko.
\newblock Permutation actions on equivariant cohomology of flag varieties.
\newblock In {\em Toric topology}, volume 460 of {\em Contemp. Math.}, pages
  365--384. Amer. Math. Soc., Providence, RI, 2008.

\end{thebibliography}

\def\cprime{$'$}

\end{document}